\documentclass[11pt]{amsart}
\usepackage{a4wide,pdfsync,hyperref}
\usepackage{amsmath,amssymb,esint}
\usepackage{enumerate}
\usepackage{color}

\numberwithin{equation}{section}

\allowdisplaybreaks

\newtheorem{theorem}{Theorem}[section]
\newtheorem{lemma}[theorem]{Lemma}

\newtheorem{corollary}[theorem]{Corollary}
\newtheorem{proposition}[theorem]{Proposition}
\newtheorem{remark}[theorem]{Remark}

\newcommand{\eps}{\varepsilon}

\newcommand{\beqq}{\begin{eqnarray}}
\newcommand{\enqq}{\end{eqnarray}}

\newcommand{\enn}{\end{equation}}
\newcommand{\bef}{\begin{proof}}
\newcommand{\enf}{\end{proof}}

\let\e=\varepsilon

\let\f=\frac

\let\om=\omega

\let\Om=\Omega

\let\pa=\partial

\def\dv{\mbox{div}}

\def\curl{\mathop{\rm curl}\nolimits}

\def\ef{\hphantom{MM}\hfill\llap{$\square$}\goodbreak}

\newcommand{\beq}{\begin{equation}}
\newcommand{\eeq}{\end{equation}}
\newcommand{\ben}{\begin{eqnarray}}
\newcommand{\een}{\end{eqnarray}}
\newcommand{\beno}{\begin{eqnarray*}}
\newcommand{\eeno}{\end{eqnarray*}}

\begin{document}
\title[Inviscid Limit of NS with point vortex data]{The Navier-Stokes equations in $\mathbb R^2_+$ with point vortex initial data: Zero-viscosity limit}

\author[C. Wang]{Chao Wang}
\address{School of Mathematical Sciences\\ Peking University\\ Beijing 100871, China}
\email{wangchao@math.pku.edu.cn}

\author[J. Yue]{Jingchao Yue}
\address{School of Mathematical Sciences\\ Peking University\\ Beijing 100871, China}
\email{wasakarumi@163.com}

\author[Z. Zhang]{Zhifei Zhang}
\address{School of Mathematical Sciences\\ Peking University\\ Beijing 100871, China}
\email{zfzhang@math.pku.edu.cn}

\begin{abstract}
This is the second of two papers devoted to the asymptotic behavior of solutions to the incompressible Navier-Stokes equations in a half-space with point vortex initial data. A major difficulty stems from the interaction between the point vortex initial data and the boundary, which complicates the derivation of a valid asymptotic expansion.
To overcome this, we carry out a precise matching between the point vortex and boundary-layer profiles to accurately capture the correct viscous behavior of the vortex in the half-plane. Based on this matched asymptotic analysis, we decompose the vorticity into three components: vorticity near the point vortex, vorticity near the boundary, and vorticity in the transition layer. A key point is that each component must be analyzed in its own distinct region. On this basis, we establish refined estimates and thereby achieve the inviscid limit for the point vortex. Finally, we rigorously prove that solutions to the Navier-Stokes equations converge to the Lamb-Oseen vortex away from the boundary, while approaching the Prandtl boundary-layer system in the near-boundary region.
\end{abstract}

\date{\today}

\maketitle

\section{Introduction}

\subsection{Presentation of the problem and related results}
This is the second in a series of two papers concerning the zero-viscosity limit of the incompressible Navier-Stokes equations in a half-space with point vortex initial data. In the previous paper \cite{WYZ}, we established the existence and uniqueness of solutions to the Navier-Stokes equations for a fixed viscosity. In the present work, we analyze their asymptotic behavior in the high-Reynolds-number limit.  More precisely, we study the inviscid limit of the incompressible Navier-Stokes equations in a half-plane $\mathbb R^2_+$ with irregular initial data:
\begin{align}\label{eq: NS velocity}
	\left\{
	\begin{aligned}
		&\pa_t U-\nu\Delta U+U\cdot\nabla U+\nabla p=0,\quad (x,y)\in\mathbb R^2_+,\\
		&\dv U=0, \\
		&U|_{y=0}=0,
	\end{aligned}
	\right.
\end{align}
where $U=(u,v)$ and $p$ stand for the fluid velocity and pressure respectively, $\nu>0$ is the viscosity constant. 

It is well established that coherent structures play a crucial role in two-dimensional turbulent flows. Experiments and numerical simulations of decaying turbulence show that, at high Reynolds numbers, isolated regions of concentrated vorticity emerge after a short transient phase and persist over very long time scales (see \cite{Couder, Mcwilliams, Mcwilliams 2}). These nearly axisymmetric structures behave like point vortices when well separated, but undergo substantial deformation or even merging when two such structures approach one another closely. The long-time evolution of such flows is therefore governed by fundamental mechanisms including vortex interaction and merging (see \cite{Le Dizes, Meunier}). A natural mathematical framework for studying vortex interactions is to use point vortices as initial data.

In this paper, we consider the initial data $U_0$, whose vorticity is concentrated as a Dirac measure, i.e.,
\begin{align}\label{initial data}
	\omega_0=\curl U_0=\alpha\delta_{X_0},\quad with \quad \alpha\in\mathbb R,
\end{align}
where $X_0$ is a point away from the boundary. Without loss of generality, we assume
\begin{align*}
X_0=(x_0,y_0)=(0,20).
\end{align*} 
To investigate the Navier-Stokes equations with this type of singular initial data, one approach is to begin with the vorticity formulation, which states that
\begin{align}\label{eq: NS vorticity}
	\pa_t\omega-\nu\Delta\omega+U\cdot\nabla\omega=0,\qquad \omega|_{t=0}=\omega_0,
\end{align}
where $\omega:=\pa_x v-\pa_y u$. The velocity $U$ can be recovered by the Biot-Savart law:
\beno
U=\nabla^\perp\Delta_D^{-1}\omega,\quad \textrm{where} 
\quad \nabla^\perp=(-\pa_y,\pa_x).
\eeno
Using the non-slip boundary condition of $U$ and the Biot-Savart law, we derive the boundary condition of $\omega$ (see \cite{Maekawa}): 
\begin{align}\label{eq: BC of NS}
	\nu(\pa_y+|D_x|)\omega|_{y=0}=\pa_y\Delta_D^{-1}(U\cdot\nabla\omega)|_{y=0}.
\end{align}

The study of the inviscid limit in the whole space has an extensive history. For smooth initial data $(u_0, v_0) \in H^{s}$ ($s>2$), the inviscid limit with an $L^2$ convergence rate of order $\nu$ was established in \cite{BM, CF}, while the convergence in $H^s$ was addressed in \cite{MR}. Meanwhile, the case of irregular initial data has also attracted considerable attention. For vortex patch initial data (i.e., $\omega_0$ is a characteristic function), Constantin and Wu \cite{CW, CW1} demonstrated that 
\beno
\|U-U^e\|_{L^2}\leq C\nu^{1/2},
\eeno
where $U^e=(u^e,v^e)$  denotes a solution to the Euler equations 
\begin{align}\label{eq: Euler}
	\left\{
	\begin{aligned}
		\pa_t U^e+U^e\cdot\nabla U^e+\nabla p^e&=0,\\
		\operatorname{div}U^e&=0,\\
		U^e|_{t=0}&=U_0.
	\end{aligned}
	\right.
\end{align}
Subsequently, Abidi and Danchin \cite{AD} established the optimal $L^2$ convergence rate of $\nu^{3/2}$. Sueur \cite{Sueur} further derived an asymptotic expansion of the solution in the vanishing viscosity limit for fluids with sharp vorticity variations. This expansion was rigorously justified by Liao, Sueur and Zhang \cite{LSZ}. More recently, Constantin, Drivas and Elgindi \cite{CDE} generalized these results to wide classes of Yudovich-type initial data, showing that
\begin{align*}
	\lim_{\nu\rightarrow 0} \|\curl U-\curl U^e\|_{L^p} =0, \quad p\in [1,\infty).
\end{align*}
For vortex-sheet initial data, the discontinuity in the tangential velocity component leads to the formation of a strong Prandtl boundary layer in the inviscid limit. Caflisch and Sammartino \cite{CS} derived a specific form of the Prandtl system to describe this limiting behavior.
For more singular data such as point vortices, Gallay \cite{Gallay 2} first proved that the vorticity of the Navier-Stokes equations converges in $L^1$ norm,in the vanishing viscosity limit, to a superposition of Oseen vortices. More precisely, for initial vorticity $\omega_0=\sum_{i=1}^N \alpha_i\delta_{X_i}$, he established
\begin{align*}
	\frac{1}{|\alpha|}\int_{\mathbb R^2}
	\left| \omega(t,X)-\sum_{i=1}^N \frac{\alpha_i}{4\pi \nu t}e^{-\frac{|X-X_i(t)|^2}{4\pi\nu t}} \right|dX
	\leq C\nu t,\qquad for \ t\in(0,T),
\end{align*}
where $|\alpha|=\sum_{i=1}^N|\alpha_i|$
and $X_i(t)$  solves the Helmholtz-Kirchhoff system. Nguyen and Nguyen \cite{TT Nguyen} studied the interaction between a point vortex and a smooth vortex patch. Extensive work has also been devoted to the long-time dynamics of vortices in the full plane $\mathbb{R}^2$. In particular, \cite{Dolce, ZP 1} analyzed the evolution of a viscous vortex dipole formed by two point vortices.

The situation changes significantly in the presence of a boundary. The mismatch in boundary conditions gives rise to a boundary layer. For slip boundary conditions, this boundary layer is weak. Masmoudi and Rousset \cite{MR} justified the inviscid limit in Sobolev spaces using the conormal derivative method. For further related results, we refer to \cite{IF, IS, WXY, WXZ, XX1}.
Under no-slip boundary conditions, however, a strong boundary layer develops. In 1904, Prandtl introduced boundary-layer theory in \cite{Prandtl}, leading to the formal expansion
 \ben\label{formal expan}
 \left\{
 \begin{array}{l}
 u^{\e}(t,x,y) =u^{e}(t,x,y)+ u^{p}(t,x,\f{y}{\nu^{1/2}})+O(\nu^{1/2}),\\
 v^{\e}(t,x,y)= v^{e}(t,x,y)+\nu^{1/2} v^{p}(t,x,\f{y}{\nu^{1/2}})+O(\nu^{1/2}),
 \end{array}\right.
 \een
 where $(u^{p}, v^{p})$ satisfies the Prandtl equations. Owing to the strong boundary layer, the inviscid limit has been established only in certain special settings, such as analytic spaces \cite{SC1, SC2, WWZ, TT Nguyen, Kukavica} and Gevrey spaces \cite{CWZ-1, GMM, GMM1}. The inviscid limit in Sobolev spaces was justified in \cite{FTZ, Maekawa} under the structural assumption that vorticity is supported away from the boundary. A key ingredient in these works is the use of the conditions $\text{div } U^e = 0$ and $\text{curl } U^e = 0$ near the boundary, which ensure analytic regularity of the flow in that region. 
 
For irregular initial data, the methods developed in the aforementioned papers cannot be directly applied to handle the interaction between the singular data and the boundary layer. In \cite{WYZ1}, we studied the interaction among the internal transition layer, the boundary layer, and the initial layer to treat vortex-patch-type initial data. The key idea in \cite{WYZ1} was to construct an approximate solution that captures the interaction between the vortex patch, the boundary layer, and the initial layer.
For more general rough vortex patches, where the vorticity is only assumed to lie in $L^\infty$-a class of Yudovich-type initial data-the analysis becomes significantly more difficult. The low regularity prevents a direct construction of an asymptotic expansion. To overcome this obstacle, in \cite{HWYZ} we established a Kato-type criterion adapted to the Yudovich framework and constructed an appropriate energy space to prove convergence to the Euler equations.  
 
 This paper investigates the case of more singular initial data in the half-plane $\mathbb R^2_+$, namely the point vortex.  
In \cite{WYZ}, we studied the interaction between the boundary and the point vortex, leading to the following existence and uniqueness result for the Navier-Stokes equations with a point vortex in the half-plane. Notably, Dalibard and Gallay recently established the same result in \cite{Dalibard} using a different method.
 
\begin{theorem}\label{thm: existence and uniqueness}$($\cite{Dalibard,WYZ}$)$
Assume the initial data satisfies \eqref{initial data}. Then the Navier-Stokes equations \eqref{eq: NS vorticity} admit a unique global solution $(\omega,U)$ satisfying 
	\begin{align}\label{initial limit}
		\omega(t)\rightharpoonup \alpha\delta_{(x_0,y_0)}-u_0\delta_{\partial \mathbb R_+^2},\quad \text{vaguely in}\  M(\overline{\mathbb R^2_+})\ \text{as}\ t\rightarrow0^+,
	\end{align}
	where $u_0$ is the first component of the initial velocity $U_0=(u_0,v_0)$.
\end{theorem}

Building on this foundation, we proceed to study the asymptotic behavior of the aforementioned solutions. Our main result in this paper is summarized as follows.

\begin{theorem}\label{thm: main result vorticity convergence}
Let the initial data be given by \eqref{initial data}. Then, there exist $\nu_0,C>0$ and $T$ independent of $\nu$, a boundary layer corrector $\omega_b\in  C^\infty(\mathbb R^2_+)$, and a function $X(t)$ such that for any $0<\nu<\nu_0$, the solution $\omega$ to the Navier-Stokes equations \eqref{eq: NS vorticity} with initial data \eqref{initial data} satisfies for $1<p\leq+\infty$,
	\begin{align}\label{est: inviscid limit away from boundary}
		\left\|\omega(t,X)-\frac{\alpha}{4\pi \nu t}e^{-\frac{|X-X(t)|^2}{4\pi\nu t}} \right\|_{L^1_X(y\geq1)}\leq C\nu t,\qquad t\in(0,T),
	\end{align}
	and
	\begin{align}\label{est: inviscid limit near boundary}
		\left\|\|\omega-\omega_b\|_{L^p_x}\right\|_{L_y^1(y\leq2)}\leq C\nu^{1/2},\qquad t\in(0,T),
	\end{align}
	where $X=(x,y)$ and $X(t)=Z(t)+O(\nu^{1/2}t)$ is a smooth function used to describe the vortex center position satisfying $X(0)=X_0$ and $Z(t)=(x_0+\frac{\alpha t}{4\pi y_0},y_0)$.
\end{theorem}

\begin{remark}
	For general $N$ vortices data $\omega_0=\sum_{i=1}^N \alpha_i\delta_{X_i}$ supported away from boundary, we can establish an analogous convergence result by applying the same methodology.
\end{remark}


\subsection{The key ingredients}

In this subsection, we outline the key ingredients of the proof of Theorem \ref{thm: main result vorticity convergence}.

\subsubsection{Formulation under self-similar variables}
One of the main challenges arises from the irregular initial data, which introduces singularities into the analysis. To address this issue, it is necessary to obtain more detailed information near the point vortex. To this end, we employ self-similar variables to reformulate the system, a method introduced in \cite{Gallay 2}. The self-similar variables are defined as follows
  \begin{align}\label{def: self-similar}
\eta:=\frac{(x,y)-X(t)}{(\nu t)^{1/2}},
  \end{align}
    where $X(t)$ denotes the position of the point vortex at time $t$ and will be determined later by asymptotic expansion.
  
First, we introduce a cut-off function $\chi_{vp}$ defined by
  \begin{align}\label{def: chi vp}
  	\chi_{vp}(x,y)=
  	\left\{
  	\begin{aligned}
  		&1,\quad |(x,y)-(0,20)|\leq5,\\
  		&0,\quad |(x,y)-(0,20)|\geq6.
  	\end{aligned}
  	\right.
  \end{align}
Multiplying $\chi_{vp}$ on both sides of \eqref{eq: NS vorticity} yields
  \begin{align}\label{eq: NS near vortex point}
  	&\pa_t(\chi_{vp}\omega)
  	+BS_{\mathbb R^2_+}[\chi_{vp}\omega]\cdot\nabla(\chi_{vp}\omega)
  	-\nu\Delta(\chi_{vp}\omega)
  	+BS_{\mathbb R^2_+}[(1-\chi_{vp})\omega]\cdot\nabla(\chi_{vp}\omega)\\
  	\nonumber
  	&\qquad=
  	U\cdot\nabla\chi_{vp}\omega
  	-2\nu\nabla\chi_{vp}\cdot\nabla\omega
  	-\nu\Delta\chi_{vp}\omega.
  \end{align}
We define
  \begin{align}\label{def: self-similar}
  \frac{\alpha}{\nu t}\mathcal W \big(\eta,t\big):=	\chi_{vp}(x,y)\omega(t,x,y),
  \end{align}
  and
    \begin{align}\label{def: widetilde eta}
\widetilde\eta=(\widetilde\eta_1,\widetilde\eta_2):=\frac{X(t)-X(t)^\ast}{(\nu t)^{1/2}}.
  \end{align}
By the Biot-Savart law (see Lemma \ref{lem: derivation of velocity formula}), we have
  \begin{align}\label{velocity self-similar}
  	BS_{\mathbb R^2_+}[\chi_{vp}\omega](t,x,y)
  	=\frac{\alpha}{(\nu t)^{1/2}}\Big(\mathcal V^{\mathcal W}\big(\eta,t)-
  	\widetilde{\mathcal V^{\mathcal W}}(\eta+\widetilde\eta,t)\Big)
  \end{align}
  where $\mathcal V^{f}:=BS_{\mathbb R^2}[f]$ and $\widetilde{\mathcal V^{\mathcal W}}$ is the reflection operator defined later (see \eqref{def: tilde-F}).

  Based on the above notations, we obtain the equation of $\mathcal W$:
  \begin{align}\label{eq: profile eq}
  	&t\pa_t \mathcal W
  	+\Big\{\frac{\alpha}{\nu}\mathcal V^{\mathcal W}(\eta,t)
  	-\frac{\alpha}{\nu}\widetilde{\mathcal V^{\mathcal W}}(\eta+\widetilde\eta,t)
  	-\sqrt{\frac{t}{\nu}}X'(t)
  	\Big\}\cdot\nabla_\eta \mathcal W(\eta,t)-\mathcal L\mathcal W\\
  	\nonumber
  	&\qquad+\sqrt{\frac{t}{\nu}} BS_{\mathbb R^2_+}[\chi_b\omega](\eta,t)\cdot\nabla_\eta \mathcal W(\eta,t)\\
  	\nonumber
  	&= \frac{\nu t^2}{\alpha}U\cdot\nabla\chi_{vp} \omega
  	-\frac{2(\nu t)^2}{\alpha}\nabla\chi_{vp}\cdot\nabla\omega
  	-\frac{(\nu t)^2}{\alpha}\Delta\chi_{vp} \omega\\
  	\nonumber
  	&\qquad-\sqrt{\frac{t}{\nu}} BS_{\mathbb R^2_+}[(1-\chi_{vp}-\chi_b)\omega](\eta,t)\cdot\nabla_\eta \mathcal W(\eta,t),
  \end{align}
  where $\mathcal L:=\Delta_\eta+\frac{1}{2}\eta\cdot\nabla_\eta+1$.

\subsubsection{Construction of approximate solutions}

To study the asymptotic behavior of solutions to the Navier-Stokes equations, a standard approach is to construct approximate solutions via asymptotic analysis. Since the Euler equations with point vortex initial data are ill-posed, we cannot perform an asymptotic expansion directly from the velocity formulation \eqref{eq: NS velocity}, as is typically done. We therefore focus on the vorticity equation \eqref{eq: NS vorticity} when constructing approximate solutions.

To obtain a refined description of vorticity near the point vortex, we first introduce the self-similar coordinates \eqref{def: self-similar} in the vortex neighborhood and rewrite the Navier-Stokes equations in the self-similar form \eqref{eq: profile eq}. Notably, the second line of \eqref{eq: profile eq} captures the interaction between the point vortex and the boundary layer. This represents a key distinction from the full-space setting or from regular initial data in the half-space.

We assume that the vorticity and the position $X(t)$ admit the following asymptotic expansion in the vicinity of the point vortex:
\begin{align}\label{expand in intro}
	\mathcal W \sim \Omega_0+\nu t\Omega_2+\nu^{3/2}t\Omega_3+\cdots,\qquad
	X(t) \sim X_0(t)+\nu^{1/2}X_1(t)+\nu X_2(t)+\cdots.
\end{align}
In the near-boundary region, we assume that the vorticity admits the following asymptotic expansion:
\begin{align}\label{asymp expansion of omega b}
	\chi_b\omega(t,x,y)
	\sim \nu^{-1/2}\omega_b^{(0)}(t,x,\frac{y}{\nu^{1/2}})
	+\omega_b^{(1)}(t,x,\frac{y}{\nu^{1/2}})
	+\nu^{1/2}\omega_b^{(2)}(t,x,\frac{y}{\nu^{1/2}})
	+\cdots,
\end{align}
where $\chi_b$ is a cut-off function near the boundary. 

All that remains is to construct suitable $(\Omega_i, \omega_b^{(i)}, X_i)$ that approximate the solution to the Navier-Stokes equations.

Firstly, we choose $\Omega_0$ to be the Gaussian function and $X_0(t)$ (defined in \eqref{def: X(t)0}) to eliminate the singularity from the point vortex. We emphasize that if we take only the first term $\Omega_0$ as the approximate solution, the remainder ${\mathcal R}_{vp}\sim O(t)$ (see Lemma \ref{lem: expansion of Rvp0}) does not converge to $0$ as $\nu\rightarrow0$. It is therefore necessary to introduce the higher-order terms. 

To define $\omega_b^{(0)}$ near the boundary, we derive the expansion of velocity induced from vorticity near point vortex, namely
\begin{align*}
	BS_{\mathbb R^2_+}[\frac{\alpha}{\nu t}\chi_{vp}\cdot\Omega_0 (\frac{\cdot-X(t)}{\sqrt{\nu t}} )](x,y)
	=(u_{vp}^{(0)}, v_{vp}^{(0)})(t,x,y)+O(\nu^{1/2}),\quad for \quad y\leq 5,
\end{align*}
where $\chi_{vp}$ is a cut-off function near point vortex. By taking Taylor expansion in terms of the boundary layer variable $z=\frac{y}{\nu^{1/2}}$ and matching the order $\nu^{-1/2}$, we obtain the equation of $\omega_b^{(0)}$ (see \eqref{eq: omega p0}). We emphasize here that the initial data of $\omega_b^{(0)}$ is not zero as usual, due to the presence of an initial layer in \eqref{initial limit}. Therefore, we propose the following boundary condition for the vorticity:
\begin{align*}
	\lim_{t\rightarrow0^+}\omega_b^{(0)}=-u_0\delta_{\partial\mathbb R^2_+}\qquad in \ M(\overline{\mathbb R^2_+}).
\end{align*}
Due to the presence of an initial layer in $\omega_b^{(0)}$, we split the equation of $\omega_b^{(0)}$ into a heat equation with an initial layer that is easier to handle, and another equation without an initial layer. See the Appendix B for more details.

Armed with $\omega_b^{(0)}$, the corresponding velocity field $U_b^{(0)}$ is derived from the Biot-Savart law:
\begin{align*}
	U_b^{(0)}(t,x,y)=BS_{\mathbb R^2_+}[\nu^{-1/2}\omega_b^{(0)}(t,x,\frac{y}{\nu^{1/2}})].
\end{align*}
Here, one import property of $U_b^{(0)}$  is that it scales like $\nu^{1/2}$ near the point vortex, see Lemma \ref{est: U_b^0}. This behavior is crucial for the asymptotic matching procedure, implying that the first-order term $\Omega_2$ near the point vortex must satisfy 
\begin{align}\label{eq: Omega 2 toy model}
	\alpha\Lambda\Omega_2
	=-\frac{1}{(\nu t)^{1/2}}U_b^{(0)}\cdot\nabla_\eta \Omega_0
	+\frac{1}{t^{1/2}}X_1'(t)\cdot\nabla_\eta\Omega_0+\cdots,
\end{align}
where $\Lambda$ is a skew-adjoint differential operator defined in \eqref{def: lambda operator}. The term $U_b^{(0)}\cdot\nabla_\eta \Omega_0$ captures the boundary-vortex interaction and would vanish in a full-space setting.  To solve \eqref{eq: Omega 2 toy model}, the property of the operator $\Lambda$ implies that the right-hand side of \eqref{eq: Omega 2 toy model}  must belong to the space $\mathcal Y_n\cap \mathcal Z$, as stated in Proposition \ref{prop: properties of Lambda}. To satisfy this condition, we require the asymptotic expansion of $U_b^{(0)}\cdot\nabla_\eta \Omega_0$. This is rigorously derived via the Biot-Savart law in Lemma \ref{lem: expansion of the interaction term}, yielding:
\begin{align}\label{expand of Ub in intro}
U_b^{(0)}\cdot\nabla_\eta \Omega_0 = \nu^{1/2} C(t)\cdot\nabla_\eta\Omega_0 + \nu t^{1/2} D(\eta,t) + O(\nu^{3/2}t).
\end{align}
Meantime, we choose $X_1(t)$ in \eqref{expand in intro} to cancel the first term on the right-hand side of \eqref{expand of Ub in intro}, namely $X_1'(t)=C(t)$. This is the main idea behind the construction of $(\Omega_1, X_1)$.

 Subsequent higher-order terms can be determined through the same procedure. Therefore, we obtain an approximate solution $\omega_a=\omega_{a,b}+\omega_{a,vp}$ and define the error $\omega_R:=\omega-\omega_a$. Next, we outline the main ideas for deriving uniform estimates for the error $\omega_R$.

Following the approach in \cite{HWYZ} and \cite{WYZ}, we partition the half-space $\mathbb{R}^2_+$ into three regions: the vortex core, the boundary layer, and an intermediate region, and introduce corresponding energy functionals to control the vorticity in each. Meanwhile, we decompose the error into three components accordingly:
\begin{align*}
	\omega_R=\chi_{vp}\omega_R+\chi_b\omega_R+(1-\chi_{vp}-\chi_b)\omega_R.
\end{align*}
Since the vorticity in the intermediate region can be treated directly by standard energy methods, we focus on the remaining two components in the sequel.

\subsubsection{Strategies near the point vortex} 
The error system near the point vortex is transformed into self-similar coordinates as follows
\begin{align*}
	&t\pa_t\mathcal W_R-\mathcal L\mathcal W_R
	+\frac{\alpha}{\nu}\left\{\mathcal V^{\mathcal W_a}\cdot\nabla_\eta\mathcal W_R
	+\mathcal V^{\mathcal W_R}\cdot\nabla_\eta\mathcal W_a \right\}\\
	&\quad =-\sqrt{\frac{t}{\nu}}
	BS_{\mathbb R^2_+}[\chi_b\omega_R]\cdot\nabla_\eta\mathcal W_a+\cdots,
\end{align*}
where $BS_{\mathbb R^2_+}[\chi_b\omega_R]\cdot\nabla_\eta\mathcal W_a$ characterizes the boundary interaction. Formally, via the energy estimates, we have
\begin{align*}
	\sup_{[0,T]}\|\mathcal W_R\|_{L^2}
	\leq \sqrt{\frac{t}{\nu}}\left\|BS_{\mathbb R^2_+}[\chi_b\omega_R]\cdot\nabla_\eta\mathcal W_a\right\|_{L^2}
	+\cdots.
\end{align*}
Since our construction of the approximate solution implies 
$\mathcal W_R\sim O(\nu t)$, we require 
\begin{align*}
\|\sqrt{\frac{t}{\nu}} BS_{\mathbb R^2_+}[\chi_b\omega_R]\|_{L^\infty}\lesssim O(\nu t).
\end{align*}
 Resolving it necessitates sharp estimate on $\chi_b\omega_R$. To this end, we utilize Biot-Savart law to derive that
\begin{align*}
    \|BS_{\mathbb R^2_+}[\chi_b\omega_R]\|_{L^\infty(\chi_{vp})}
&\leq e^{-10|\xi|}\int_0^{+\infty} \frac{1-e^{-2|\xi|z}}{2z}\cdot z\chi_b(z)|(\omega_R)_\xi(z)|dz\\
&\sim \nu t \cdot\sup_{z>0}\left\| e^{-|\xi|}e^{\frac{\eps_0z^2}{\nu t}}|(\chi_b\omega_R)_\xi(z)\right\|_{L^1_\xi}\\
&\sim (\nu t)^{3/2},
\end{align*}
where we utilize the weight function $e^{\frac{\varepsilon_0 z^2}{\nu t}}$, which is part of the boundary energy functional $E_b(t)$ introduced in Section 4, to gain the factor $\nu t$.
The asymptotic property 
\begin{align*}
	\sup_{z>0}\left\| e^{-|\xi|}e^{\frac{\eps_0z^2}{\nu t}}|(\chi_b\omega_R)_\xi(z)\right\|_{L^1_\xi}
\sim (\nu t)^{1/2},
\end{align*}
can be expected from the construction of the approximate solution and is established using \eqref{integral eq of omega intro} below. For further details, we refer to Proposition \ref{prop: est of Eb(t)}.

\subsubsection{Strategies near the boundary}
 
When dealing with the vorticity near the boundary, we derive the integral equation for $\chi_b\omega_R$ as in \cite{HWYZ}, namely,
\begin{align}\label{integral eq of omega intro}
	(\chi_b\omega_R)_\xi(t,y)
	=&\int_0^t\int_0^{+\infty}\big( H_\xi(t-s,y,z)
 	+R_\xi(t-s,y,z)\big) (N_R)_\xi(s,z)dzds\\
 	\nonumber
 	&-\int_0^t\big(H_\xi(t-s,y,0)+ R_\xi(t-s,y,0)\big) (B_R)_\xi(s)ds,
\end{align}
where the semigroup kernel $R_\xi$ originates from the operator $|D_x|$ in the boundary condition \eqref{eq: BC of NS} and induces a loss of one derivative. The interaction terms $U_a\cdot\nabla\omega_R$ and $U_R\cdot\nabla\omega_R$ also introduce a loss of one derivative, either $\pa_x$ or $y\pa_y$. To resolve this difficulty, we employ a method inspired by \cite{TT Nguyen, Maekawa}, introducing the weight functions $e^{\eps_0(1+\mu)|\xi|}$ and $e^{\eps_0(1+\mu)\frac{y^2}{\nu t}}$. We then apply the inequalities in Lemma \ref{lem: analytic recovery}, which convert the derivative loss into a small-divisor problem.

Another difficulty in treating $\chi_b\omega_R$ lies in estimating the boundary term $B_R$. Since our construction of the approximate solution yields $\chi_b\omega_R\sim O(\nu)$, we require $B_R\lesssim O(\nu)$.
The boundary term $B_R$ is expressed as
\begin{align*}
	B_R=\pa_y\Delta_D^{-1}(U_R\cdot\nabla\omega_a+\cdots)|_{y=0},
\end{align*}
where the dominant term is $U_R\cdot\nabla\omega_a$. By means of the Biot-Savart law, we obtain
\begin{align}\label{est of BR intro}
	|(B_R)_\xi(s)|
	&\leq \left|\int_{\chi_{vp}} e^{2\pi i x\xi-2\pi y|\xi|}\curl\dv(U_a\otimes U_R)dxdy\right|+\cdots\\
	\nonumber
	&\leq e^{-10|\xi|}\int_{\chi_{vp}}|U_aU_R(s,X)|dxdy+\cdots,
\end{align}
where the approximate velocity satisfies $U_a(s,X)\sim \frac{1}{|X-X(s)|}$ near the point vortex. We therefore require a sharp pointwise estimate for $U_R$. The Biot-Savart law yields 
\begin{align*}
	|U_R(s,X)|
	&\leq \left|BS_{\mathbb R^2}[\chi_{vp}\omega_R](s,X)\right|
	+\cdots\\
	&\leq \left|\frac{1}{(\nu s)^{1/2}}BS_{\mathbb R^2}[\mathcal W_R](\frac{X-X(s)}{(\nu s)^{1/2}},s)\right|+\cdots,
\end{align*}
where we use the self-similar transformation and extend $\chi_{vp}\omega_R$ to the full plane $\mathbb R^2$. The construction of $\omega_a$ implies the cancellation property: $|\int_{\mathbb R^2} \mathcal W_R dxdy|\sim O(e^{-\frac{C}{\nu t}})$, which is sufficiently small. Using this property, we derive a sharper decay estimate for $BS_{\mathbb R^2}[\mathcal W_R]$ as $|\eta|\rightarrow+\infty$:
\begin{align}\label{est of WR intro squre}
	|\eta|^2|BS_{\mathbb R^2}[\mathcal W_R](\eta,t)|\leq O(\nu t).
\end{align}
For more details, we refer to Lemma \ref{lem: velocity VR pointwise}. The estimate \eqref{est of WR intro squre} implies the point-wise estimate for $U_R$:
\begin{align*}
	|U_R(s,X)|\leq (\nu s)^{1/2}\Big(1+\Big|\frac{X-X(s)}{(\nu s)^{1/2}}\Big| \Big)^{-2}+\cdots,
\end{align*}
which together with \eqref{est of BR intro} yields the estimate for $B_R$:
\begin{align*}
	|(B_R)_\xi(s)|
	\leq e^{-10|\xi|}\int_{\chi_{vp}}(\nu s)^{1/2}\Big(1+\Big|\frac{X-X(s)}{(\nu s)^{1/2}}\Big| \Big)^{-2}\cdot\frac{1}{|X-X(s)|}dX+\cdots
	\leq O(\nu s).
\end{align*}
It should be noted that if we only obtain the weaker decay rate $|\eta|^{-1}$ in \eqref{est of WR intro squre}, the integral in the above expression fails to converge as $\nu s\rightarrow0$.

\subsection{Notations}\label{Notations}
This subsection is dedicated to enumerate several notations frequently used.

     \begin{enumerate}
     
     \item In this paper, velocity is recovered from the vorticity. We frequently utilize the following notations to denote the Biot-Savart law in $\mathbb R^2$ or $\mathbb R^2_+$ for convenience.
  \begin{align}\label{BS law formulation 1}
  	\mathcal V^f
  	:=BS_{\mathbb R^2}[f]:&=\nabla^\perp\Delta^{-1} f
  	=\frac{1}{2\pi}\int_{\mathbb R^2}\frac{(X-Y)^\perp}{|X-Y|^2}f(Y)dY,
  \end{align}
  \begin{align}\label{BS law formulation 2}
  BS_{\mathbb R^2_+}[f]:&=\nabla^\perp\Delta_D^{-1} f
  	=\frac{1}{2\pi}\int_{\mathbb R^2_+}\big(\frac{(X-Y)^\perp}{|X-Y|^2}-\frac{(X-Y^*)^\perp}{|X-Y^*|^2}\big)f(Y)dY,
  \end{align}
  where $X=(x_1,x_2), Y=(y_1,y_2)\in\mathbb R^2$ and we denote $Y^*:=(y_1,-y_2)$.
  
     \item $F=(f_1,f_2):\mathbb R^2\rightarrow\mathbb R^2$, we denote  
  \begin{align}\label{def: tilde-F}
  	\widetilde F(x,y)=(-f_1(x,-y),f_2(x,-y)).
  \end{align}
  
     \item We use $f_\xi(y)$ to denote the Fourier transform about $x$ variable of function $f(x,y)$.
  
     \item $C_0$ denotes a constant independent of $\eps_0,\gamma,t,\delta$ and $C$ independent of $\gamma, t, \delta$.
  
     \item $\nabla$ denotes derivative for variable $(x,y)$ and $\nabla_\eta$ for self-similar variable $\eta$ defined later.
  
     \item For norm $\|\cdot\|_X$, $\|\cdot\|_Y$, we let $\|(1,x)f\|_X:=\|f\|_X+\|xf\|_X$ and $\|f\|_{X\cap Y}:=\|f\|_X+\|f\|_Y$.

     \item We use $f=O_{X}(g)$ to denote $\|f\|_{X}\leq g$.
  
     \item For functions $f, h$, we use $\|f\|_{L^p(h)}$ to denote $(\int_{\operatorname{supp}h}|f|^pdxdy)^{1/p}$.
  
     \item The cut-off function $\chi_{vp}$ is defined by \eqref{def: chi vp}. $\chi_m$ is defined by
     \begin{align}\label{def: chi m}
  	\chi_m(x,y)=
  	\left\{
  	\begin{aligned}
  		&1,\quad|(x,y)-(0,20)|\geq4 \ \text{and}\ y\geq\frac{3}{8},\\
  		&0,\quad|(x,y)-(0,20)|\leq3 \ \text{or}\ y\leq\frac{1}{4},
  	\end{aligned}
  	\right.
  \end{align}
  and
   $\chi_b$ is defined by
     \begin{align}\label{def: chi b}
	\chi_b(y)=
	\left\{
	\begin{aligned}
		&1,\quad y\leq2,\\
		&0,\quad y\geq3.
	\end{aligned}
	\right.
\end{align}
     
%
%
%

     \item Through this paper, $\mu_0=\frac{1}{10}$, $\beta\in(\f12,1)$ are fixed number.

     \end{enumerate}

\medskip

\section{Viscous profile of point vortex in the half plane}\label{Sec: Construction of the Approximate Solutions}

 \subsection{Approximate solutions}
 This subsection is devoted to outlining the key ideas behind the construction of the approximate solutions.
As the point vortex is positioned away from the boundary, we seek to formulate two kinds of approximate solutions: one centered around the point vortex and the other around the boundary. These strategies have been examined in different scholarly works, see \cite{Gallay 2, WWZ}. Our primary contribution lies in unifying these two kinds of solutions into a cohesive framework. 
Thus, we construct an approximate solutions enjoying the following structure:
\begin{align}\label{def:appp}
	\omega_a=\omega_{a,b}+\omega_{a,vp},
	\end{align}
where $\omega_{a,b}, \omega_{a,vp}$ provide approximate solutions in the vicinity of the boundary and the point vortex, respectively. Motivated by \cite{Gallay 2, WWZ}, we seek the approximate solution in the vicinity of the point vortex $\omega_{a,vp}$ with the following structure:
\begin{align}\label{def: omega a,vp U a,vp}
	\omega_{a,vp}(X,t):=\frac{\alpha}{\nu t}\mathcal W_a^\ast(\frac{X-X(t)}{\sqrt{\nu t}},t),\qquad
	U_{a,vp}:=BS_{\mathbb R^2_+}[\omega_{a,vp}],\qquad \textrm{with} \ X=(x,y),
\end{align}
with
\ben\label{def: X(t) ultimate}
X(t) =X_0(t)
	+\nu^{1/2}X_1(t)+\nu X_2(t),
\een
and
\begin{align}\label{def: app solution ultimate 1}
	\mathcal W_a^\ast=\chi_{vp}\mathcal W_a,\qquad
	\mathcal W_a:=\Omega_0+\nu t\Omega_2+\nu^{3/2}t\Omega_3.
\end{align}
The approximate solution near the boundary $\omega_{a,b}$ is defined with the following structure:
\begin{align}\label{def: app solution ultimate 2}
	\omega_{a,b}(t,x,y):=\chi_b(y)\big(\nu^{-1/2}\omega_b^{(0)}+\omega_b^{(1)}\big)(t,x,\frac{y}{\nu^{1/2}}),
	\qquad
	U_{a,b}:=BS_{\mathbb R^2_+}[\omega_{a,b}].
\end{align}

\begin{remark}
The expansion defined above is sufficient to establish the inviscid limit. Higher-order approximate solutions can be systematically constructed using the method presented in this paper.
\end{remark}


The construction of the approximate solution follows the sequence below:
\begin{align}\label{sequence of approximate solution}
	(\Omega_0, X_0)
	\rightarrow \omega_b^{(0)}
	\rightarrow (\Omega_2, X_1)
	\rightarrow \omega_b^{(1)}
	\rightarrow (\Omega_3, X_2)
	\rightarrow\cdots.
\end{align}

%
%
%
%

%

\medskip

To handle the advection terms, we also need expansions of the velocity near the boundary via the Biot-Savart law.
Strictly speaking, we define
\begin{align}\label{eq: def of up(k)}
	U_b^{(k)}= \big( u_b^{(k)}(t,x,y), v_b^{(k)}(t,x,y) \big)
	:=BS_{\mathbb R^2_+}[\nu^{-1/2}\chi_b\omega_b^{(k)}(t,x,\frac{y}{\nu^{1/2}})],\quad k\geq 0,
\end{align}
and
\begin{align}\label{eq: def of Uap}
	U_{a,b}=(u_{a,b}, v_{a,b})
	= (u_b^{(0)},v_b^{(0)})+\nu^{1/2} (u_b^{(1)},v_b^{(1)}).
\end{align}
Near the point vortex, the velocity is expanded by
\begin{align}\label{eq: velocity near boundary from vortex point}
	& BS_{\mathbb R^2_+}[\frac{\alpha}{\nu t}\chi_{vp}\cdot\big(\Omega_0+\nu t\Omega_2+\nu^{3/2}t\Omega_3 \big) (\frac{\cdot-X(t)}{\sqrt{\nu t}} ,t)](x,y)\\
	\nonumber
	&\sim \sum_{k\geq0}\nu^{k/2}\big(u_{vp}^{(k)}, v_{vp}^{(k)}\big)(t,x,y),\quad for \quad y\leq 5.
\end{align}
We define
\begin{align}
	U_{a,vp}=(u_{a,vp},v_{a,vp})
	=\big(u_{vp}^{(0)}, v_{vp}^{(0)}\big)
	+\nu^{1/2}\big(u_{vp}^{(1)}, v_{vp}^{(1)}\big),\quad for \quad y\leq 5.
\end{align}

\begin{remark}
 Our expansions near the point vortex are w.r.t $\nu^{1/2}$ instead of $(\nu t)^{1/2}$ introduced in \cite{Gallay 2} owing to the interaction between boundary layer and point vortex. More precisely, the term $BS_{\mathbb R^2_+}[\chi_b\omega]$ which is rooted in the vorticity near the boundary is expanded w.r.t $\nu^{1/2}$.
	\end{remark}

To characterize the approximation properties, we introduce two remainder operators to quantify the residual terms. We define the remainder operator near point vortex as
\begin{align}\label{def: Rvp}
	&\mathcal R_{vp}=L_{vp}(\mathcal W_a, \mathcal V^{\mathcal W_a}, X(t),U_{a,b})\\
	&=t\pa_t \mathcal W_a
	-\mathcal L\mathcal W_a
	+\sqrt{\frac{t}{\nu}} U_{a,b}\cdot\nabla_\eta \mathcal W_a(\eta,t)
    	\nonumber	\\
  	\nonumber
  	&\qquad\qquad+\Big\{\frac{\alpha}{\nu}\mathcal V^{\mathcal W_a}(\eta,t)
  	-\frac{\alpha}{\nu}\widetilde{\mathcal V^{\mathcal W_a}}(\eta+\widetilde\eta,t)
  	-\sqrt{\frac{t}{\nu}}X'(t)\Big\}\cdot\nabla_\eta \mathcal W_a(\eta,t),
\end{align}
and the remainder operator near the boundary as
\begin{align}\label{def: Rb}
R_b=L_b(\omega_{a,b},U_{a,b},U_{a,vp})
	=\pa_t\omega_{a,b}
	-\nu\Delta\omega_{a,b}
	+(U_{a,b}+U_{a,vp})\cdot\nabla\omega_{a,b}.
\end{align}

The following is the main result of this section.
\begin{proposition}\label{prop: est of remainder}
There  exist  constants $\nu_0, T$ small enough such that for any $0<\nu\leq\nu_0$, $0<t\leq T,\lambda \in(0,1),$ there exists  $\om_a$ defined by \eqref{def:appp} which satisfies the following estimates: 
\begin{align}\label{est: X(t)-X0}
	|X(t)-X_0|\leq1.
\end{align}
	\begin{align}\label{est of R}
		|\mathcal R_{vp}(t,\eta)|
		\leq C_\lambda\nu te^{-\lambda|\eta|^2/4},\qquad
		\left\| e^{C'|\xi|}\left\|e^{\frac{C'y^2}{\nu t}}((1,x)R_b)_\xi(t,y)\right\|_{L^1_y}\right\|_{L^1_\xi\cap L^2_\xi} \leq C\nu,
	\end{align}
	where $C_\lambda, C, C'$ are independent of $t, \nu$.
	
\end{proposition}


To conclude this subsection, we present some definitions necessary for constructing the approximate solutions.

In order to describe the profiles $\Omega_k$, we introduce the following Hilbert space as in \cite{Gallay 2}.
\begin{align}\label{def: space Y}
	\mathcal Y=\left\{ f\in L^2(\mathbb R^2): \int_{\mathbb R^2}|f(\eta)|^2e^{|\eta|^2/4}d\eta<+\infty \right\},
\end{align}
equipped with scalar product $\langle f_1,f_2\rangle_{\mathcal Y}=\int_{\mathbb R^2}f_1(\eta)f_2(\eta)e^{|\eta|^2/4}d\eta$. The space $\mathcal Y$ can be decomposed into different Fourier modes, that is
\begin{align*}
	\mathcal Y=\oplus_{n\geq0}\mathcal Y_n,\qquad with \quad \mathcal Y_n:=P_n \mathcal Y,
\end{align*}
where the orthogonal projection $P_n$ is defined in polar coordinate $\eta=(r\cos\theta, r\sin\theta)$:
\begin{align*}
  	P_n w(\eta):=a_n(r)\cos(n\theta)+b_n(r)\sin(n\theta)
\end{align*}
with
\begin{align*}
  w(\eta) =\sum_{n\geq0}a_n(r)\cos(n\theta)+b_n(r)\sin(n\theta).\end{align*}

The linear operator $\Lambda: D(\Lambda)\rightarrow\mathcal Y$ defined on domain $D(\Lambda)=\{w\in\mathcal Y, \mathcal V^G\cdot\nabla_\eta w\in \mathcal Y\}$ is defined by
\begin{align}\label{def: lambda operator}
	\Lambda w:=\mathcal V^G\cdot\nabla_\eta w+\mathcal V^w\cdot\nabla_\eta G, \qquad w\in D(\Lambda).
\end{align}

Following \cite{Gallay 2, Gallay 4}, we present below several properties of the operator $\Lambda$.
\begin{proposition}\label{prop: properties of Lambda}
	(1) The operator $\Lambda$ is skew-adjoint in $\mathcal Y$ and 
	\begin{align*}
		Ker(\Lambda)=\mathcal Y_0\oplus\{\alpha_1\pa_{\eta_1}G+\alpha_2\pa_{\eta_2}G:\alpha_1,\alpha_2\in\mathbb R\}.
	\end{align*}
	
	(2) If $f\in \mathcal Y_n\cap \mathcal Z$ for some $n\geq2$, there exists a unique $w\in \mathcal Y_n\cap\mathcal Z$ such that $\Lambda w=f.$ For $n=1$, if $f\in\mathcal Y_1\cap\mathcal Z$ and $\langle f,\pa_{\eta_1}G\rangle_{\mathcal Y}=\langle f,\pa_{\eta_2}G\rangle_{\mathcal Y}=0 $, there exists a unique $w\in \mathcal Y_1\cap\mathcal Z\cap Ker(\Lambda)^{\perp}$ such that $\Lambda w=f.$
	Here 
	\begin{align*}
		\mathcal Z:=\left\{ w : e^{|\eta|^2/4}w\in S_\ast(\mathbb R^2) \right\}\subseteq \mathcal Y,
	\end{align*}
	and $S_\ast(\mathbb R^2)$ consists of all smooth functions $w$ such that $w$ and all its derivatives have at most a polynomial growth at infinity.
\end{proposition}

\subsection{The naive order: $(\Omega_0, X_0(t))$}
We start with the Lamb-Oseen vortex, the solution to the Navier-Stokes equations in $\mathbb{R}^2$ :
\begin{align}\label{Lamb-Oseen}
	\omega(X,t)=\frac{\alpha}{\nu t}G(\frac{X}{\sqrt{\nu t}}),
	\qquad U(X,t)=\frac{\alpha}{(\nu t)^{1/2}}\mathcal V^G(\frac{X}{\sqrt{\nu t}}),\quad X \in\mathbb R^2,
\end{align}
where
\begin{align}\label{Oseen profile}
	G(\eta)=\frac{1}{4\pi}e^{-|\eta|^2/4},\qquad
	\mathcal V^G(\eta)=\frac{1}{2\pi}\frac{\eta^\perp}{|\eta|^2}\big(1-e^{-|\eta|^2/4}\big),\quad\eta\in\mathbb R^2.
\end{align}

Next, we neglect the support of $\mathcal W$ in \eqref{def: self-similar} at first. As in \cite{Gallay 2} and \cite{Dolce}, we choose 
\begin{align}\label{def: choice of Omega0}
	\Omega_0(\eta):=G(\eta)=\frac{1}{4\pi} e^{-|\eta|^2/4},
	\qquad
	\mathcal V^{\Omega_0}
	:=\mathcal V^G(\eta)=\frac{1}{2\pi}\frac{\eta^\perp}{|\eta|^2}\big(1-e^{-|\eta|^2/4}\big).
\end{align}
If we take $\mathcal W_a=\Omega_0, \mathcal V_a=\mathcal V^{\Omega_0}, U_{a,b}=0$ into \eqref{def: Rvp}, the fact \beno
t\pa_t \Omega_0=\mathcal L \Omega_0=0
\eeno
implies
\begin{align}\label{def: tilde Rvp0}
 {\mathcal R}_{vp}^{(0)}:=L_{vp}(\Omega_0, \mathcal V^{\Omega_0}, X_0(t),0)=\left\{\frac{\alpha}{\nu}\mathcal V^G(\eta)
	-\frac{\alpha}{\nu}\mathcal V^G(\eta+\widetilde\eta)
	-\sqrt{\frac{t}{\nu}}X_0'(t) \right\}\cdot\nabla_\eta G(\eta),
\end{align}
where we use the fact $\widetilde {\mathcal V^G}=\mathcal V^G$ and $\widetilde{\eta}$ is defined in \eqref{def: widetilde eta}. In order to eliminate the singularity in \eqref{def: tilde Rvp0} as $\nu\rightarrow0$, we set
\begin{align}\label{def: X(t)0}
	X_0'(t)=-\frac{\alpha}{\sqrt{\nu t}}\mathcal V^G(\widetilde\eta),
	\qquad
	X_0(0)=X_0.
\end{align}
Due to $\mathcal V^G\cdot\nabla G=0$, we obtain
\begin{align*} 
 {\mathcal R}_{vp}^{(0)} =-\frac{\alpha}{\nu}\left\{ \mathcal V^G(\eta+\widetilde\eta)
	-\mathcal V^G(\widetilde\eta) \right\}\cdot\nabla_\eta G(\eta).
\end{align*}


We employ the following lemma (Proposition 1 of \cite{Gallay 2}) to analyze $ {\mathcal R}_{vp}^{(0)}$ precisely.
\begin{lemma}\label{lem: expansion of Rvp0}
There exist functions $A\in \mathcal Y_2\cap \mathcal Z, B\in \mathcal Y_3\cap \mathcal Z, \widetilde {\mathcal R}_{vp}^{(0)}\in \mathcal Y\cap \mathcal Z$ such that
\begin{align*}
  {\mathcal R}_{vp}^{(0)}
	=tA(\eta,t)+\nu^{1/2}t^{3/2}B(\eta,t)
	+\nu t^2\widetilde R_{vp}^{(0)}(\eta,t),
\end{align*}
and for $T$ small and any $\lambda\in(0,1)$, we have
\begin{align}\label{est of AB,R0}
	 | A(\eta,t)|+| B(\eta,t)|
	+\left|\widetilde R_{vp}^{(0)}(\eta,t)\right|
	\leq Ce^{-\lambda|\eta|^2/4},
	\qquad
	0<t\leq T.
\end{align}
	
\end{lemma}



\medskip

To conclude this subsection, we derive an estimate for $\mathcal{V}^G$.
\begin{lemma}\label{lem: 1/nu of VG differ}
	For $T$ small, it holds that
	\begin{align*}
		\left|\mathcal V^G(\eta+\widetilde\eta)-\mathcal V^G(\widetilde\eta)\right|
		\leq C\nu t\langle\eta\rangle,\quad for\quad |\eta|\leq\frac{8}{\sqrt{\nu t}},\ 0<t\leq T.
	\end{align*}
\end{lemma}

\begin{proof}
	Using the definition of $\mathcal V^G$ in \eqref{Oseen profile} and $\widetilde\eta=\frac{X(t)-X(t)^\ast}{\sqrt{\nu t}}$, we have
	\begin{align*}
		&\left|\mathcal V^G(\eta+\tilde\eta)-\mathcal V^G(\tilde\eta)\right|
		\leq\frac{1}{2\pi} \left| \frac{(\eta+\tilde\eta)^\perp}{|\eta+\tilde\eta|^2}-\frac{\tilde\eta^\perp}{|\tilde\eta|^2}\right|
		+\frac{1}{2\pi} \big(e^{-\frac{|\eta+\tilde\eta|^2}{4}}
		+e^{-\frac{|\tilde\eta|^2}{4}} \big)\\
		&=\frac{1}{2\pi} \frac{\Big|\eta^\perp|\tilde\eta|^2-\tilde\eta^\perp(|\eta|^2+2\eta\cdot\tilde\eta) \Big|}{|\tilde\eta|^2|\eta+\tilde\eta|^2}
		+\frac{1}{2\pi} \big(e^{-\frac{|\eta+\tilde\eta|^2}{4}}+e^{-\frac{|\tilde\eta|^2}{4}} \big).
	\end{align*}
	For $T$ small, we use Lemma \ref{lem: ODE est} to have $|\eta|\leq\frac{8}{\sqrt{\nu t}}\leq\frac{20}{\sqrt{\nu t}}\leq|\tilde\eta|$, which implies 
	\begin{align*}
		\left|\mathcal V^G(\eta+\tilde\eta)-\mathcal V^G(\tilde\eta)\right|
		\leq C\nu t \big(\langle\eta\rangle+\frac{\langle\eta\rangle^2}{|\tilde\eta|} \big)+e^{-\frac{1}{\nu t}-\frac{|\eta|^2}{8}}
		\leq C\nu t\langle\eta\rangle.
	\end{align*}
\end{proof}


\subsection{The first order: $(\omega_b^{(0)}, \Omega_2, X_1)$}\label{sec: The first order expansion}
The aim of this subsection is to choose $(\omega_b^{(0)}, \Omega_2, X_1)$ to obtain sharper bounds for the remainder $\mathcal{R}_{vp}$.

\noindent\textbullet\ \ \textbf{Choice of $\omega_b^{(0)}$.}
Before defining $\omega_b^{(0)}$, we first describe the behavior of the velocity field induced by the vorticity near the point vortex. From the Biot-Savart law \eqref{BS law formulation 2}, we have
\begin{align}\label{eq: expansion of velocity from vortex point 0}
	&BS_{\mathbb R^2_+}[\frac{\alpha}{\nu t}\chi_{vp}\cdot\Omega_0 (\frac{\cdot-X(t)}{\sqrt{\nu t}} )](x,y)\\
	\nonumber
	&=\frac{1}{2\pi}\int_{\mathbb R^2}\big(\frac{(X-Y)^\perp}{|X-Y|^2}-\frac{(X-Y^*)^\perp}{|X-Y^*|^2}\big)\frac{\alpha}{\nu t}\chi_{vp}(Y)G(\frac{Y-X(t)}{\sqrt{\nu t}})dY \\
	\nonumber
	&=\frac{\alpha}{2\pi}\int_{\mathbb R^2}\big(
	\frac{(X-X(t)-\sqrt{\nu t}Z)^\perp}{|X-X(t)-\sqrt{\nu t}Z|^2}
	-\frac{(X-X(t)^\ast-\sqrt{\nu t}Z^\ast)^\perp}{|X-X(t)^\ast-\sqrt{\nu t}Z^\ast|^2}\big)\chi_{vp}(X(t)+\sqrt{\nu t}Z)G(Z)dZ  \\
	\nonumber
	&=(u_{vp}^{(0)}, v_{vp}^{(0)})(t,x,y)+O(\nu^{1/2}),\quad for \quad y\leq 5,
\end{align}
where $U_{vp}^{(0)}=(u_{vp}^{(0)}, v_{vp}^{(0)})$ is defined by 
\begin{align}\label{def: Uvp 0}
	U_{vp}^{(0)}(t,x,y)
	&=\frac{\alpha}{2\pi}\int_{\mathbb R^2}\big(
	\frac{(X-X(t))^\perp}{|X-X(t)|^2}
	-\frac{(X-X(t)^\ast)^\perp}{|X-X(t)^\ast|^2}\big)\chi_{vp}(X(t))G(Z)dZ\\
	\nonumber
	&=\frac{\alpha}{2\pi}\big(
	\frac{(X-X(t))^\perp}{|X-X(t)|^2}
	-\frac{(X-X(t)^\ast)^\perp}{|X-X(t)^\ast|^2}\big)\chi_{vp}(X(t)).
\end{align}

Next, we substitute the asymptotic expansion \eqref{asymp expansion of omega b}, the approximate velocity \eqref{eq: def of Uap}, and the boundary velocity \eqref{eq: velocity near boundary from vortex point} into the remainder definition \eqref{def: Rb}. Expanding the result in the boundary-layer variable $z=\frac{y}{\nu^{1/2}}$ via a Taylor series and matching terms of order $\nu^{-1/2}$ yields the governing equation for $\omega_b^{(0)}(t,x,z)$:
\begin{align}\label{eq: omega p0}
	\left\{
	\begin{aligned}
		&\pa_t \omega_b^{(0)}
		-\pa_z^2 \omega_b^{(0)}
		+\big(u_b^{(0)}+u_{vp}^{(0)}(t,x,0) \big)\pa_x\omega_b^{(0)}
		+\big(\frac{v_b^{(0)}}{y} +\pa_y v_{vp}^{(0)}(t,x,0) \big) z\pa_z\omega_b^{(0)}=0,\\
		&\lim_{t\rightarrow0^+}\omega_b^{(0)}=-u_0\delta_{\partial\mathbb R^2_+} .
	\end{aligned}
	\right.
\end{align}
Here the velocity $(u_b^{(0)}, v_b^{(0)})$ is deduced from $\omega_b^{(0)}$ by 
 \ben\label{def:u_b^0}
(u_b^{(0)}, v_b^{(0)})(t, x, y)=BS_{\mathbb R^2_+}[\nu^{-1/2}\chi_b\omega_b^{(0)}(t,x,\frac{y}{\nu^{1/2}})].
 \een

Solving the above equation requires the boundary condition for $\omega_b^{(0)}$  given in Section \ref{sec: boundary condition expansion}.

\begin{remark}
The initial condition for $\omega_b^{(0)}$ arises because the initial Navier-Stokes velocity does not satisfy the no-slip boundary condition and the vorticity is discontinuous at $t=0$ (see \eqref{initial limit}). Further details are provided in \cite{WYZ}. 
 \end{remark}

 To derive the higher-order expansion, we require the asymptotic behavior of $(u_b^{(0)}, v_b^{(0)})$.
\begin{lemma}\label{est: U_b^0}
Let $\omega_b^{(0)}$ be constructed by \eqref{eq: omega p0}. Assume that
\beno
 \int_0^{+\infty} |(1+|z|)(\omega_b^{(0)})_\xi(t,z) |dz \leq C_0.
\eeno
Then we have
\begin{align}\label{eq: behavior of (u_p^{(0)}, v_p^{(0)})}
	(u_b^{(0)},\frac{ v_b^{(0)}}{y}) \sim
	\left\{
	\begin{aligned}
		&O(1),\qquad
		y\leq5,\\
		&O(\nu^{1/2}),\qquad
		y\geq10.
	\end{aligned}
	\right.
\end{align}

\end{lemma}
\begin{proof}

Firstly, by Lemma \ref{lem: derivation of velocity formula}, we obtain that for $y\leq5$,
\begin{align*}
	|( u_b^{(0)})_\xi(t,y)|
	\leq \f12 \int_0^{+\infty}\nu^{-1/2}\left|(\omega_b^{(0)})_\xi(t,\frac{z}{\nu^{1/2}})\right|\chi_b(z)dz
	\leq \int_0^{+\infty} \left|(\omega_b^{(0)})_\xi(t,z)\right|dz=O(1).
\end{align*}
For $y\geq10$, due to \eqref{BS law formulation 2}, it holds that
\begin{align*}
	u_b^{(0)}(t,x,y)
	&=\frac{1}{2\pi}\int_{\mathbb R^2_+}
	\big(\frac{y+\tilde y}{(x-\tilde x)^2+(y+\tilde y)^2}
	-\frac{y-\tilde y}{(x-\tilde x)^2+(y-\tilde y)^2}\big)
	\nu^{-1/2}\omega_b^{(0)}(t,\tilde x,\frac{\tilde y}{\nu^{1/2}})\chi_b(\tilde y)d\tilde xd\tilde y\\
	&=\frac{1}{\pi}\int_{\mathbb R^2_+}
	\frac{\nu^{1/2}\tilde z\big((x-\tilde x)^2+\nu\tilde z^2-y^2\big)}{\big((x-\tilde x)^2+(y+\nu^{1/2}\tilde z)^2\big)\big((x-\tilde x)^2+(y-\nu^{1/2}\tilde z)^2\big)}
	\omega_b^{(0)}(t,\tilde x,\tilde z)\chi_b(\nu^{1/2}\tilde z)d\tilde xd\tilde z.
\end{align*}
Since $\nu^{1/2}\tilde z\leq3$ on $supp\chi_b$, thus for $y\geq10$, 
\begin{align*}
	\left| \frac{\nu^{1/2}\tilde z\big((x-\tilde x)^2+\nu\tilde z^2-y^2\big)}{\big((x-\tilde x)^2+(y+\nu^{1/2}\tilde z)^2\big)\big((x-\tilde x)^2+(y-\nu^{1/2}\tilde z)^2\big)} \right|
	\leq \frac{\nu^{1/2}\tilde z}{(x-\tilde x)^2+(y-\nu^{1/2}\tilde z)^2}
	\leq \nu^{1/2}\tilde z,
\end{align*}
which implies
\begin{align*}
	|u_b^{(0)}(t,x,y)|
	\leq \frac{\nu^{1/2}}{\pi}\int_{\mathbb R^2_+}\tilde z |\omega_b^{(0)}(t,\tilde x,\tilde z)|\chi_b(\nu^{1/2}\tilde z)d\tilde xd\tilde z
	=O(\nu^{1/2}).
\end{align*}
The proof of $\frac{v_b^{(0)}}{y}$ is obtained by the same argument.
\end{proof}

\medskip

\noindent\textbullet\ \ \textbf{Construction of $(\Omega_2,X_1)$.}
 In \eqref{def: Rvp}, we take $\mathcal W_a, \mathcal V^{\mathcal W_a}, X(t)$ and $U_{a,b}$ as following
\begin{align*}
	\mathcal W_a=\Omega_0+\nu t\Omega_2,\qquad
	\mathcal V^{\mathcal W_a}=\mathcal V^{\Omega_0}+\nu t\mathcal V^{\Omega_2},\qquad
	X'(t)=-\frac{\alpha}{\sqrt{\nu t}}\mathcal V^G(\frac{X(t)-X(t)^\ast}{\sqrt{\nu t}})+\nu^{1/2}X_1'(t),
\end{align*}
and $U_{a,b}=U_b^{(0)}=(u_b^{(0)},v_b^{(0)})$ defined by \eqref{def:u_b^0}.  Then, we obatin
\begin{align}\label{def: tilde Rvp1}
	 {\mathcal R}^{(1)}_{vp}:=&(t\pa_t-\mathcal L)(\nu t\Omega_2)
	+\Big\{ \frac{\alpha}{\nu}\mathcal V^G(\eta)
	-\frac{\alpha}{\nu}\mathcal V^G(\eta+\widetilde\eta)
	+\frac{\alpha}{\nu}\mathcal V^G(\widetilde\eta)
	+\alpha t\mathcal V^{\Omega_2}(\eta,t)\\
	\nonumber
	&
	\quad-\alpha t \widetilde{{\mathcal V}^{\Omega_2}}(\eta+\widetilde\eta,t)
	-t^{1/2}X_1'(t)+\sqrt{\frac{t}{\nu}}U_b^{(0)} \Big\}\cdot\nabla_\eta (\Omega_0+\nu t\Omega_2)\\
	\nonumber
	=& {\mathcal R}_{l,1}+  {\mathcal R}_{h,1},
\end{align}
where $  {\mathcal R}_{l,1}$ and $ {\mathcal R}_{h,1}$ consist of low and high order terms respectively and are defined by
\begin{align*}
  {\mathcal R}_{l,1}
	= {\mathcal R}^{(0)}_{vp}
	+\alpha t\Lambda\Omega_2-t^{1/2}X_1'(t)\cdot\nabla_\eta G
	+\sqrt{\frac{t}{\nu}}U_b^{(0)}\cdot\nabla_\eta G,
\end{align*}
and
\begin{align*}
  {\mathcal R}_{h,1}
	=&-\alpha t \left\{\mathcal V^G(\eta+\widetilde\eta)-\mathcal V^G(\widetilde\eta) \right\}\cdot\nabla_\eta \Omega_2
	+(t\pa_t-\mathcal L)(\nu t\Omega_2)
	-\nu t^{3/2} X_1(t)\cdot\nabla_\eta\Omega_2\\
	&+\nu^{1/2}t^{3/2}U_b^{(0)}\cdot\nabla_\eta\Omega_2
	-\alpha t\widetilde{{\mathcal V}^{\Omega_2}}(\eta+\widetilde\eta,t)\cdot\nabla_\eta G -\alpha\nu t^2\left\{\widetilde{{\mathcal V}^{\Omega_2}}(\eta+\widetilde\eta,t)-\mathcal V^{\Omega_2}(\eta,t)\right\}\cdot\nabla_\eta\Omega_2.
\end{align*}

 \medskip

The following is the main results of this subsection.
\begin{lemma}\label{lem: R^{(1)}_{vp} est}
There exist $\Om_2 \in\mathcal Y_2\cap\mathcal Z$ and smooth function $X_1(t)$ such that
\beno
 {\mathcal R}^{(1)}_{vp} \sim O_{\mathcal Y}(\nu^{1/2}t).
\eeno
\end{lemma}

\medskip

Before proving Lemma \ref{lem: R^{(1)}_{vp} est},  we firstly treat the interaction term $U_b^{(0)}\cdot\nabla_\eta G$. Here, we give a general version which treats $(u_b^{(k)},v_b^{(k)})\cdot\nabla_\eta G$, where $(u_b^{(k)},v_b^{(k)})$ is defined by \eqref{eq: def of up(k)}.

\begin{lemma}\label{lem: expansion of the interaction term}
	There exist bounded function $C_k(t)$, and $D_k\in\mathcal Y_2\cap\mathcal Z, \widetilde{R}_{int}^{(k)}\in\mathcal Y\cap\mathcal Z$ such that
	\begin{align*}
		(u_b^{(k)},v_b^{(k)})\cdot\nabla_\eta G(\eta)
		=\nu^{1/2}C_k(t)\cdot\nabla_\eta G(\eta)
		+\nu t^{1/2}D_k(\eta,t)
		+\nu^{3/2}t \widetilde{R}_{int}^{(k)}(\eta,t),\quad for \ |\eta|\leq\frac{8}{\sqrt{\nu t}},
	\end{align*}
	and for $T$ small and any $\lambda\in(0,1)$, we have
	\begin{align}\label{est of CDk,Rintk}
		|D_k(\eta,t)|
		+\left|\widetilde{R}_{int}^{(k)}(\eta,t)\right|
		\leq Ce^{-\lambda|\eta|^2/4},
	\qquad \textrm{for}\quad   |\eta|\leq\frac{8}{\sqrt{\nu t}},\quad
	0<t\leq T.
	\end{align}
\end{lemma}

\begin{proof}
	Due to the Biot-Savart law \eqref{BS law formulation 2}, the definition of $(u_b^{(k)},v_b^{(k)})$ \eqref{eq: def of up(k)} and the relation $\nabla_\eta G=-\f12\eta G$, we have
	\begin{align*}
		&(u_b^{(k)},v_b^{(k)})\cdot\nabla_\eta G(\eta)\\
		=&-\frac{1}{4\pi}\int_{\mathbb R^2_+} \frac{\big(-(y-\tilde y),x-\tilde x\big)}{(x-\tilde x)^2+(y-\tilde y)^2}\cdot\nu^{-1/2}\omega_b^{(k)}(t,\tilde x,\frac{\tilde y}{\nu^{1/2}})\chi_b(\tilde y)d\tilde xd\tilde y\cdot\eta G(\eta)\\
		&+\frac{1}{4\pi}\int_{\mathbb R^2_+} \frac{\big(-(y+\tilde y),x-\tilde x\big)}{(x-\tilde x)^2+(y+\tilde y)^2}\cdot\nu^{-1/2}\omega_b^{(k)}(t,\tilde x,\frac{\tilde y}{\nu^{1/2}})\chi_b(\tilde y)d\tilde xd\tilde y\cdot\eta G(\eta)=I_1+I_2.
	\end{align*}
	
	For $I_1$, due to $\eta=\frac{(x,y)-X(t)}{\sqrt{\nu t}}$, we have
	\begin{align*}
		&\frac{\big(-(y-\tilde y),x-\tilde x\big)}{(x-\tilde x)^2+(y-\tilde y)^2}\cdot\eta
		=\frac{\big(-\sqrt{\nu t}\eta_2-x_2(t)+\tilde y,\sqrt{\nu t}\eta_1+x_1(t)-\tilde x \big)}{(\sqrt{\nu t}\eta_1+x_1(t)-\tilde x)^2 +(\sqrt{\nu t}\eta_2+x_2(t)-\tilde y)^2}\cdot(\eta_1,\eta_2)\\
		&=\frac{-\eta_1\big(x_2(t)-\tilde y\big)+\eta_2\big(x_1(t)-\tilde x\big)}{(\sqrt{\nu t}\eta_1+x_1(t)-\tilde x \big)^2 +(\sqrt{\nu t}\eta_2+x_2(t)-\tilde y\big)^2}\\
		&=(\nu t)^{-1/2}\frac{-\eta_1 b_2+\eta_2 b_1}{|\eta|^2+2(\eta_1 b_1+\eta_2 b_2)+|b|^2}
		=(\nu t)^{-1/2}\frac{x_2}{1+2x_1+|x|^2},
	\end{align*}
	where we set $b_1=\frac{x_1(t)-\tilde x}{\sqrt{\nu t}}, b_2=\frac{x_2(t)-\tilde y}{\sqrt{\nu t}}, c_1=\frac{\eta_1 b_1+\eta_2 b_2}{|b|^2}, c_2=\frac{-\eta_1 b_2+\eta_2 b_1}{|b|^2}$.

The following expansion has been established in the Appendix of \cite{Dolce}:
\begin{align*}
	\frac{c_2}{1+2c_1+|c|^2}
	=c_2-2c_1c_2+O(|c|^3),\qquad for \quad |c|\leq\f34.
\end{align*}
Plugging the definition of $c_1, c_2$, we obtain for $|\eta|\leq\frac{8}{\sqrt{\nu t}}$,
\begin{align*}
	\frac{\big(-(y-\tilde y),x-\tilde x\big)}{(x-\tilde x)^2+(y-\tilde y)^2}\cdot\eta
		=&\frac{-\eta_1(x_2(t)-\tilde y)+\eta_2(x_1(t)-\tilde x)}{(x_1(t)-\tilde x)^2+(x_2(t)-\tilde y)^2} \\
		&+2(\nu t)^{1/2}(\eta_1^2-\eta_2^2)\frac{(x_1(t)-\tilde x)(x_2(t)-\tilde y)}{\big((x_1(t)-\tilde x)^2+(x_2(t)-\tilde y)^2\big)^2}\\
		&-2(\nu t)^{1/2}\eta_1\eta_2\frac{(x_1(t)-\tilde x)^2-(x_2(t)-\tilde y)^2}{\big((x_1(t)-\tilde x)^2+(x_2(t)-\tilde y)^2\big)^2}+r_1(\eta,\tilde x,\tilde y),
\end{align*}
where $|r_1(\eta,\tilde x,\tilde y)|\leq C\nu t|\eta|^3$. $I_2$ can be computed similarly by replacing $x_2(t)-\tilde y$ with $x_2(t)+\tilde y$. Thus, there exist functions $f_1, f_2, f_3, f_4$ $\in C_b^\infty( supp \chi_b)$ for $t$ small such that
\begin{align*}
	&-\frac{\big(-(y-\tilde y),x-\tilde x\big)}{(x-\tilde x)^2+(y-\tilde y)^2}\cdot\eta
	+\frac{\big(-(y+\tilde y),x-\tilde x\big)}{(x-\tilde x)^2+(y+\tilde y)^2}\cdot\eta\\
	&=\Big\{\frac{\eta_1(x_2(t)-\tilde y)-\eta_2(x_1(t)-\tilde x)}{(x_1(t)-\tilde x)^2+(x_2(t)-\tilde y)^2}+\frac{-\eta_1(x_2(t)+\tilde y)+\eta_2(x_1(t)-\tilde x)}{(x_1(t)-\tilde x)^2+(x_2(t)+\tilde y)^2}\Big\} \\
	&\quad+2(\nu t)^{1/2}(\eta_1^2-\eta_2^2)\Big\{-\frac{(x_1(t)-\tilde x)(x_2(t)-\tilde y)}{\big((x_1(t)-\tilde x)^2+(x_2(t)-\tilde y)^2\big)^2}
	+\frac{(x_1(t)-\tilde x)(x_2(t)+\tilde y)}{\big((x_1(t)-\tilde x)^2+(x_2(t)+\tilde y)^2\big)^2} \Big\} \\
	&\quad+2(\nu t)^{1/2}\eta_1\eta_2 \Big\{\frac{(x_1(t)-\tilde x)^2-(x_2(t)-\tilde y)^2}{\big((x_1(t)-\tilde x)^2+(x_2(t)-\tilde y)^2\big)^2}-\frac{(x_1(t)-\tilde x)^2-(x_2(t)+\tilde y)^2}{\big((x_1(t)-\tilde x)^2+(x_2(t)+\tilde y)^2\big)^2}\Big\} \\
	&\quad+r_1(\eta,\tilde x,\tilde y)+r_2(\eta,\tilde x,\tilde y) \\
	&=\tilde y\Big(\eta_1 f_1(X(t),\tilde x,\tilde y)+\eta_2 f_2(X(t),\tilde x,\tilde y)  \Big)\\
	&\qquad +(\nu t)^{1/2}\tilde y\Big(\eta_1\eta_2  f_3(X(t),\tilde x,\tilde y)+(\eta_1^2-\eta_2^2) f_4(X(t),\tilde x,\tilde y) \Big)
	+\tilde y r(\eta,\tilde x,\tilde y),
\end{align*}
where $|r(\eta,\tilde x,\tilde y)|\leq C\nu t|\eta|^3$ and the factor $\tilde y$ appears since 
\begin{align*}
	\frac{\big(-(y-\tilde y),x-\tilde x\big)}{(x-\tilde x)^2+(y-\tilde y)^2}=\frac{\big(-(y+\tilde y),x-\tilde x\big)}{(x-\tilde x)^2+(y+\tilde y)^2},\qquad when \quad \tilde y=0.
\end{align*}
Then we have
\begin{align*}
	&(u_b^{(k)},v_b^{(k)})\cdot\nabla_\eta G(\eta)
	=\frac{1}{4\pi}\int_{\mathbb R^2_+}
	\Big\{ \nu^{1/2}\tilde z\Big(\eta_1 f_1(X(t),\tilde x,\nu^{1/2}\tilde z)+\eta_2 f_2(X(t),\tilde x,\nu^{1/2}\tilde z)  \Big)\\
	&\qquad\qquad+\nu t^{1/2}\tilde z\Big(\eta_1\eta_2  f_1(X(t),\tilde x,\nu^{1/2}\tilde z)+(\eta_1^2-\eta_2^2) f_2(X(t),\tilde x,\nu^{1/2}\tilde z) \Big) +\nu^{1/2}\tilde z r(\eta,\tilde x,\nu^{1/2}\tilde z) \Big\}\\
	&\qquad\qquad\qquad\qquad\cdot\omega_b^{(0)}(t,\tilde x,\tilde z)\chi_b(\nu^{1/2}\tilde z)d\tilde xd\tilde z\cdot G(\eta),
\end{align*}
which along with Proposition \ref{prop: WP of boundary layer}
 gives the desired result.
\end{proof}


Now we turn to prove Lemma \ref{lem: R^{(1)}_{vp} est}.\medskip

\noindent{\bf Proof of Lemma \ref{lem: R^{(1)}_{vp} est}}. Substituting Lemma \ref{lem: expansion of Rvp0} and Lemma \ref{lem: expansion of the interaction term} to $ {\mathcal R}_{l,1}$, we obtain
\begin{align*}
  {\mathcal R}_{l,1}
	=&\alpha t\Lambda\Omega_2
	+tA(\eta,t)
	+t^{1/2}C_0(t)\cdot\nabla_\eta G(\eta)
	-t^{1/2}X_1'(t)\cdot\nabla_\eta G(\eta) \\
	&+\nu^{1/2}t \big(t^{1/2}B(\eta,t)+D_0(\eta,t)\big)
	+\nu t^{3/2} \big(t^{1/2}R_{vp}^{(0)}(\eta,t)
	+R_{int}^{(0)}(\eta,t)\big).
\end{align*}
It is natural to set
\begin{align*}
	t\alpha\Lambda\Omega_2
	+tA(\eta,t)
	+t^{1/2}C_0(t)\cdot\nabla_\eta G(\eta)
	-t^{1/2}X_1(t)\cdot\nabla_\eta G(\eta)=0.
\end{align*}
Recalling $A\in\mathcal Y_2$, thus by Proposition  \ref{prop: properties of Lambda}, to fulfill the solvability condition we set 
\begin{align}\label{eq: eq of Omega2}
	X_1'(t)=C_0(t),\quad X_1(0)=(0,0), \qquad \textrm{and} \quad \alpha\Lambda\Omega_2
	+A(\eta,t)=0.
\end{align}
Thus, by Proposition  \ref{prop: properties of Lambda}, there exists a unique $\Omega_2\in\mathcal Y_2\cap\mathcal Z$. By now, we obtain 
\beno
  {\mathcal R}_{l,1}\sim O_{\mathcal Y}(\nu^{1/2}t).
\eeno

For $  {\mathcal R}_{h,1}$, using Lemma \ref{lem: 1/nu of VG differ}, Lemma \ref{lem: decay rate of velocity} and Lemma \ref{est: U_b^0}, we obtain
\beno
|  {\mathcal R}_{h,1}| \leq C_\lambda\nu^{1/2}t e^{-\lambda|\eta|^2/4}, \quad \textrm{for}\quad |\eta|\leq\frac{8}{\sqrt{\nu t}},\quad \lambda\in(0,1),
\eeno
which completes the proof of the lemma.

\begin{remark}
	Setting $\mathcal{W}_a = \Omega_0 + (\nu t)^{1/2}\Omega_1 + \nu t \Omega_2$ in \eqref{def: Rvp} and matching terms of order $\nu^{1/2}$ yields $\Lambda\Omega_1=0$. We therefore set $\Omega_1=0$ to simplify the subsequent analysis.
\end{remark}

\subsection{The second order: $(\omega_b^{(1)}, \Omega_3, X_2)$}\label{sec: The second order}

In this subsection, we aim to choose $(\omega_b^{(1)}, \Omega_3, X_2)$ to establish sharper bounds for the remainders ${R}_b\sim \nu$ and $\mathcal R_{vp}\sim \nu t$.

\noindent\textbullet\ \ \textbf{Choice of $\omega_b^{(1)}$.}
To define $\omega_b^{(1)}$, we firstly derive the velocity induced by $\Omega_0+\nu t\Omega_2$: 
\begin{align}\label{eq: expansion of velocity from vortex point 1}
	&BS_{\mathbb R^2_+}[\frac{\alpha}{\nu t}\chi_{vp}\cdot\big(\Omega_0+\nu t\Omega_2 \big) (\frac{\cdot-X(t)}{\sqrt{\nu t}} ,t)](x,y)\\
	\nonumber
	&=\frac{\alpha}{2\pi}\int_{\mathbb R^2}\big(
	\frac{(X-X(t)-\sqrt{\nu t}Z)^\perp}{|X-X(t)-\sqrt{\nu t}Z|^2}
	-\frac{(X-X(t)^\ast-\sqrt{\nu t}Z^\ast)^\perp}{|X-X(t)^\ast-\sqrt{\nu t}Z^\ast|^2}\big)\\
	\nonumber
	&\qquad\qquad\cdot\chi_{vp}(X(t)+\sqrt{\nu t}Z)\cdot (G+\nu t\Omega_2)(t,Z)dZ\\
	\nonumber
	&=(u_{vp}^{(0)}, v_{vp}^{(0)})(t,x,y)+\nu^{1/2}(u_{vp}^{(1)}, v_{vp}^{(1)})(t,x,y) +O(\nu),\quad for \quad y\leq 5,
\end{align}
where  $U_{vp}^{(1)}=(u_{vp}^{(1)}, v_{vp}^{(1)})$ is defined by
\begin{align}\label{def: Uvp 1}
	U_{vp}^{(1)}:=\lim_{\nu\rightarrow0^+}\nu^{-1/2}
	\Big\{BS_{\mathbb R^2_+}[\frac{\alpha}{\nu t}\chi_{vp}\cdot\big(\Omega_0+\nu t\Omega_2 \big) (\frac{\cdot-X(t)}{\sqrt{\nu t}} ,t)](x,y)
	-U_{vp}^{(0)}(t,x,y)\Big\}.
\end{align}

Substituting  \eqref{eq: def of Uap}, \eqref{eq: velocity near boundary from vortex point} into \eqref{eq: NS vorticity}, expanding the result in the boundary-layer variable  $z=\frac{y}{\nu^{1/2}}$ via a Taylor series and matching the order of $\nu^0$ yields the governing equation for $\omega_b^{(1)}(t,x,z)$ as
\begin{align}\label{eq: omega p1}
	\left\{
	\begin{aligned}
		&\pa_t \omega_b^{(1)}-\pa_z^2 \omega_b^{(1)}
		+ \big(u_b^{(0)}+u_{vp}^{(0)}(t,x,0) \big)\pa_x\omega_b^{(1)}
		+\big(\frac{v_b^{(0)}}{y}+\pa_y v_{vp}^{(0)}(t,x,0) \big)z\pa_z\omega_b^{(1)}
		\\
		&\qquad+\Big\{u_b^{(1)}+z\pa_y u_b^{(0)}(t,x,0)+u_{vp}^{(1)}(t,x,0) \Big\}\pa_x\omega_b^{(0)}\\
		&\qquad+\Big\{ \frac{v_b^{(1)}}{y}+\frac{z}{2}\pa_y^2 v_{vp}^{(1)}(t,x,0)+\pa_y v_{vp}^{(1)}(t,x,0) \Big\}z\pa_z\omega_b^{(0)}=0, \\
		&\omega_b^{(1)}|_{t=0}=0.
	\end{aligned}
	\right.
\end{align}

Now,  recalling the definition of reminder term $R_b$, we define
\beno
\widetilde{R}_b=L_b\big(\nu^{-1/2}\omega_b^{(0)}+\omega_b^{(1)}, U_b^{(0)}+\nu^{1/2}U_b^{(1)},U_{vp}^{(0)}+\nu^{1/2}U_{vp}^{(1)} \big).
\eeno
Then, substituting the above $(U_{vp}^i, \omega_b^{(i)})$ into \eqref{def: Rb} and Proposition \ref{prop: WP of boundary layer}, for $T$ small enough, we obtain 
\begin{align}\label{est: tilde Rb}
	\sum_{i+j\leq8}\left\| e^{C'|\xi|}\left\|e^{\frac{C'y^2}{\nu t}}(\pa_x^i(y\pa_y)^j(1,x)\widetilde{R}_b)_\xi(t,y)\right\|_{L^1_y}\right\|_{L^1_\xi\cap L^2_\xi} \leq C\nu,\qquad 0\leq t\leq T,
\end{align}
where $C,C'$ are independent of $\nu$.

\smallskip

\noindent\textbullet\ \ \textbf{Choice of $(\Omega_3, X_2)$.}
Taking
\begin{align*}
	\mathcal W_a
	&=\Omega_0+\nu t\Omega_2+\nu^{3/2}t\Omega_3,\qquad
	\mathcal V_a
	=\mathcal V^{\Omega_0}+\nu t\mathcal V^{\Omega_2}+\nu^{3/2}t\mathcal V^{\Omega_3} ,\\
  X'(t)&=-\frac{\alpha}{\sqrt{\nu t}}\mathcal V^G(\widetilde\eta)+\nu^{1/2}X_1'(t)+\nu X_2'(t),\quad U_{a,b}=U_b^{(0)}+\nu^{1/2}U_b^{(1)}.
\end{align*}
Substituting the them into \eqref{def: Rvp}, we obtain
\begin{align*}
 {\mathcal R}^{(2)}_{vp}:= {\mathcal R}_{l,2}+ {\mathcal R}_{h,2},
\end{align*}
where $  {\mathcal R}_{l,2}$ and $ {\mathcal R}_{h,2}$ consist of low and high order terms respectively and are defined by
\begin{align*}
 {\mathcal R}_{l,2}
	=&{\mathcal R}^{(0)}_{vp}+\alpha t\Lambda\Omega_2-t^{1/2}X_1'(t)\cdot\nabla_\eta G
	+\alpha \nu^{1/2}t\Lambda\Omega_3-\nu^{1/2}t^{1/2}X_2'(t)\cdot\nabla_\eta G\\
	&+\sqrt{\frac{t}{\nu}}\left\{ U_b^{(0)}+\nu^{1/2}U_b^{(1)}\right\}\cdot\nabla_\eta G,
\end{align*}
and
\begin{align*}
 {\mathcal R}_{h,2}
	=&(t\pa_t-\mathcal L)(\nu t\Omega_2+\nu^{3/2}t\Omega_3)
	-\nu t^{3/2}\big( X_1'(t)+\nu^{1/2}X_2'(t) \big)\cdot\nabla_\eta(\Omega_2+\nu^{1/2}\Omega_3)\\
	&-\alpha t\left\{ \widetilde{\mathcal V^{\Omega_2}}(\eta+\widetilde\eta)
	+\nu^{1/2}\widetilde{\mathcal V^{\Omega_3}}(\eta+\widetilde\eta) \right\}\cdot\nabla_\eta G\\
	&-\alpha t\left\{ \mathcal V^G(\eta+\widetilde\eta)-\mathcal V^G(\widetilde\eta)\right\}\cdot\nabla_\eta(\Omega_2+\nu^{1/2}\Omega_3) \\
	&-\alpha\nu t^2 \left\{\widetilde{\mathcal V^{\Omega_2}}(\eta+\widetilde\eta)-\mathcal V^{\Omega_2}(\eta,t)\right\}\cdot\nabla_\eta(\Omega_2+\nu^{1/2}\Omega_3)\\
	&-\alpha\nu^{3/2} t^2 \left\{\widetilde{\mathcal V^{\Omega_3}}(\eta+\widetilde\eta)-\mathcal V^{\Omega_3}(\eta,t)\right\}\cdot\nabla_\eta(\Omega_2+\nu^{1/2}\Omega_3)\\
	&+\nu^{1/2}t^{3/2}\left\{ U_b^{(0)}+\nu^{1/2}U_b^{(1)}\right\}\cdot\nabla_\eta(\Omega_2+\nu^{1/2}\Omega_3).
\end{align*}

 Here, this is the main results of this subsection
\begin{lemma}\label{lem: R^{(2)}_{vp}}
There exist $\Om_3 \in(\mathcal Y_2\oplus\mathcal Y_3)\cap\mathcal Z$ and smooth function $X_2(t)$ such that
\beno
 {\mathcal R}^{(2)}_{vp} \sim O_{\mathcal Y}(\nu t).
\eeno
\end{lemma}
\begin{proof}
Firstly, using Lemma \ref{lem: expansion of Rvp0}, Lemma \ref{lem: expansion of the interaction term} and \eqref{eq: eq of Omega2}, we rewrite $ {\mathcal R}_{l,2}$ as following
\begin{align*}
 {\mathcal R}_{l,2}
	=&\nu^{1/2}t \left\{\alpha\Lambda\Omega_3+t^{1/2}B(\eta,t)+D_0(\eta,t) \right\}
	+\nu^{1/2}t^{1/2} \big(C_1(t)\cdot\nabla_\eta G(\eta)-X_2'(t)\cdot\nabla_\eta G\big)\\
	&+\nu t \big(D_1(\eta,t)+t^{1/2}\widetilde{ R}_{int}^{(0)}(\eta,t)+t \widetilde{ R}_{vp}^{(0)}(\eta,t)+\nu^{1/2}t^{1/2}\widetilde{\mathcal R}_{int}^{(1)}(\eta,t) \big).
\end{align*}
To eliminate the terms with the order $\nu^{\f12}$ in the above equation, we set
\begin{align}\label{eq: eq of Omega3}
    X_2'(t)=C_1(t),\quad X_2(0)=(0,0),
	\qquad \alpha\Lambda\Omega_3+t^{1/2}B(\eta,t)+D_0(\eta,t)=0.
\end{align}
Recalling $B\in\mathcal Y_3, D_0\in\mathcal Y_2$,  by Proposition  \ref{prop: properties of Lambda}, there exists a unique $\Omega_3\in\mathcal Y_2\oplus\mathcal Y_3$ such that
\beno
\alpha\Lambda\Omega_3+tB(\eta,t)+D_0(\eta,t)=0.
\eeno
Therefore, we obtain $ {\mathcal R}_{l,2}\sim O_{\mathcal Y}(\nu t)$.

For $  {\mathcal R}_{h,2}$, by Lemma \ref{lem: 1/nu of VG differ}, Lemma \ref{lem: decay rate of velocity} and Lemma \ref{est: U_b^0}, we obtain
\beno
|  {\mathcal R}_{h,2}| \leq C_\lambda\nu^{1/2}t e^{-\lambda|\eta|^2/4}, \quad \textrm{for}\quad |\eta|\leq\frac{8}{\sqrt{\nu t}},\quad \lambda\in(0,1).
\eeno

For $  {\mathcal R}_{h,2}$, we have $  {\mathcal R}_{h,2}\sim O(\nu t)$ for $|\eta|\leq\frac{8}{\sqrt{\nu t}}$, which is obtained from $|\mathcal V^G(\eta)|\sim \min\{|\eta|,\frac{1}{|\eta|}\}$, Lemma \ref{lem: decay rate of velocity}, Lemma \ref{lem: 1/nu of VG differ} and $U_b^{(0)}\sim \nu^{1/2}$ away from boundary (see \eqref{eq: behavior of (u_p^{(0)}, v_p^{(0)})}).

Summing the above estimates, we obtain the desired results.
\end{proof}


\subsection{Boundary condition}\label{sec: boundary condition expansion}

This section is devoted to deriving the asymptotic expansion of the boundary condition \eqref{eq: BC of NS}. Bringing $\om_{b}$ into the left hand side of \eqref{eq: BC of NS} gives
\begin{align*}
	\nu(\pa_y+|D_x|)\omega|_{y=0}
	\sim
	\pa_z\omega_b^{(0)}|_{z=0}
	+\nu^{1/2}\big(\pa_z\omega_b^{(1)}+|D_x|\omega_b^{(0)}\big)|_{z=0}
	+\cdots.
\end{align*}

The right hand side of \eqref{eq: BC of NS} contains nonlocal operator $\pa_y\Delta_D^{-1}$ which relies on the behavior of $U\cdot\nabla\omega$ in the whole $\mathbb R^2_+$.

We consider $U\cdot\nabla\omega$ near the boundary at first.
\begin{align*}
	U_a\cdot\nabla\omega_{a,b}
	=&
	\left\{U_{vp}^{(0)}
	+U_b^{(0)}
	+\nu^{1/2}U_{vp}^{(0)}
	+\nu^{1/2}U_b^{(1)}
	+\cdots\right\}\\
	&\quad\cdot\nabla \left\{\nu^{-1/2}\chi_b(y)\omega_b^{(0)}(t,x,\frac{y}{\nu^{1/2}})
	+\chi_b(y)\omega_b^{(1)}(t,x,\frac{y}{\nu^{1/2}})
	+\cdots\right\}\\
	=
	&BC_{b,0}+\nu^{1/2}BC_{b,1},\qquad
	for \quad y\leq5,
\end{align*}
where $BC_{b,0}, BC_{b,1}$ are defined by
\begin{align}\label{def: BC b0}
	BC_{b,0}
	=&\nu^{-1/2}\Big\{ (u_{vp}^{(0)}+u_b^{(0)})(t,x,y)(\chi_b(y)\pa_x\omega_b^{(0)})(t,x,\frac{y}{\nu^{1/2}})\\
	\nonumber
	&+\frac{v_{vp}^{(0)}+v_b^{(0)}}{y}(t,x,y)y\pa_y\big(\chi_b(y)\omega_b^{(0)}(t,x,\frac{y}{\nu^{1/2}})\big) \Big\},
\end{align}
\begin{align}\label{def: BC b1}
	BC_{b,1}=\nu^{-1/2}\big(U_a\cdot\nabla\omega_{a,b}-BC_{b,0}\big).
	\end{align}
	
\begin{remark}
	We stress here that by Lemma \ref{lem: velocity formula} and transformation $y=\nu^{1/2}z$, we observe
\begin{align*}
	\big(\pa_y\Delta_D^{-1}BC_{b,0} \big)_\xi|_{y=0}
	=-\int_0^{+\infty}e^{-|\xi|y} (BC_{b,0})_\xi(t,y)dy
	\sim O(1).
\end{align*}
$BC_{b,1}$ can be analyzed similarly.
\end{remark}

We consider $U\cdot\nabla\omega$ near the point vortex subsequently.
\begin{align*}
	& U_a\cdot\nabla\omega_{a,vp} =
	\Big\{ U_b^{(0)}
	+\nu^{1/2}U_b^{(1)}
	+\frac{\alpha}{\sqrt{\nu t}}\big(\mathcal V^G(\frac{X-X(t)}{\sqrt{\nu t}})-\mathcal V^G(\frac{X-X(t)^\ast}{\sqrt{\nu t}}) \big)\\
	&\qquad+\sum_{k=2,3}\nu^{\frac{k-1}{2}}t^{\f12} \big(\mathcal V^{\Omega_k}(\frac{X-X(t)}{\sqrt{\nu t}},t)-\widetilde{\mathcal V^{\Omega_k}}(\frac{X-X(t)^\ast}{\sqrt{\nu t}},t) \big) \Big\}\\
	&\quad\cdot\nabla \Big\{ \frac{\alpha}{\nu t}G(\frac{X-X(t)}{\sqrt{\nu t}})\chi_{vp}
	+\sum_{k=2,3}(\nu t)^{\frac{k-2}{2}}\Omega_k(\frac{X-X(t)}{\sqrt{\nu t}})\chi_{vp} \Big\}\\
	&= BC_{vp,0}+\nu^{1/2}BC_{vp,1},
	\qquad for \quad y\geq10,
\end{align*}
where $BC_{vp,0}, BC_{vp,1}$ are defined by
\begin{align}\label{def: BC vp0}
	BC_{vp,0}
	=-\frac{\alpha}{\sqrt{\nu t}}\mathcal V^G(\frac{X-X(t)^\ast}{\sqrt{\nu t}})\cdot\nabla \Big\{\frac{\alpha}{\nu t}G(\frac{X-X(t)}{\sqrt{\nu t}})\chi_{vp} \Big\} ,
\end{align}
\begin{align}\label{def: BC vp1}
	BC_{vp,1}
	=\nu^{-1/2}\big(U_a\cdot\nabla\omega_{a,vp}-BC_{vp,0}\big).
\end{align}

\begin{remark}
	We stress here that by \eqref{BS law formulation 2} and integration by parts, we have
	\begin{align*}
	\pa_y\Delta_D^{-1}BC_{vp,0}|_{y=0}
	&\sim \frac{1}{\pi}\int_{\operatorname{supp}\chi_{vp}}
	\frac{\tilde y}{(x-\tilde x)^2+\tilde y^2}
	\frac{\alpha}{\sqrt{\nu t}}\mathcal V^G(\frac{\widetilde X-X(t)^\ast}{\sqrt{\nu t}})\cdot\nabla \Big\{\frac{\alpha}{\nu t}G(\frac{\widetilde X-X(t)}{\sqrt{\nu t}}) \Big\} d\tilde xd\tilde y\\
	&\sim -\frac{1}{\pi}\int_{\operatorname{supp}\chi_{vp}} \nabla\big(\frac{\tilde y}{(x-\tilde x)^2+\tilde y^2} \big)\frac{\alpha^2}{(\nu t)^{3/2}}\mathcal V^G(\frac{\widetilde X-X(t)^\ast}{\sqrt{\nu t}})
	G(\frac{\widetilde X-X(t)}{\sqrt{\nu t}}) 
	d\tilde xd\tilde y\\
	&\leq C\int_{\mathbb R^2}\frac{1}{(\nu t)^{1/2}}
	\left|\mathcal V^G(\eta+\frac{X(t)-X(t)^\ast}{\sqrt{\nu t}})\right|G(\eta)d\eta
	\leq C,
\end{align*}
where we used the transformation $\eta=\frac{\widetilde X-X(t)}{\sqrt{\nu t}}$ and $\mathcal V^G(\eta)\sim\min\{|\eta|,\frac{1}{|\eta|}\}$ in the last line.
$BC_{vp,1}$ can be analyzed similarly by noticing $U_b^{(0)}\sim\nu ^{1/2}$ away from the boundary.
\end{remark}

Summing up, we set 
\begin{align}\label{eq: BC of omega p0}
	\pa_z\omega_b^{(0)}|_{z=0}
	=\big(\pa_y\Delta_D^{-1}BC_{b,0} \big)|_{y=0}
	+\big(\pa_y\Delta_D^{-1}BC_{vp,0} \big)|_{y=0},
\end{align}
\begin{align}\label{eq: BC of omega p1}
	\pa_z\omega_b^{(1)}|_{z=0}
	=\big(\pa_y\Delta_D^{-1}BC_{b,1} \big)|_{y=0}
	+\big(\pa_y\Delta_D^{-1}BC_{vp,1} \big)|_{y=0}
	-\nu^{1/2}|D_x|\omega_b^{(0)}|_{z=0}.
\end{align}

\smallskip

\subsection{Proof of Proposition \ref{prop: est of remainder}}\label{Sec: Proof of Proposition est of remainder}

With $\Omega_i, \omega_b^{(i)}, U_b^{(i)}, U_{vp}^{(i)}, X_i(t)$ defined above, the ultimate approximate solutions are defined by
\begin{align*}
	\omega_a=\omega_{a,b}+\omega_{a,vp},
	\end{align*}
where
\begin{align*}
	\omega_{a,vp}(X,t):=\frac{\alpha}{\nu t}\mathcal W_a^\ast(\frac{X-X(t)}{\sqrt{\nu t}},t),\quad
	\omega_{a,b}(t,x,y):=\chi_b(y)\big(\nu^{-1/2}\omega_b^{(0)}+\omega_b^{(1)}\big)(t,x,\frac{y}{\nu^{1/2}}),
\end{align*}
with
\begin{align*} 
	\mathcal W_a^\ast=\chi_{vp}\mathcal W_a,\qquad
	\mathcal W_a:=\Omega_0+\nu t\Omega_2+\nu^{3/2}t\Omega_3.
\end{align*}

The vortex center $X(t)$ is defined as 
\begin{align*}
	X'(t)=-\frac{\alpha}{\sqrt{\nu t}}\mathcal V^G(\widetilde\eta)+\nu^{1/2}X_1'(t)+\nu X_2'(t),\qquad X(0)=X_0,
\end{align*}
with $X_1(t), X_2(t)$ defined before.  

We define the remainder ${\mathcal R}_{vp}$ as
\begin{align*}
	{\mathcal R}_{vp}:=L_{vp}(\mathcal W_a^\ast,\mathcal V^{\mathcal W_a^\ast},X(t),U_{a,b}).
\end{align*}
The difference between ${\mathcal R}_{vp,2}$ and ${\mathcal R}_{vp}$ are terms involving derivatives on $\chi_{vp}$, which can be estimated through the exponential decay at infinity of $\Omega_0\sim\Omega_3$. Thus, by Lemma \ref{lem: R^{(2)}_{vp}}, we deduce that the remainder ${\mathcal R}_{vp}$ satisfies for any $\lambda\in(0,1)$,
\begin{align}\label{est of Rvp ultimate}
	|{\mathcal R}_{vp}(t,\eta)|
	\leq C_\lambda \nu te^{-\lambda|\eta|^2/4}.
\end{align}

By the same argument, the remainder $R_b:=L_b(\omega_{a,b},U_{a,b},U_{a,vp})$ satisfies 
\begin{align}\label{est of Rb ultimate}
	\sum_{i=1}^8\left\| e^{C'|\xi|}\left\|e^{\frac{C'y^2}{\nu t}}(\pa_x^i(y\pa_y)^j(1,x)R_b)_\xi(t,y)\right\|_{L^1_y}\right\|_{L^1_\xi\cap L^2_\xi} \leq C\nu,
\end{align}
when $T$ small enough.

It is easy to prove  \eqref{est: X(t)-X0} by the definition of $X(t)$ and Lemma \ref{lem: ODE est}. 
By now, we finish the proof of the Proposition \ref{prop: est of remainder}.
\ef

\begin{remark}
	Here we stress that the definitions of $X(t)$ are different in Section \ref{sec: The first order expansion}, \ref{sec: The second order}. The definitions of $\omega_b^{(0)},\omega_b^{(1)}, U_b^{(0)}, U_b^{(1)}, U_{vp}^{(0)}, U_{vp}^{(1)}, \Omega_2,\Omega_3$ are all depend on preceding terms in the expansion of $X'(t)$ and are well defined therefore.
\end{remark}

%

\subsection{Estimates of  the approximate solution}

For later use, we enumerate several estimates for the approximate solutions below.
\begin{proposition}\label{prop: app estimates}
	There exist $C', T>0$ such that for $0< t\leq T$, $\lambda\in(0,1), 2<p\leq+\infty$, there holds that
	\begin{align}\label{est: app1}
		\sum_{i+j\leq8}\left\| e^{C'|\xi|}\left\|e^{\frac{C'y^2}{\nu t}}\big((1,x)\pa_x^i(y\pa_y)^j\omega_{a,b}\big)_\xi(t,y)\right\|_{L^1_y}\right\|_{L^1_\xi\cap L^2_\xi} \leq C,
	\end{align}
	\begin{align}\label{est: app2}
		\sum_{i\leq8}\left|\nabla_\eta^i\mathcal W_a(\eta,t)\right|
		\leq C_\lambda e^{-\lambda|\eta|^2/4},
		\qquad
		\int_{\mathbb R^2} \mathcal W_a(t,\eta) d\eta=1,
	\end{align}
	\begin{align}\label{est: app3}
		\sum_{i+j\leq8}\|\pa_x^i\pa_y^jU_{a,b}\|_{L^\infty(y\geq\f14)}\leq C\nu^{1/2},
		\qquad
		|U_{a,vp}(X,t)|
		\leq \frac{C}{|X-X(t)|}+O_{L^p}(e^{-\frac{C_1}{\nu t}}),
	\end{align}
	\begin{align}\label{est: app4}
		\sum_{i+j\leq8}\left\| e^{C'|\xi|}\sup_{0<y<5}\left|\big(\pa_x^i(y\pa_y)^j (u_a,\frac{v_a}{y})\big)_\xi(t,y)\right|\right\|_{L^1_\xi}\leq C,
	\end{align}
	where we use the notation $O_{L^p}(e^{-\frac{C_1}{\nu t}})$ to denote some function $f$ satisfying $\|f\|_{L^p}\leq Ce^{-\frac{C_1}{\nu t}}$.
\end{proposition}

The proof of Proposition \ref{prop: app estimates} is postponed in the Appendix \ref{sec: Estimates for approximate solutions}.

\section{Proof of zero-viscosity limit}

This section is devoted to the proof of  Theorem \ref{thm: main result vorticity convergence}.
\subsection{The remainder system}

Let
\begin{align}\label{def: omega R, U R}
	\omega_R=\omega-\omega_a,\qquad
	U_R=(u_R, v_R):=U-U_a=BS_{\mathbb R^2_+}[\omega_R].
\end{align}
where $\omega_a=\omega_{a,b}+\omega_{a,vp}$ is defined by \eqref{def:appp} and $U_a=BS_{\mathbb R^2_+}[\omega_a]$.
Therefore, a direct computation yields the system of $(\omega_R, U_R)$:
\begin{align}\label{eq: omega R, U R}
	\left\{
	\begin{aligned}
		&\pa_t\omega_R-\nu\Delta\omega_R+U_a\cdot\nabla\omega_R+U_R\cdot\nabla\omega_a+U_R\cdot\nabla\omega_R=-R_b-R_{vp},\\
		&\omega_R|_{t=0}=0,\qquad\qquad 
		\nu(\pa_y+|D_x|)\omega_R|_{y=0}=B_R,
	\end{aligned}
	\right.
\end{align}
where $R_b:=L_b(\omega_{a,b},U_{a,b},U_{a,vp})$ and 
\begin{align}\label{def: Rvp original v}
	R_{vp}:=\pa_t\omega_{a,vp}-\nu\Delta\omega_{a,vp}
	+(U_{a,b}+U_{a,vp})\cdot\nabla\omega_{a,vp}.
\end{align}
Here, the boundary condition $B_R$ is defined by 
\begin{align}\label{def: BR}
B_R:=\pa_y\Delta_D^{-1}\big(U_a\cdot\nabla\omega_R+U_R\cdot\nabla\omega_a+U_R\cdot\nabla\omega_R \big)|_{y=0}-\nu|D_x|\omega_b^{(1)}|_{z=0} .
\end{align}

Obviously, our construction implies
\begin{align}\label{support property of app and rem}
	\operatorname{supp}\omega_{a,b}, \operatorname{supp} R_b \subseteq \operatorname{supp} \chi_b,\qquad
	\operatorname{supp}\omega_{a,vp}, \operatorname{supp}R_{vp} \subseteq \operatorname{supp} \chi_{vp}.
\end{align}

\subsection{Energy functionals}
This subsection introduces the energy functionals used to prove Theorem \ref{thm: main result vorticity convergence}. We partition the half-space into three parts: near the point vortex, a middle region, and near the boundary. Correspondingly, we define three energy functionals, $E_{vp}$, $E_{m}$ and $E_{b}$.

\smallskip
 
\noindent\textbullet\ \ \textbf{Energy functional near the point vortex: $E_{vp}$}\smallskip

It is beneficial to introduce the following self-similar transformation to study the remainder near the point vortex:
\begin{align}\label{def: ssv rem}
	\chi_{vp}(x,y)\omega_R(t,x,y)
	=\frac{\alpha}{\nu t}\mathcal W_R(\eta,t),\qquad \eta:=\frac{(x,y)-X(t)}{(\nu t)^{1/2}},
\end{align}
where $X(t)$ is defined by \eqref{def: X(t) ultimate}. Due to \eqref{est: X(t)-X0} and $\operatorname{supp}\chi_{vp}$, it holds that
\begin{align}\label{supp of WR}
 	\operatorname{supp}\mathcal W_R\subseteq \{|\eta|\leq 7(\nu t)^{-1/2} \}.
 \end{align}
 
 In order to control vorticity near the point vortex, we need the following weight
 \begin{align}\label{def: weight p}
 	p(\eta,t)=
 	\left\{
 	\begin{aligned}
 		&e^{|\eta|^2/4},\qquad\qquad \uppercase\expandafter {\romannumeral1}  =\{ |\eta|\leq 2a_0(\nu t)^{-1/4} \},\\
 		&e^{a_0^2(\nu t)^{-1/2}},\qquad \uppercase\expandafter {\romannumeral2}  = \{ 2a_0(\nu t)^{-1/4}\leq |\eta|\leq 7(\nu t)^{-1/2} \} ,
 	\end{aligned}
 	\right.
 \end{align}
 here $a_0$ is sufficiently small and will be determined later. 
 
 The following energy functional is to describe vorticity near the point vortex
 \begin{align}\label{def: Evp(t)}
 	E_{vp}(t):=\sup_{0<s<t}\left\|p(\cdot,s)^{\f12}\mathcal W_R(\cdot,s)\right\|_{L^2}
 	+\big(\int_0^t\frac{D_{vp}(s)^2}{s}ds \big)^{1/2},
 \end{align}
 where 
 \begin{align*}
 	D_{vp}(t)=\left\|p(\cdot,t)^{\f12}\nabla_\eta\mathcal W_R(\cdot,t)\right\|_{L^2}.
 \end{align*}

\medskip

\noindent\textbullet\ \ \textbf{Energy functional in middle region: $E_m$}\smallskip

To define the energy functional on $\operatorname{supp}\chi_m$, we have to introduce the following functions.
  The smooth function $\theta(x,y)$ is defined as 
  \begin{align}\label{def: theta}
  	\theta(x,y)=
  	\left\{
  	\begin{aligned}
  		&1,\qquad\quad|(x,y)-(0,20)|\leq4 \ \text{or}\ y\leq\frac{3}{8},\\
  		&\frac{1}{4},\qquad\quad|(x,y)-(0,20)|\geq5 \ \text{and}\ y\geq\frac{1}{2},
  	\end{aligned}
  	\right.
  \end{align}
  and $\Gamma(t):=\{(x,y)\in\mathbb R^2_+: 1-\gamma t-\theta(x,y)\geq0\}, \gamma$ is a large quantity determined later.
  
  The following weight functions are of great use below
  \begin{align}\label{def: psi}
  	\psi(x,y)=y^2(1+|x|),\quad
  	\Psi_{t_0}(t,x,y)=\frac{20\eps_0}{\nu(t+t_0)}\big(1-\gamma t-\theta(x,y)\big)_+^2,\quad
  	\Psi:=\lim_{t_0\rightarrow0}\Psi_{t_0},
  \end{align}
  where $\eps_0\ll1$ is a constant determined later in the proof.
   Then, for $0\leq t\leq \frac{1}{4\gamma}, y\geq\frac{1}{2}$ and $|(x,y)-(0,20)|\geq5$, it holds that $\Psi_{t_0}(t,x,y)\geq\frac{5\eps_0}{t+t_0}$.
   
   The following energy functional is to describe the vorticity in the middle region
   \begin{align}\label{def: Em(t)}
   	E_m(t):=\sup_{0<s<t}\left\|e^\Psi\psi\chi_m\omega_R(s)\right\|_{L^2}
   	+\nu^{1/2}\big(\int_0^t D_m(s)^2ds\big)^{1/2}
   	+e^{\frac{10\eps_0}{\nu t}}\sup_{0< s< t}\|\omega_R(s)\|_{H^3(\frac{7}{8}\leq y\leq 4)} ,
   \end{align}
   where
   \begin{align*}
   	D_m(t):=\left\|e^\Psi\psi\nabla(\chi_m\omega_R)(t)\right\|_{L^2},
   \end{align*}
   and the last term in $E_m(t)$ captures the high order estimates within the middle region.

\smallskip

\noindent\textbullet\ \ \textbf{Energy functional near the boundary: $E_{b}$}\smallskip

 To control the vorticity near the boundary, we introduce the following norms
\begin{align*}
	\| f\|_{\mu,t}
	=\int_0^{1+\mu}e^{\eps_0(1+\mu)\frac{y^2}{\nu t}}|f(y)|dy,
	\qquad
	\|f\|_{Y^k_{\mu,t}}=\left\|\|e^{\eps_0(1+\mu-y)_+|\xi|}f_\xi\|_{\mu,t}\right\|_{L^k_\xi},\ k=1,2,
\end{align*}
\begin{align*}
	\|f\|_{Y_k(t)}
	=\sup_{\mu<\mu_0-\gamma t}
	\Big(\sum_{i+j\leq1}\|\pa_x^i(y\pa_y)^jf\|_{Y^k_{\mu,t}}
	+(\mu_0-\mu-\gamma t)^\beta\sum_{i+j=2}\|\pa_x^i(y\pa_y)^jf\|_{Y^k_{\mu,t}}\Big),\quad k=1,2.
\end{align*}
Here $0<\mu\leq\mu_0=\frac{1}{10}$. Throughout the paper, we fix $\beta\in(\frac{1}{2},1)$ and suppose $t\in(0,\frac{1}{4\gamma})$.

The following energy functional is to describe the vorticity near the boundary
\begin{align}\label{def: Eb(t)}
	E_b(t):=\sup_{0<s<t}\|(1,x)\omega_R(s)\|_{Y_1(s)\cap Y_2(s)}.
\end{align}

Now, we are in a position to define the energy functional:
\begin{align}\label{def: E(t)}
	E(t):=(\nu t)^{-1}E_{vp}(t)
	+E_m(t)
	+\nu^{-1} E_b(t),\quad
	D(t):=D_m(t)+D_{vp}(t). 
\end{align}

\subsection{Energy estimates}

For $E(t)$, we have the following estimates.

\begin{proposition}\label{prop: est of Evp(t)}
	There exist $T,C,C_1>0$ such that for $0< t\leq T$,
	\begin{align*}
		E_{vp}(t)
		\leq C\nu t\big(1+E(t)\big)^3.
	\end{align*}
\end{proposition}

\begin{proposition}\label{prop: est of Em(t)}
	There exist $T,C,C_1>0$ such that for $0< t\leq T$,
	\begin{align*}
		E_m(t)
		\leq Ct^{1/2}\big(1+E(t)\big)^3,
	\end{align*}
	and
	\begin{align*}
		e^{\frac{12\eps_0}{\nu t}}\sup_{[0,t]}\|\omega_R(s)\|_{H^3(\frac{7}{8}\leq y\leq 4)}
		\leq C\big(1+E(t)\big)^{8}.
	\end{align*}
\end{proposition}

\begin{proposition}\label{prop: est of Eb(t)}
There exist $T,C>0$ such that for $0< t\leq T$,
	\begin{align*}
		E_b(t)
		\leq C(\frac{\nu}{\gamma^{1/2}}+\nu t)\big(1+E(t)\big)^4.
	\end{align*}
For $\eps_0$ small, it holds that
	\begin{align*}
		\sup_{y>0}  e^{\frac{\eps_0 y^2}{\nu t}}\left\|e^{-|\xi|}(\chi_b\omega_R)_\xi(y,t)\right\|_{L^1_\xi\cap L^2_\xi}
		\leq& C(\nu t)^{1/2}\big(1+E(t)\big)^4.
	\end{align*}
\end{proposition}

Proposition \ref{prop: est of Evp(t)} $\sim$ Propositions \ref{prop: est of Eb(t)} will be proved in Section \ref{sec: est of Evp} $\sim$ Section \ref{sec: estimate of E_b(t)}.\smallskip
Armed with Proposition \ref{prop: est of Evp(t)} $\sim$ Propositions \ref{prop: est of Eb(t)}, we readily obtain
 
\begin{proposition}\label{prop: est of E(t) sum up}
	There exist $\nu_0,T,C>0$ such that for $0< t\leq T,0<\nu\leq\nu_0$,
	\begin{align*}
		E(t)\leq C.
	\end{align*}
\end{proposition}

\begin{proof}
	By Proposition \ref{prop: est of Evp(t)} $\sim$ Propositions \ref{prop: est of Eb(t)}, we have
	\begin{align*}
		E(t)&\leq C\nu t\big(1+E(t)\big)^3
		+Ct^{1/2}\big(1+E(t)\big)^3
		+Ce^{-\frac{2\eps_0}{\nu t}}\big(1+E(t)\big)^{8}
		+C(\frac{\nu}{\gamma^{1/2}}+\nu t)\big(1+E(t)\big)^4\\
		&\leq C(\frac{\nu}{\gamma^{1/2}}+t^{1/2})\big(1+E(t)\big)^8.
	\end{align*}
	 Taking $T,\nu_0$ small and $\gamma$ large, we derive the desired result by a bootstrap argument.
\end{proof}

\begin{remark}
	Strictly speaking, we should define the energy functional $E(t)$ as 
	\begin{align*}
		E(t):=\big(\nu (t+t_0) \big)^{-1}E_{vp}(t)
	+E_m(t)
	+\nu^{-1} E_b(t).
	\end{align*}
	Letting $t_0\rightarrow0^+$ and employing the same method in this paper, we can show that all constants $C$ remain uniformly bounded with respect to $t_0$, thereby eliminating the singularity at $t=0$. 
\end{remark}

\subsection{Proof of Theorem \ref{thm: main result vorticity convergence}}

Based on Proposition \ref{prop: est of E(t) sum up} and the definitions of $E_{vp}(t), E_m(t), E_b(t)$, it holds that
\begin{align*}
	&\left\|\omega(t,X)-\frac{\alpha}{4\pi \nu t}e^{-\frac{|X-X(t)|^2}{4\pi\nu t}} \right\|_{L^1_X(y\geq1)}\\
	&\leq \left\|\chi_{vp} \Big(\omega(t,X)-\frac{\alpha}{4\pi \nu t}e^{-\frac{|X-X(t)|^2}{4\pi\nu t}} \Big) \right\|_{L^1}
	+\left\|(1-\chi_{vp})\Big(\omega(t,X)-\frac{\alpha}{4\pi \nu t}e^{-\frac{|X-X(t)|^2}{4\pi\nu t}} \Big) \right\|_{L^1(y\geq1)}\\
	&\leq C\|\mathcal W-G\|_{L^1_\eta}
	+C\|(1-\chi_{vp})(\omega_a+\omega_R)\|_{L^2(y\geq1)}
	+\left\|(1-\chi_{vp})\cdot\frac{\alpha}{4\pi \nu t}e^{-\frac{|X-X(t)|^2}{4\pi\nu t}}  \right\|_{L^1(y\geq1)}\\
	&\leq C\|p^{1/2}\big(\nu t\Omega_2+\nu^{3/2}t\Omega_3\big)\|_{L^2_\eta}
	+C\|p^{1/2}\mathcal W_R\|_{L^2_\eta}
	+Ce^{-\frac{C_1}{\nu t}}\left\|e^\Psi\psi\chi_m(\omega_a+\omega_R)\right\|_{L^2}
	+Ce^{-\frac{C_1}{\nu t}}\\
	&\leq C\nu t+CE_{vp}(t)+Ce^{-\frac{C_1}{\nu t}} E_m(t)\\
	&\leq C\nu t,
\end{align*}
where the weight functions $p,\Psi,\psi$ are defined in \eqref{def: weight p}, \eqref{def: psi}.

For the second result, we set $\omega_b=\nu^{-1/2}\omega_b^{(0)}$. For any $1<p\leq+\infty$, we obatin
\begin{align*}
	&\left\|\|\omega-\nu^{-1/2}\omega_b^{(0)}\|_{L^p_x}\right\|_{L_y^1(y\leq2)}\\
	&\leq \left\|\|\omega_R\|_{L^p_x}\right\|_{L_y^1(y\leq3/4)}
	+\left\|\|\omega_R\|_{L^p_x}\right\|_{L_y^1(3/4 \leq y\leq2)}
	+\left\|\|\omega_b^{(1)}(t,x,\frac{y}{\nu^{1/2}})\|_{L^p_x}\right\|_{L^1_y}\\
	&\leq \left\|\|(1,x)\omega_R\|_{L^\infty_x}\right\|_{L^1_y(y\leq3/4)}
	+Ce^{-\frac{C_1}{\nu t}}\left\|e^\Psi\psi\chi_m\omega_R\right\|_{L^2}
	+\left\|\left\|(1,x)\omega_b^{(1)}(t,x,\frac{y}{\nu^{1/2}})\right\|_{L^\infty_x}\right\|_{L^1_y}\\
	&\leq \left\|\|\big((1,x)\omega_R\big)_\xi\|_{L^1_y(y\leq3/4)}\right\|_{L^1_\xi}
	+Ce^{-\frac{C_1}{\nu t}}E_m(t)
	+\nu^{1/2}\left\|\left\|\big((1,x)\omega_b^{(1)}(t,x,z)\big)_\xi\right\|_{L^1_z}\right\|_{L^1_\xi}\\
	&\leq E_b(t)+Ce^{-\frac{C_1}{\nu t}}
	+C\nu^{1/2}
	\leq C\nu^{1/2}.
\end{align*}

By now, we finish the proof of Theorem \ref{thm: main result vorticity convergence}.
\ef

\section{Velocity Estimates via Biot-Savart Law}\label{sec: Elliptic Estimates}

In this section, we provide velocity estimates which are frequently used in the proof of  Proposition \ref{prop: est of Evp(t)} $\sim$ Propositions \ref{prop: est of Eb(t)}. The estimates in this section are derived in a manner similar to those in Section 3 of \cite{WYZ}, except that here we emphasize the need for more refined estimates to capture the dependence of the bounds on viscosity.

To begin with, we give the Biot-Savart law under Fourier transformation which is proved in \cite{Maekawa}.

\begin{lemma}\label{lem: velocity formula}
	The velocity $U$ is recovered from $\omega$ through $U=BS_{\mathbb R^2_+}[\omega]$ or formulas below
	\begin{align*}
		&u_\xi(y)=\frac{1}{2}\Big(-\int_0^y e^{-|\xi|(y-z)}\big(1-e^{-2|\xi|z}\big)\omega_\xi(z)dz
		+\int_y^{+\infty}e^{-|\xi|(z-y)}\big(1+e^{-2|\xi|y}\big)\omega_\xi(z)dz\Big),\\
		&v_\xi(y)=-\frac{i\xi}{2|\xi|}\Big(\int_0^y e^{-|\xi|(y-z)}\big(1-e^{-2|\xi|z}\big)\omega_\xi(z)dz
		+\int_y^{+\infty}e^{-|\xi|(z-y)}\big(1-e^{-2|\xi|y}\big)\omega_\xi(z)dz\Big).
	\end{align*}
\end{lemma}

\medskip

The following elliptic estimates are used to handle transport terms near the boundary.
\begin{lemma}\label{lem: velocity estimates 1}
There exists a constant $C>0$ such that for $U_R=(u_R,v_R)$ near the boundary, it holds that
	\begin{align*}
		\left\|\sup_{0<y<1+\mu}e^{\eps_0(1+\mu-y)_+|\xi|}\left|\Big(u_R,\pa_x u_R,\frac{v_R}{y},y\pa_y u_R,y\pa_y(\frac{v_R}{y}) \Big)_\xi(s,y)\right|\right\|_{L^1_\xi}
		\leq C\nu E(s),
	\end{align*}
    and
	\begin{align*}
		\left\|\sup_{0<y<1+\mu}e^{\eps_0(1+\mu-y)_+|\xi|}\left| \big(\frac{\pa_x v_R}{y} \big)_\xi(s,y)\right|\right\|_{L^1_\xi}
		\leq C\nu(\mu_0-\mu-\gamma s)^{-\beta} E(s).
	\end{align*}

\end{lemma}

\begin{proof}
 We first focus on the case $u_R$. Lemma \ref{lem: velocity formula} implies
\begin{align*}
		(u_R)_\xi(s,y)=&-\frac{1}{2}\int_0^y e^{-|\xi|(y-z)}\big(1-e^{-2|\xi|z}\big)(\omega_R)_\xi(s,z)dz\\
		&+\frac{1}{2}\big(\int_y^{1+\mu}+\int_{1+\mu}^{+\infty}\big)e^{-|\xi|(z-y)}\big(1+e^{-2|\xi|y}\big)(\omega_R)_\xi(s,z)dz:=I_1+I_2+I_3.
	\end{align*}
	Based on the relation
	\begin{align}\label{weight transform 3}
		e^{\eps_0(1+\mu-y)_+|\xi|}e^{-|\xi||y-z|}
		\leq e^{\eps_0(1+\mu-z)_+|\xi|},
	\end{align}
	we have
	\begin{align*}
		e^{\eps_0(1+\mu-y)_+|\xi|}\big(|I_1|+|I_2|\big)
		\leq C_0\int_0^{1+\mu}e^{\eps_0(1+\mu-z)_+|\xi|}|(\omega_R)_\xi(s,z)|dz,
	\end{align*}
	which gives
	\begin{align*}
		&\left\|\sup_{0<y<1+\mu}e^{\eps_0(1+\mu-y)_+|\xi|}\big(|I_1|+|I_2|\big)\right\|_{L^1_\xi}
		\leq C_0\|\omega_R(s)\|_{Y^1_{\mu,s}}
		\leq C_0\nu E(s).
	\end{align*}
	
	A direct computation gives
	\begin{align*}
		e^{\eps_0(1+\mu-y)_+|\xi|}|I_3|
		\leq C_0\int_{1+\mu}^2 |(\omega_R)_\xi(s,z)|dz
		+C_0\int_2^{+\infty}e^{-|\xi|/2}|(\omega_R)_\xi(s,z)|dz,
	\end{align*}
	which implies
	\begin{align*}
		&\left\|\sup_{0<y<1+\mu}e^{\eps_0(1+\mu-y)_+|\xi|}|I_3|\right\|_{L^1_\xi}\\
		&\leq C_0\left\|\int_{1+\mu}^2 |(\omega_R)_\xi(s,z)|dz\right\|_{L^1_\xi}
		+C_0\left\|\int_2^{+\infty}e^{-|\xi|/2}|(\omega_R)_\xi(s,z)|dz\right\|_{L^1_\xi}\\
		&\leq C_0\int_1^2\left\|(1+|\xi|^2)^{-1/2}\right\|_{L^2_\xi}\left\|(1+|\xi|^2)^{1/2}(\omega_R)_\xi(s,z)\right\|_{L^2_\xi}dz
		+C_0\|\omega_R(s)\|_{L^1(z\geq2)}\\
		&\leq C_0\|\omega_R\|_{H^1(1\leq y\leq2)}
		+\|(1-\chi_{vp})\omega_R\|_{L^1}
		+C_0\|\chi_{vp}\omega_R\|_{L^1}\\
		&\leq C_0\|\omega_R\|_{H^1(1\leq y\leq2)}
		+C_0e^{-\frac{C_1}{\nu s}}\left\|e^\Psi\psi\chi_m\omega_R\right\|_{L^2}
		+C_0\|p^{1/2}\mathcal W_R\|_{L^2}\\
		&\leq C_0e^{-\frac{10\eps_0}{\nu s}}E(s)
		+C_0\nu sE(s)\\
		&\leq C_0\nu sE(s).
	\end{align*}
	
	Collecting the estimates together, we obtain the result for $u_R$. And $\pa_x u_R$ is treated similarly by replacing $\omega_R$ with $\pa_x\omega_R$.

    \medskip
	
	The estimates for $\frac{v_R}{y}$ and $\frac{\pa_x v_R}{y}$ is obtained by a similar procedure and the following
	\begin{align*}
		\frac{(v_R)_\xi(s,y)}{y}
		=& \frac{1}{2y}\int_0^y e^{-|\xi|(y-z)}\big(1-e^{-2|\xi|z}\big)(\omega_R)_\xi(s,z)dz\\
		&+\frac{1}{2y}\big(\int_y^{1+\mu}+\int_{1+\mu}^{+\infty}\big)e^{-|\xi|(z-y)}\big(1-e^{-2|\xi|y}\big)(\omega_R)_\xi(s,z)dz,\\
		\left|1-e^{-2|\xi|z}\right|&\leq 2|\xi|z\leq2|\xi|y,
		\quad
		\left|1-e^{-2|\xi|y}\right|\leq 2|\xi|y,
		\quad
		\text{for}
		\quad
		z\leq y.
	\end{align*}

    \medskip
	
	For $y\pa_y u_R$, a direct computation and Lemma \ref{lem: velocity formula} lead to
	\begin{align*}
		y\pa_y (u_R)_\xi(s,y)
		=&\frac{y}{2}\Big(\int_0^y e^{-|\xi|(y-z)}\big(1-e^{-2|\xi|z}\big)|\xi|(\omega_R)_\xi(s,z)dz\\
		&+\int_y^{+\infty}e^{-|\xi|(z-y)}\big(1+e^{-2|\xi|y}\big)|\xi|(\omega_R)_\xi(s,z)dz\\
		&-2\int_y^{+\infty}e^{-|\xi|(z-y)}e^{-2|\xi|y}|\xi|(\omega_R)_\xi(s,z)dz\Big)
		-y(\omega_R)_\xi(s,y).
	\end{align*}
	The first three terms are treated as above. For the last term, fundamental theorem of calculus gives rise to
	\begin{align*}
		\left\|\sup_{0<y<1+\mu}e^{\eps_0(1+\mu-y)_+|\xi|}|y(\omega_R)_\xi(s,y)|\right\|_{L^1_\xi}
		&\leq \|\omega_R(s)\|_{Y^1_{\mu,s}}
		+\|y\pa_y\omega_R(s)\|_{Y^1_{\mu,s}}
		+\|\pa_x\omega_R(s)\|_{Y^1_{\mu,s}}\\
		&\leq 3E_b(s)\leq C\nu E(s).
	\end{align*}
	Thus, we obtain the inequality for $y\pa_y u_R$. The case $y\pa_y\big(\frac{v_R}{y}\big)$ is derived from the relation
	\begin{align*}
		y\pa_y\big(\frac{v_R}{y}\big)
		=\pa_y v_R-\frac{v_R}{y}
		=-(\pa_x u_R)-\frac{v_R}{y}.
	\end{align*}
	\end{proof}


Next lemma provides velocity estimates in the middle region.
\begin{lemma}\label{lem: velocity estimates 2}
	(1) For $U_R$ on $\operatorname{supp}\chi_m$, we have
	\begin{align*}
		\|U_R(s)\|_{L^\infty(\chi_m)}
		\leq C\nu E(s)
		+Ce^{-\frac{C_1}{\nu s}}E(s)^{3/4}
		D(s)
		^{1/4}.
	\end{align*}

    (2) For $i,j\geq0$, $\frac{1}{2}\leq a<b\leq3$, we have
	\begin{align*}
		\left\|\pa_x^i\pa_y^j U_R(s)\right\|_{L^\infty(a\leq y\leq b)}
		\leq C\nu E(s)
		+C\|\omega_R(s)\|_{H^{i+1}(a-\frac{1}{100}\leq z\leq{b+\frac{1}{100}})}.
	\end{align*}
\end{lemma}

\begin{proof}
	(1) We need cut-off functions $\zeta_1, \zeta_2$ defined by 
	\begin{align}\label{def: zeta 12}
		\zeta_1=
		\left\{
		\begin{aligned}
			&1,\quad y\leq3/8,\\
			&0,\quad y\geq1/2,
		\end{aligned}
		\right.
		\qquad
		\zeta_2=
		\left\{
		\begin{aligned}
			&1,\quad |(x,y)-(0,20)|\leq2,\\
			&0,\quad |(x,y)-(0,20)|\geq3.
		\end{aligned}
		\right.
	\end{align}
	Obviously, it holds that
\begin{align*}
		\|U_R\|_{L^\infty(\chi_m)}
		\leq& \|BS_{\mathbb R^2_+}[\zeta_1\omega_R]\|_{L^\infty(\chi_m)}
		+\|BS_{\mathbb R^2_+}[(1-\chi_{vp})(1-\zeta_1)\omega_R]\|_{L^\infty}\\
		&+\|BS_{\mathbb R^2_+}[(1-\zeta_2)\chi_{vp}\omega_R]\|_{L^\infty}
		+\|BS_{\mathbb R^2_+}[\zeta_2\omega_R]\|_{L^\infty(\chi_m)}.
	\end{align*}
	
	Based on Lemma \ref{lem: velocity formula}, we find
	\begin{align*}
		\|BS_{\mathbb R^2_+}[\zeta_1\omega_R]\|_{L^\infty(\chi_m)}\leq C\|\omega_R\|_{Y_1(s)}\leq CE_b(s).
	\end{align*}
	
	Sobolev embedding implies for some $C_1>0$,
	\begin{align*}
		&\|BS_{\mathbb R^2_+}[(1-\chi_{vp})(1-\zeta_1)\omega_R]\|_{L^\infty}\\
		&\leq C\|(1-\chi_{vp})(1-\zeta_1)\omega_R\|_{L^{4/3}}^{1/2}\|(1-\chi_{vp})(1-\zeta_1)\omega_R\|_{L^4}^{1/2}\\
		&\leq Ce^{-\frac{C_1}{\nu s}}\left\|e^\Psi\psi\chi_m\omega_R\right\|_{L^2}^{3/4} \left\|e^\Psi\psi\nabla(\chi_m\omega_R)\right\|_{L^2}^{1/4}
		\leq Ce^{-\frac{C_1}{\nu s}} E_m(s)^{3/4}D_m(s)^{1/4},
	\end{align*}
	and
	\begin{align*}
		&\|BS_{\mathbb R^2_+}[(1-\zeta_2)\chi_{vp}\omega_R]\|_{L^\infty}
	   \leq C\|(1-\zeta_2)\chi_{vp}\omega_R\|_{L^{4/3}}^{1/2}\|(1-\zeta_2)\chi_{vp}\omega_R\|_{L^4}^{1/2}\\
	   &\leq C\left\|(1-\zeta_2)\frac{\alpha}{\nu s}\mathcal W_R(\frac{(x,y)-X(s)}{(\nu s)^{1/2}},s)\right\|_{L^{4/3}}^{1/2}
	   \left\|(1-\zeta_2)\frac{\alpha}{\nu s}\mathcal W_R(\frac{(x,y)-X(s)}{(\nu s)^{1/2}},s)\right\|_{L^4}^{1/2}\\
	   &\leq Ce^{-\frac{C_1}{\nu s}} \left\|p^{1/2}\mathcal W_R\right\|_{L^2}^{3/4}\left\|p^{1/2}\nabla_\eta\mathcal W_R\right\|_{L^2}^{1/4}
	   \leq Ce^{-\frac{C_1}{\nu s}} E_{vp}(s)^{3/4}D_{vp}(s)^{1/4}.
	\end{align*}
	
	Due to Lemma \ref{lem: BS est appendix}, we deduce
	\begin{align*}
		\|BS_{\mathbb R^2_+}[\zeta_2\omega_R]\|_{L^\infty(\chi_m)}
		\leq C\|\zeta_2\omega_R\|_{L^1}
		\leq C\|\mathcal W_R\|_{L^1}
		\leq C\left\|p^{1/2}\mathcal W_R\right\|_{L^2}
		\leq C E_{vp}(s).
	\end{align*}
	
	Collecting these estimates together gives 
	\begin{align*}
		\|U_R\|_{L^\infty(\chi_m)}
		\leq
		C\nu E(s)
		+Ce^{-\frac{C_1}{\nu s}}E(s)^{3/4}
		D(s)^{1/4}.
	\end{align*}

    (2)  We only deal with the case  $\pa_x^i u_R$, since other cases can be treated in a similar way. Lemma \ref{lem: velocity formula} gives
\begin{align*}
		|(\pa_x^i u_R)_\xi(s,y)|
		&\leq \int_0^{a-\frac{1}{100}}e^{-\frac{|\xi|}{100}}|\xi|^i|(\omega_R)_\xi(s,z)|dz
		+\int_{a-\frac{1}{100}}^{b+\frac{1}{100}} |\xi|^i|(\omega_R)_\xi(s,z)|dz\\
		&\qquad+\int_{b+\frac{1}{100}}^{+\infty}e^{-\frac{|\xi|}{100}}|\xi|^i|(\omega_R)_\xi(s,z)|dz\\
		&\leq C_0\int_0^{a-\frac{1}{100}}|(\omega_R)_\xi(s,z)|dz
		+C_0\|(\pa_x^i\omega_R)_\xi(s,z)\|_{L^2_z(a-\frac{1}{100}\leq z\leq{b+\frac{1}{100}})}\\
		&\qquad+C_0 \int_{b+\frac{1}{100}}^{+\infty}e^{-\frac{|\xi|}{200}}|(\omega_R)_\xi(s,z)|dz,
	\end{align*}
	which implies
\begin{align*}
		&\|\pa_x^i u_R(s)\|_{L^\infty(a\leq y\leq b)}
		\leq C_0\|\omega_R(s)\|_{Y_1(s)}
		+C_0\|\omega_R(s)\|_{L^1(z\geq3)} \\
		&\quad\qquad\qquad\qquad+C_0 \left\|(1+|\xi|^2)^{-\f12}\right\|_{L^2_\xi}\left\|\left\|(1+|\xi|^2)^{\f12}(\pa_x^i\omega_R)_\xi(s,z)\right\|_{L^2_z(a-\frac{1}{100}\leq z\leq{b+\frac{1}{100}})}\right\|_{L^2_\xi}\\
		&\leq C_0E_b(s)
		+C_0\|\omega_R(s)\|_{H^{i+1}(a-\frac{1}{100}\leq z\leq{b+\frac{1}{100}})}
		+C_0e^{-\frac{C_1}{\nu s}}\|e^\Psi \psi\chi_m\omega_R(s)\|_{L^2}
		+C_0\|\mathcal W_R\|_{L^1}\\
		&\leq C_0\nu E(s)
		+C_0\|\omega_R(s)\|_{H^{i+1}(a-\frac{1}{100}\leq z\leq{b+\frac{1}{100}})}.
	\end{align*}

\end{proof}


Next lemma provides estimates used to handle transport terms near the point vortex.

\begin{lemma}\label{lem: velocity estimates 3}
	(1) For $U_R$ on $\operatorname{supp}\chi_{vp}$, we have for $X\in\operatorname{supp}\chi_{vp}$, $\eta=\frac{X-X(s)}{\sqrt{\nu s}}$,
	\begin{align*}
		|U_R(s,X)|
		&\leq C\nu E(s)
		+Ce^{-\frac{C_1}{\nu s}}E_m(s)^{3/4} D_m(s)^{1/4}
		+\frac{1}{(\nu s)^{1/2}}\left|\mathcal V^{\mathcal W_R}(\frac{X-X(s)}{(\nu s)^{1/2}},s)\right|\\
		&\qquad
		+\frac{1}{(\nu s)^{1/2}}\left|\widetilde{\mathcal V^{\mathcal W_R}}(\frac{X-X(s)^\ast}{(\nu s)^{1/2}},s)\right|.
	\end{align*}
	
	(2) The following estimate will be used to describe the interaction between the point vortex and boundary effect
	\begin{align*}
		\left\|BS_{\mathbb R^2_+}[\chi_b\omega_R]\right\|_{L^\infty(\chi_{vp})}
		&\leq C(\nu t)^{3/2}\big(1+E(t)\big)^4.
	\end{align*}
\end{lemma}

\begin{proof}
	(1) Recalling $\zeta_1$ defined by \eqref{def: zeta 12}, we deduce that
	\begin{align*}
		&|U_R(s,X)|
		\leq \|BS_{\mathbb R^2_+}[\zeta_1\omega_R]\|_{L^\infty}
		+\|BS_{\mathbb R^2_+}[(1-\zeta_1-\chi_{vp})\omega_R]\|_{L^\infty}
		+|BS_{\mathbb R^2_+}[\chi_{vp}\omega_R](s,X)|\\
		&\leq C\|\omega_R(s)\|_{Y_1(s)}
		+Ce^{-\frac{C_1}{\nu s}}\|e^\Psi\psi\chi_m\omega_R\|_{L^2}^{3/4}\|e^\Psi\psi\nabla(\chi_m\omega_R)\|_{L^2}^{1/4}\\
		&\qquad+\frac{1}{(\nu s)^{1/2}}\left|\mathcal V^{\mathcal W_R}(\frac{X-X(s)}{(\nu s)^{1/2}},s)\right|
		+\frac{1}{(\nu s)^{1/2}}\left|\widetilde{\mathcal V^{\mathcal W_R}}(\frac{X-X(s)^\ast}{(\nu s)^{1/2}},s)\right| \\
		&\leq C\nu E(s)
		+Ce^{-\frac{C_1}{\nu s}}E_m(s)^{3/4} D_m(s)^{1/4}
		+\frac{1}{(\nu s)^{1/2}}\left|\mathcal V^{\mathcal W_R}(\frac{X-X(s)}{(\nu s)^{1/2}},s)\right|\\
		&\qquad
		+\frac{1}{(\nu s)^{1/2}}\left|\widetilde{\mathcal V^{\mathcal W_R}}(\frac{X-X(s)^\ast}{(\nu s)^{1/2}},s)\right|.
	\end{align*}
	
	(2) We denote $BS_{\mathbb R^2_+}[\chi_b\omega_R]=(u_1,u_2)$. By Lemma \ref{lem: velocity formula}, we have
	\begin{align*}
		&\left\|BS_{\mathbb R^2_+}[\chi_b\omega_R]\right\|_{L^\infty(\chi_{vp})}
		\leq \Big\|\sup_{10\leq y\leq30}(u_1)_\xi(y)\Big\|_{L^1_\xi}
		+\Big\|\sup_{10\leq y\leq30}(u_2)_\xi(y)\Big\|_{L^1_\xi}\\
		&\leq \left\| e^{-10|\xi|}\int_0^y \frac{1-e^{-2|\xi|z}}{2z}z\chi_b(z)|(\omega_R)_\xi(z)|dz \right\|_{L^1_\xi}
		\leq C\int_0^{+\infty}z \left\|e^{-|\xi|}(\chi_b\omega_R)_\xi(z)\right\|_{L^1_\xi}dz\\
		&\leq C\int_0^{+\infty}ze^{-\frac{\eps_0z^2}{\nu s}}dz
		\cdot\sup_{z>0}  e^{\frac{\eps_0 z^2}{\nu t}}\left\|e^{-|\xi|}(\chi_b\omega_R)_\xi(z,t)\right\|_{L^1_\xi\cap L^2_\xi}\\
		&\leq C(\nu t)^{3/2}\big(1+E(t)\big)^4,
	\end{align*}
	where we used the second inequality in Propositions \ref{prop: est of Eb(t)}.
\end{proof}

\section{Energy estimate near the point vortex}\label{sec: est of Evp}

This section is devoted to the proof of Proposition \ref{prop: est of Evp(t)}.

\subsection{Error system and basic estimates}

 Multiplying $\chi_{vp}$ on \eqref{eq: omega R, U R}, with \eqref{def: ssv rem} and definitions of approximate solutions in hand, we derive the system of $\mathcal W_R$ by a direct computation:
\begin{align}\label{eq: mathcal W R}
	\left\{
	\begin{aligned}
		&t\pa_t\mathcal W_R(\eta,t)-\mathcal L\mathcal W_R(\eta,t)
	+\frac{\alpha}{\nu}\left\{\mathcal V^{\mathcal W_a}\cdot\nabla_\eta\mathcal W_R
	+\mathcal V^{\mathcal W_R}\cdot\nabla_\eta\mathcal W_a \right\}(\eta,t)
	=\sum_{1\leq i\leq5}F_i,\\
	&\lim_{t\rightarrow0}\mathcal W_R=0,
	\end{aligned}
	\right.
\end{align}
where
\begin{align*}
	F_1=\frac{\alpha}{\nu}\left\{ \widetilde{\mathcal V^{\mathcal W_a}}(\eta+\widetilde\eta,t)
	-\mathcal V^G(\widetilde\eta,t) \right\}\cdot\nabla_\eta\mathcal W_R(\eta,t),
\end{align*}
\begin{align*}
	F_2=\Big( t^{1/2}X_1(t)+(\nu t)^{1/2}X_2(t)-\sqrt{\frac{t}{\nu}}U_{a,b} \Big)\cdot\nabla_\eta\mathcal W_R(\eta,t)
	-\chi_{vp}\mathcal R_{vp}(\eta,t) ,
\end{align*}
\begin{align*}
	F_3=\frac{\alpha}{\nu}\widetilde{\mathcal V^{\mathcal W_R}}(\eta+\frac{X(t)-X(t)^\ast}{\sqrt{\nu t}},t)\cdot\nabla_\eta\mathcal W_a(\eta,t),
\end{align*}
\begin{align*}
	F_4=-\sqrt{\frac{t}{\nu}}U_R\cdot\nabla_\eta\mathcal W_R(\eta,t)
	-\chi_{vp}\sqrt{\frac{t}{\nu}}
	BS_{\mathbb R^2_+}[(1-\chi_{vp})\omega_R]\cdot\nabla_\eta\mathcal W_a(\eta,t),
\end{align*}
and
\begin{align*}
	F_5=&-\frac{\alpha}{\nu}\left\{\mathcal V^{(1-\chi_{vp})\mathcal W_a}(\eta,t)
	-\widetilde{\mathcal V^{(1-\chi_{vp})\mathcal W_a}}(\eta+\widetilde\eta,t)\right\}\cdot\nabla_\eta\mathcal W_R(\eta,t)\\
	&+\frac{\alpha}{\nu}(1-\chi_{vp})\left\{\mathcal V^{\mathcal W_R}(\eta,t)
	-\widetilde{\mathcal V^{\mathcal W_R}}(\eta+\widetilde\eta,t)\right\}\cdot\nabla_\eta\mathcal W_a(\eta,t)\\
	&+t\chi_{vp}U_R\cdot\nabla\left\{(1-\chi_{vp})\mathcal W_a(\eta,t)\right\}
	+\frac{\nu t^2}{\alpha}U_R\cdot\nabla\chi_{vp}\omega_R
	-\frac{2(\nu t)^2}{\alpha}\nabla\chi_{vp}\cdot\nabla\omega_R\\
	&-\frac{(\nu t)^2}{\alpha}\Delta\chi_{vp}\omega_R
	+\frac{\nu t^2}{\alpha}(U_{a,b}+U_{a,vp})\cdot\nabla\chi_{vp}\omega_R.
\end{align*}

Before estimating $E_{vp}(t)$, we enumerate several lemmas which will be used later. The first lemma shows that the zero moment of $\mathcal W_R$ is sufficiently small in terms of $\nu t$.

\begin{lemma}\label{lem: est of zero moment}
 There exist a $T$ and  $C, C_1>0$ such that the following holds
 	\begin{align*}
 		\left|\int_{\mathbb R^2}\mathcal W_R(t,\eta)d\eta\right|
 		\leq C\big(1+ E(t)\big)^2 e^{-\frac{C_1}{\nu t}},\quad \textrm{for} \quad 0<t\leq T.
 	\end{align*}
\end{lemma} 
 
\begin{proof}
 	Let $f(t)=\int_{\mathbb R^2_+}\chi_{vp}\omega dxdy$. Then \eqref{def: ssv rem} implies
 	\begin{align*}
 		\alpha\int_{\mathbb R^2}\mathcal W_R(t,\eta)d\eta
 		&=\int_{\mathbb R^2_+}\chi_{vp}\omega_R dxdy
 		=\int_{\mathbb R^2_+}\chi_{vp}\omega dxdy-\int_{\mathbb R^2_+}\chi_{vp}\omega_{a,vp}dxdy\\
 		&=\int_{\mathbb R^2_+}\chi_{vp}\omega dxdy-\int_{\mathbb R^2}\chi_{vp}^2\cdot\frac{\alpha}{\nu t}\mathcal W_a(\frac{X-X(t)}{(\nu t)^{1/2}} ,t)dxdy\\
 		&=f(t)-f(0)+\int_{\mathbb R^2}(1-\chi_{vp}^2)\cdot\frac{\alpha}{\nu t}\mathcal W_a(\frac{X-X(t)}{(\nu t)^{1/2}} ,t)dxdy,
 	\end{align*}
 	here we use $\operatorname{supp}\omega_{a,b}\cap\operatorname{supp}\chi_{vp}=\emptyset$ and Proposition \ref{prop: app estimates}.
 	
 	We use \eqref{eq: NS vorticity}, Lemma \ref{lem: velocity estimates 2}, Proposition \ref{prop: app estimates} and integrate by parts to find for some $C>0$
 	\begin{align*}
 		|f'(t)|&=\left|\int_{\mathbb R^2_+}\chi_{vp}\big(\nu\Delta\omega-U\cdot\nabla\omega\big)dxdy\right|
 		=\left|\int_{\mathbb R^2_+}(\nu\Delta\chi_{vp}
 		+U\cdot\nabla\chi_{vp})\omega dxdy\right|\\
 		&\leq C\big(\nu+\|U\|_{L^4(\nabla\chi_{vp})}\big)\|\omega\|_{L^2(\nabla\chi_{vp})}\\
 		&\leq C\big(\nu+\|U_a\|_{L^4(\nabla\chi_{vp})}
 		+\|U_R\|_{L^\infty(\nabla\chi_{vp})}\big)
 		\big(\|\omega_{a}\|_{L^2(\nabla\chi_{vp})}
 		+\|\omega_{R}\|_{L^2(\nabla\chi_{vp})}\big)\\
 		&\leq C\big(1+\|U_R\|_{L^\infty(\nabla\chi_{vp})}\big)
 		\big(e^{-\frac{C_1}{\nu t}}+e^{-\frac{C_1}{\nu t}}\|e^{\Psi}\psi\chi_m\omega_{R}\|_{L^2}\big)\\
 		&\leq Ce^{-\frac{C_1}{\nu t}}\big(1+E(t)\big)\big(1+\nu E(t)+e^{-\frac{C'}{\nu t}}E(t)^{3/4}D(t)^{1/4}\big),
 	\end{align*}
 	which implies 
 	\begin{align*}
 		|f(t)-f(0)|&\leq  C\int_0^t e^{-\frac{C_1}{\nu s}}\big(1+E(s)\big)\big(1+\nu E(s)+e^{-\frac{C'}{\nu s}}E(s)^{3/4}D(s)^{1/4}\big)ds\\
 		&\leq Ce^{-\frac{C_1}{\nu t}}\big(1+E(t)\big)^2,
 	\end{align*}
 	where we recall the definition of $E(t)$ incorporates dissipative terms.

 	Due to $\operatorname{supp}\chi_{vp}$, it holds that
 	\begin{align*}
 		\left|\int_{\mathbb R^2}(1-\chi_{vp}^2)\cdot\frac{\alpha}{\nu t}\mathcal W_a(\frac{X-X(t)}{(\nu t)^{1/2}} ,t)dxdy\right|
 		\leq C e^{-\frac{C_1}{\nu t}}.
 	\end{align*}
 	
 	Collecting these estimates together, we obtain the desired result.
\end{proof}

Next lemma will be used to handle the dissipative term.
 \begin{lemma}\label{lem: dissi term vp}
 	Recall $\mathcal L=\Delta_\eta+\frac{\eta}{2}\cdot\nabla_\eta+1$. It holds that
 	\begin{align*}
 		\langle \mathcal W,\mathcal L\mathcal W\rangle_{\mathcal Y}
 		&=\|\mathcal W\|^2_{\mathcal Y}
 		-\|\nabla_\eta\mathcal W\|^2_{\mathcal Y}\\
 		&\leq -\frac{1}{12}\|\mathcal W\|^2_{\mathcal Y}
 		-\frac{1}{96}\|\eta\mathcal W\|^2_{\mathcal Y}
 		-\frac{1}{6}\|\nabla_\eta\mathcal W\|^2_{\mathcal Y}
 		+C\left|\int_{\mathbb R^2}\mathcal W d\eta\right|^2.
 	\end{align*}
 \end{lemma}
 
 \begin{proof}
 	The case $\int_{\mathbb R^2}\mathcal W d\eta=0$ has been proved in Lemma 5.1 of \cite{Gallay 3}. General case is obtained from replacing $\mathcal W$ by $\mathcal W-\int_{\mathbb R^2}\mathcal W d\eta\cdot G$.
 \end{proof}

Next lemma provides velocity estimates.
\begin{lemma}\label{lem: velocity VR pointwise}
	For $\mathcal V^{\mathcal W_R}=(\mathcal V_1, \mathcal V_2)=BS_{\mathbb R^2}[\mathcal W_R]$, it holds that
	\begin{align*}
		&(1+|\eta|^2)|\mathcal V^{\mathcal W_R}(\eta,t)|\\
		&\leq C\left\|p^{1/2}\mathcal W_R\right\|_{L^2}
		+C\left\|p^{1/2}\mathcal W_R\right\|_{L^2}^{3/4}\left\|p^{1/2}\nabla_\eta\mathcal W_R\right\|_{L^2}^{1/4}
		+C|\eta|\big(1+ E(t)\big)^2 e^{-\frac{C_1}{\nu t}}.
	\end{align*}
\end{lemma}

\begin{proof}
	For $|\eta|\leq1$, Sobolev embedding implies
	\begin{align*}
		\left|\mathcal V^{\mathcal W_R}(\eta,t)\right|
		&\leq C\|\mathcal W_R\|_{L^{4/3}}^{1/2}\|\mathcal W_R\|_{L^4}^{1/2}
		\leq C\left\|p^{1/2}\mathcal W_R\right\|^{3/4}_{L^2}\left\|p^{1/2}\nabla_\eta\mathcal W_R\right\|^{1/4}_{L^2}.
	\end{align*}
	
	For $|\eta|\geq1$, let $\eta=(\eta_1,\eta_2), \eta'=(\eta_1',\eta_2')$, a direct computation gives
	\begin{align*}
		\frac{\eta_1-\eta_1'}{|\eta-\eta'|^2}
		-\frac{\eta_1}{|\eta|^2}
		=\frac{1}{|\eta|^2|\eta-\eta'|^2}
		\big((\eta_1-\eta_1')(\eta\cdot\eta')
		+(\eta_2-\eta_2')(\eta^\perp\cdot\eta')\big),
	\end{align*}
	which along with \eqref{BS law formulation 1} implies
	\begin{align*}
		|\mathcal V_2(\eta,t)|
		&\leq \frac{1}{2\pi}\left|\int_{\mathbb R^2}(\frac{\eta_1-\eta_1'}{|\eta-\eta'|^2}
		-\frac{\eta_1}{|\eta|^2})\mathcal W_R(\widetilde\eta,t)\widetilde\eta\right|
		+\frac{1}{2\pi}\frac{1}{|\eta|}\left|\int_{\mathbb R^2}\mathcal W_R(\eta',t)d\eta'\right|\\
		&\leq C\int_{\mathbb R^2}\frac{1}{|\eta||\eta-\eta'|}|\eta'||\mathcal W_R(\eta',t)|d\eta'
		+\frac{C}{|\eta|}\big(1+ E(t)\big)^2 e^{-\frac{C_1}{\nu t}},
	\end{align*}
	where we used Lemma \ref{lem: est of zero moment}.
	
	Thus, by Hardy-Littlewood-Sobolev inequality, we have
	\begin{align*}
		&|\eta|^2|\mathcal V_2(\eta,t)|
		\leq C\int_{\mathbb R^2}\frac{|\eta-\eta'|+|\eta'|}{|\eta-\eta'|}|\eta'||\mathcal W_R(\eta',t)|d\eta'
		+C|\eta|\big(1+ E(t)\big)^2 e^{-\frac{C_1}{\nu t}}\\
		&\leq C\int_{\mathbb R^2}|\eta'||\mathcal W_R(\eta',t)|d\eta'
		+C\int_{\mathbb R^2}\frac{1}{|\eta-\eta'|}|\eta'|^2|\mathcal W_R(\eta',t)|d\eta'
		+C|\eta|\big(1+ E(t)\big)^2 e^{-\frac{C_1}{\nu t}}\\
		&\leq C\left\|p^{1/2}\mathcal W_R\right\|_{L^2}
		+C\left\|\eta^2\mathcal W_R(\eta,t)\right\|_{L^{4/3}}^{1/2}
		\left\|\eta^2\mathcal W_R(\eta,t)\right\|_{L^4}^{1/2}
		+C|\eta|\big(1+ E(t)\big)^2 e^{-\frac{C_1}{\nu t}}\\
		&\leq C\left\|p^{1/2}\mathcal W_R\right\|_{L^2}
		+C\left\|p^{1/2}\mathcal W_R\right\|_{L^2}^{3/4}\left\|p^{1/2}\nabla_\eta\mathcal W_R\right\|_{L^2}^{1/4}
		+C|\eta|\big(1+ E(t)\big)^2 e^{-\frac{C_1}{\nu t}}.
	\end{align*}
	$\mathcal V_1$ can be estimated similarly.
\end{proof}

By Lemma \ref{lem: velocity VR pointwise} and Lemma \ref{lem: velocity estimates 3}, we have
\begin{corollary}\label{cor: est of UR}
	For $U_R$ on $\operatorname{supp}\chi_{vp}$, we have for $X\in\operatorname{supp}\chi_{vp}$, $\eta=\frac{X-X(s)}{\sqrt{\nu s}}$,
	\begin{align*}
		|U_R(s,X)|
		\leq& C\nu\big(1+E(s)\big)^2
		+Ce^{-\frac{C_1}{\nu s}}E(s)^{3/4}D_m(s)^{1/4}
		\\
		&+C\big(1+|\eta|\big)^{-2}\big((\nu s)^{1/2}E(s)+(\nu s)^{1/4}E(s)^{3/4}D_{vp}(s)^{1/4}\big).
	\end{align*}
\end{corollary}

\smallskip

\subsection{Proof of Proposition \ref{prop: est of Evp(t)}}

Using \eqref{eq: mathcal W R}, we obtain
\begin{align*}
 	\f12 t\frac{d}{dt}\left\|p^{1/2}\mathcal W_R\right\|_{L^2}^2
 	=\int_{\mathbb R^2}\frac{t}{2}\pa_t p |\mathcal W_R|^2 d\eta
 	+\int_{\mathbb R^2}p\mathcal W_R t\pa_t\mathcal W_R d\eta
 	=\sum_{i=1}^{7}I_i,
 \end{align*}
 where
 \begin{align*}
 	I_1&=\int_{\mathbb R^2}\frac{t}{2}\pa_t p |\mathcal W_R|^2 d\eta
 	+\int_{\mathbb R^2}p\mathcal W_R \mathcal L\mathcal W_R d\eta,\\
 	I_2&=-\frac{\alpha}{\nu}\int_{\mathbb R^2}p\mathcal W_R \big(\mathcal V^{\mathcal W_a}\cdot\nabla_\eta\mathcal W_R
 	+\mathcal V^{\mathcal W_R}\cdot\nabla_\eta\mathcal W_a\big)d\eta,\\
	I_{i+2}&=\int_{\mathbb R^2}p\mathcal W_R F_i d\eta,\quad 1\leq i\leq5.
 \end{align*}
 Estimates for $I_1\sim I_7$ are given in Lemma \ref{lem: estimates for I1,I8} below. Armed with Lemma \ref{lem: estimates for I1,I8}, we take $t$ small to obtain
 \begin{align*}
 	t\frac{d}{dt}\left\|p^{1/2}\mathcal W_R\right\|_{L^2}^2
 	\leq -\frac{1}{100}\left\|p^{1/2}\nabla_\eta\mathcal W_R\right\|_{L^2}^2
 	+C(\nu t)^2\big(1+E(t)\big)^6
 	+Ce^{-\frac{C_1}{\nu t}}D(t)^2.
 \end{align*}
 We divide $t$ on both sides, integrate from $0$ to $t$ and take $t$ small to obtain Proposition \ref{prop: est of Evp(t)}.
 \ef

 To close the estimates, all we left is the estimates of $I_i$.
  \begin{lemma}\label{lem: estimates for I1,I8}
  	It holds that for some $C_1>0$
  	\begin{align*}
  	I_1
  	\leq -\frac{1}{48}\left\|p^{1/2}\mathcal W_R\right\|_{L^2}^2
 		-\frac{1}{6}\left\|p^{1/2}\nabla_\eta\mathcal W_R\right\|_{L^2}^2
 		-\frac{1}{96}\left\|p^{1/2}\eta\mathcal W_R\right\|_{L^2}^2
 		+C\big(1+ E(t)\big)^4e^{-\frac{C_1}{(\nu t)^{1/2}}},
  \end{align*}
  \begin{align*}
  	\sum_{i=2}^5|I_i|
  	&\leq \frac{1}{100}\left\|p^{1/2}\mathcal W_R\right\|_{L^2}^2
  	+Ct\left\|p^{1/2}\nabla_\eta\mathcal W_R\right\|_{L^2}^2
  	+Ct\left\|p^{1/2}\eta\mathcal W_R\right\|_{L^2}^2\\
  	&\qquad+C(\nu t)^2+C\big(1+E(t)\big)^4e^{-\frac{C_1}{\nu t}},
  \end{align*}
  \begin{align*}
  	|I_6|\leq  \frac{1}{1000}D_{vp}(t)^2
		+C\nu^2 t^3\big(1+E(t)\big)^6
		+Ce^{-\frac{C_1}{\nu t}}D(t)^2,
  \end{align*}
  and
  \begin{align*}
  	|I_7|\leq Ce^{-\frac{C_1}{\nu t}}
 	\Big(\left\|p^{1/2}\mathcal W_R\right\|_{L^2}^2
 	+
 	\left\|p^{1/2}\nabla_\eta\mathcal W_R\right\|_{L^2}^2
 	+\big(1+E(t)\big)^4
 	+D(t)^2 \Big).
  \end{align*}
  \end{lemma}
  
\begin{proof}
  Subsequently, we shall treat $I_1\sim I_7$ term by term.\smallskip
  
  \textbf{Estimate of $I_1$.}
   We first introduce two cut-off functions $\zeta_3, \zeta_4$ such that
  \begin{align*}
  	 \zeta_3|_{\{|\eta|\leq a_0(\nu t)^{-1/4} \}}=1, \quad \zeta_4|_{\{|\eta|\geq 2a_0(\nu t)^{-1/4} \}}=1,\quad \zeta_3^2+\zeta_4^2=1,\ |\nabla_\eta\zeta_i|\leq C(\nu t)^{1/4}.
  \end{align*}
  Denote $\mathcal W_1=\zeta_3\mathcal W_R$, $\mathcal W_2=\zeta_4\mathcal W_R$, thus,
  \begin{align*}
  	|\mathcal W_R|^2=|\mathcal W_1|^2+|\mathcal W_2|^2,\quad \ |\nabla_\eta\mathcal W_R|^2
  	=|\nabla_\eta \mathcal W_1|^2
  	+|\nabla_\eta \mathcal W_2|^2
  	-(|\nabla_\eta\zeta_3|^2+|\nabla_\eta\zeta_4|^2)|\mathcal W_R|^2.
  \end{align*}
  We use integration by parts and recall \eqref{def: weight p} to obtain
  \begin{align*}
  	I_1=&-\frac{1}{4}a_0^2(\nu t)^{-1/2}\int_{II}p|\mathcal W_R|^2d\eta
  	-\int_{\mathbb R^2}p|\nabla_\eta\mathcal W_R|^2d\eta
  	-\int_{\mathbb R^2}\mathcal W_R \nabla_\eta \mathcal W_R\cdot\nabla_\eta p d\eta\\
  	&-\frac{1}{4}\int_{\mathbb R^2}\dv(\eta p)|\mathcal W_R|^2 d\eta
  	+\int_{\mathbb R^2}p|\mathcal W_R|^2d\eta\\
  	\leq& 
  	-\sum_{i=1}^2\int_{\mathbb R^2}p|\nabla_\eta\mathcal W_i|^2d\eta
  	-\sum_{i=1}^2\int_{\mathbb R^2}\mathcal W_i \nabla_\eta \mathcal W_i\cdot\nabla_\eta p d\eta
    -\frac{1}{4}\sum_{i=1}^2\int_{\mathbb R^2}\dv(\eta p)|\mathcal W_i|^2 d\eta\\
  	&
  	+\sum_{i=1}^2\int_{\mathbb R^2}p|\mathcal W_i|^2d\eta
  	+\int_{\mathbb R^2}p(|\nabla_\eta\zeta_3|^2+|\nabla_\eta\zeta_4|^2)|\mathcal W_R|^2d\eta.
  \end{align*}
  
   For $\mathcal W_1$, integration by parts and Lemma \ref{lem: dissi term vp} imply
  \begin{align*}
  	&\quad-\int_{\mathbb R^2}p|\nabla_\eta\mathcal W_1|^2d\eta
  	-\int_{\mathbb R^2}\mathcal W_1 \nabla_\eta \mathcal W_1\cdot\nabla_\eta p d\eta
  	-\frac{1}{4}\int_{\mathbb R^2}\dv(\eta p)|\mathcal W_1|^2 d\eta
  	+\int_{\mathbb R^2}p|\mathcal W_1|^2d\eta\\
  	&=-\int_{\mathbb R^2}p|\nabla_\eta\mathcal W_1|^2d\eta
  	+\int_{\mathbb R^2}p|\mathcal W_1|^2d\eta\\
  	&\leq -\frac{1}{12}\int_{\mathbb R^2}p|\mathcal W_1|^2d\eta
 		-\frac{1}{96}\int_{\mathbb R^2}p|\eta|^2|\mathcal W_1|^2d\eta
 		-\frac{1}{6}\int_{\mathbb R^2}p|\nabla_\eta\mathcal W_1|^2d\eta
 		+C\left|\int_{\mathbb R^2}\mathcal W_1d\eta\right|^2 .
  \end{align*}
  
  For $\mathcal W_2$, Young inequality gives
  \begin{align*}
  	&-\int_{\mathbb R^2}p|\nabla_\eta\mathcal W_2|^2d\eta
  	-\int_{\mathbb R^2}\mathcal W_2 \nabla_\eta \mathcal W_2\cdot\nabla_\eta p d\eta
  	-\frac{1}{4}\int_{\mathbb R^2}\dv(\eta p)|\mathcal W_2|^2 d\eta
  	+\int_{\mathbb R^2}p|\mathcal W_2|^2d\eta\\
  	&\leq -\f14\int_{\mathbb R^2}p|\nabla_\eta\mathcal W_2|^2d\eta
  	+\f13 \int_{I}\frac{|\nabla p|^2}{p}|\mathcal W_2|^2d\eta
  	-\f18 \int_{\mathbb R^2}p|\eta|^2|\mathcal W_2|^2d\eta
  	+\f12 \int_{\mathbb R^2}p|\mathcal W_2|^2d\eta\\
  	&\leq -\f14\int_{\mathbb R^2}p|\nabla_\eta\mathcal W_2|^2d\eta
  	-\frac{1}{24}\int_{\mathbb R^2}p|\eta|^2|\mathcal W_2|^2d\eta
  	+\f12 \int_{\mathbb R^2}p|\mathcal W_2|^2d\eta\\
  	&\leq -\f14\int_{\mathbb R^2}p|\nabla_\eta\mathcal W_2|^2d\eta
  	-\frac{1}{36}\int_{\mathbb R^2}p|\eta|^2|\mathcal W_2|^2d\eta,
  \end{align*}
  here we used $\operatorname{supp}\mathcal W_2\subseteq \{|\eta|\geq a_0(\nu t)^{-1/4} \}$ in the last line.
  
   And by $|\nabla_\eta\zeta_i|\leq C(\nu t)^{1/4}$, we have
  \begin{align*}
  	\int_{\mathbb R^2}p(|\nabla_\eta\zeta_3|^2+|\nabla_\eta\zeta_4|^2)|\mathcal W_R|^2d\eta
  	\leq C(\nu t)^{1/2}\int_{\mathbb R^2}p|\mathcal W_R|^2d\eta.
  \end{align*}
  
  Lemma \ref{lem: est of zero moment} gives
  \begin{align*}
  	&\left|\int_{\mathbb R^2}\mathcal W_1 d\eta\right|^2
  	\leq \left|\int_{\mathbb R^2}\mathcal W_R d\eta\right|^2
  	+\left|\int_{\mathbb R^2}(1-\zeta_3)\mathcal W_R d\eta\right|^2\\
  	&\leq C\big(1+ E(t)\big)^4 e^{-\frac{C_1}{\nu t}}
  	+Ce^{-\frac{C_1}{(\nu t)^{1/2}}} \left\|p^{1/2}\mathcal W_R\right\|_{L^2}^2
  	\leq C\big(1+ E(t)\big)^4e^{-\frac{C_1}{(\nu t)^{1/2}}} .
  \end{align*}
  
  Collecting these estimates together and taking $\nu t$ small, we obtain the desired result.

  \textbf{Estimate of $I_2$.}
  We rewrite $\mathcal W_a$ as 
  \begin{align}\label{eq: def of mathcal F}
  	\mathcal W_a=G+\nu t\mathcal F,\qquad
  	\mathcal V^{\mathcal W_a}=\mathcal V^G+\nu t\mathcal V^{\mathcal F}.
  \end{align}
  where $\mathcal F:=\Omega_2+\nu^{1/2}\Omega_3$. Thus,
  \begin{align*}
  	I_2&=-\frac{\alpha}{\nu}\int_{\mathbb R^2}p\mathcal W_R\mathcal V^G\cdot\nabla_\eta\mathcal W_R d\eta
  	-\frac{\alpha}{\nu}
  	\int_{\mathbb R^2}p\mathcal W_R\mathcal V^{\mathcal W_R}\cdot\nabla_\eta G d\eta\\
  	&\quad-\alpha t\int_{\mathbb R^2}p\mathcal W_R\big(\mathcal V^{\mathcal F}\cdot\nabla_\eta\mathcal W_R
  	+\mathcal V^{\mathcal W_R}\cdot\nabla_\eta\mathcal F\big)d\eta:=I_{21}+I_{22}+I_{23}.
  \end{align*}
  
  For $I_{21}$, integration by parts implies
  \begin{align*}
  	I_{21}=\frac{\alpha}{2\nu}\int_{\mathbb R^2}\dv(p\mathcal V^G)|\mathcal W_R|^2d\eta
  	=\frac{\alpha}{2\nu}\int_{\mathbb R^2}\mathcal V^G\cdot\nabla_\eta p |\mathcal W_R|^2d\eta=0.
  \end{align*}
  
  For $I_{22}$, using the fact $\mathcal W\rightarrow \mathcal V^{\mathcal W}\cdot\nabla G$ is skew-adjoint in space $\mathcal Y$ which is proved in \cite{Gallay 2}, \cite{Maekawa 1}, we have
  \begin{align*}
  	|I_{22}|
  	&=\left|\frac{\alpha}{\nu}\int_{\mathbb R^2}(e^{|\eta|^2/4}-p)\mathcal W_R \mathcal V^{\mathcal W_R}\cdot\nabla_\eta Gd\eta\right|
  	\leq \frac{|\alpha|}{\nu}\int_{II}|\eta||\mathcal W_R|\left|\mathcal V^{\mathcal W_R}\right|d\eta\\
  	&\leq \frac{|\alpha|}{\nu}\|\eta\mathcal W_R\|_{L^{4/3}(II)}\left\|\mathcal V^{\mathcal W_R}\right\|_{L^4}
  	\leq Ce^{-\frac{C_1}{(\nu t)^{1/2}}}\left\|p^{1/2}\mathcal W_R\right\|^2_{L^2},
  \end{align*}
  where we used
  \begin{align}\label{est of V L4}
  	\|\mathcal V_R\|_{L^4}
    \leq C\|\mathcal W_R\|_{L^{4/3}}
    \leq C\left\|p^{1/2}\mathcal W_R\right\|_{L^2}.
  \end{align}
  
  For $I_{23}$, we utilize Lemma \ref{lem: decay rate of velocity} and \eqref{est of V L4}  to have
  \begin{align*}
  	|I_{23}|
  	&\leq |\alpha t|\int_{\mathbb R^2}p|\mathcal W_R|
  	\big(|\mathcal V^{\mathcal F}||\nabla_\eta\mathcal W_R|
  	+|\mathcal V^{\mathcal W_R}||\nabla_\eta \mathcal F|\big)d\eta\\
  	&\leq |\alpha t|\left\|p^{1/2}\mathcal W_R\right\|_{L^2} \big(\left\|p^{1/2}\nabla_\eta\mathcal W_R\right\|_{L^2}\|\mathcal V^{\mathcal F}\|_{L^\infty}
  	+\left\|p^{1/2}\nabla_\eta \mathcal F\right\|_{L^4}\|\mathcal V^{\mathcal W_R}\|_{L^4}\big)\\
  	&\leq Ct\left\|p^{1/2}\mathcal W_R\right\|_{L^2}^2
  	+Ct\left\|p^{1/2}\nabla_\eta\mathcal W_R\right\|_{L^2}^2.
  \end{align*}
  
  Collecting these estimates together yields the desired result.
  
  \textbf{Estimate of $I_3$.}
  Due to \eqref{eq: def of mathcal F}, we have
  \begin{align*}
  	I_3=&\frac{\alpha}{\nu}\int_{\mathbb R^2}p\mathcal W_R \Big(\mathcal V^G(\eta+\widetilde\eta)-\mathcal V^G(\widetilde\eta)\Big)\cdot\nabla_\eta\mathcal W_R d\eta\\
  	&+\alpha t\int_{\mathbb R^2}p\mathcal W_R \widetilde{\mathcal V^{\mathcal F}}(\eta+\widetilde\eta,t)\cdot\nabla_\eta\mathcal W_R d\eta=I_{31}+I_{32}.
  \end{align*}
  
  For $I_{31}$, integration by parts gives
  \begin{align*}
  	I_{31}&=-\frac{\alpha}{2\nu}\int_I \Big(\mathcal V^G(\eta+\widetilde\eta)-\mathcal V^G(\widetilde\eta)\Big)\cdot\nabla_\eta p|\mathcal W_R|^2d\eta\\
  	&=-\frac{\alpha}{\nu}\int_I \Big(\mathcal V^G(\eta+\widetilde\eta)-\mathcal V^G(\widetilde\eta)\Big)\cdot\eta p|\mathcal W_R|^2d\eta.
  \end{align*}
Thus, by Lemma \ref{lem: 1/nu of VG differ}, it holds that
\begin{align*}
	|I_{31}|\leq Ct\left\|p^{1/2}\langle\eta\rangle\mathcal W_R\right\|_{L^2(I)}^2
	\leq Ct\left\|p^{1/2}\eta\mathcal W_R\right\|_{L^2}^2
	+Ct\left\|p^{1/2}\mathcal W_R\right\|_{L^2}^2.
\end{align*}

For $I_{32}$, by Lemma \ref{lem: decay rate of velocity}, it holds that
\begin{align*}
	I_{32}
	\leq Ct\left\|p^{1/2}\mathcal W_R\right\|_{L^2}\left\|p^{1/2}\nabla_\eta\mathcal W_R\right\|_{L^2}\left\|\mathcal V^{\mathcal F}\right\|_{L^\infty}
	\leq Ct\left\|p^{1/2}\mathcal W_R\right\|_{L^2}^2
	+Ct\left\|p^{1/2}\nabla_\eta\mathcal W_R\right\|_{L^2}^2.
\end{align*}

Collecting these estimates together yields the desired result.

\smallskip

\textbf{Estimate of $I_4$.}
Due to Proposition \ref{prop: app estimates} and \eqref{est of Rvp ultimate}, we have
\begin{align*}
	|I_4|&\leq 
	C\int_{\mathbb R^2}p|\mathcal W_R|
	\big(t^{1/2}|\nabla_\eta\mathcal W_R|
	+\nu te^{-|\eta|^2/5}\big)d\eta\\
	&\leq \frac{1}{1000}\left\|p^{1/2}\mathcal W_R\right\|_{L^2}^2
	+Ct\left\|p^{1/2}\nabla_\eta\mathcal W_R\right\|_{L^2}^2
	+C(\nu t)^2.
\end{align*}

\smallskip

\textbf{Estimate of $I_5$.}
\eqref{est: X(t)-X0} implies $|X(t)-X(t)^\ast|\geq38$, which along with $ \operatorname{supp}\mathcal W_R\subseteq\{|\eta|\leq7(\nu t)^{-1/2}\}$ and Lemma \ref{lem: velocity VR pointwise} implies
\begin{align*}
  	\left|\widetilde{\mathcal V_R}(\eta+\widetilde\eta,t)\right|
  	\leq C\nu t \Big(\left\|p^{1/2}\mathcal W_R\right\|_{L^2}
  	+\left\|p^{1/2}\nabla_\eta\mathcal W_R\right\|_{L^2}\Big)
  	+\big(1+ E(t)\big)^2e^{-\frac{C_1}{\nu t}} .
  \end{align*}
  Thus,
  \begin{align*}
  	|I_5|
  	&\leq \frac{|\alpha|}{\nu}\int_{\mathbb R^2}p|\mathcal W_R|\left|\widetilde{\mathcal V^{\mathcal W_R}}(\eta+\widetilde\eta,t)\right||\nabla_\eta\mathcal W_a|d\eta\\
  	&\leq C\left\|p^{1/2}\mathcal W_R\right\|_{L^2}
  	\Big(t\left\|p^{1/2}\mathcal W_R\right\|_{L^2}
  	+t\left\|p^{1/2}\nabla_\eta\mathcal W_R\right\|_{L^2}
  	+\big(1+ E(t)\big)^2e^{-\frac{C_1}{\nu t}}\Big)\\
  	&\leq Ct\left\|p^{1/2}\mathcal W_R\right\|_{L^2}^2
  	+Ct\left\|p^{1/2}\nabla_\eta\mathcal W_R\right\|_{L^2}^2
  	+C\big(1+ E(t)\big)^4e^{-\frac{C_1}{\nu t}}.
  \end{align*}
  
  \smallskip
  
 \textbf{Estimate of $I_6$.}
 Due to $\operatorname{supp}\mathcal W_R\subseteq\operatorname{supp}\chi_{vp}$, Corollary \ref{cor: est of UR} implies
 \begin{align*}
 	&|I_6|
 	\leq C\sqrt{\frac{t}{\nu}}
 	\left\|p^{1/2}\mathcal W_R\right\|_{L^2}
 	\left\|p^{1/2}\nabla_\eta\mathcal W_R\right\|_{L^2}
 	\Big(\nu\big(1+E(t)\big)^2
 	+(\nu t)^{1/2}E(t)\\
 	&\quad\qquad+(\nu t)^{1/4}E(t)^{3/4}D_{vp}(t)^{1/4}
 	+e^{-\frac{C_1}{\nu t}}E(t)^{3/4}D_m(t)^{1/4} \Big)\\
 	&\qquad+C\sqrt{\frac{t}{\nu}}
 	\left\|p^{1/2}\mathcal W_R\right\|_{L^2}
 	\Big(\|BS_{\mathbb R^2_+}[(1-\chi_{vp}-\chi_b)\omega_R]\|_{L^4_\eta}\left\|p^{1/2}\nabla_\eta\mathcal W_a\right\|_{L^4}\\
 	&\quad\qquad+\|BS_{\mathbb R^2_+}[\chi_b\omega_R]\|_{L^\infty}
 	\left\|p^{1/2}\nabla_\eta\mathcal W_a\right\|_{L^2} \Big)\\
 	&\leq C\sqrt{\frac{t}{\nu}}
 	E_{vp}(t)D_{vp}(t)
 	\Big(\nu^{1/2}\big(1+E(t)\big)^2
 	+(\nu t)^{1/4}E(t)^{3/4}D_{vp}(t)^{1/4}
 	+e^{-\frac{C_1}{\nu t}}E(t)^{3/4}D_m(t)^{1/4} \Big)\\
 	&\qquad+C\sqrt{\frac{t}{\nu}}
 	E_{vp}(t)
 	\cdot\big(\nu t \big)^{3/2}\big(1+E(t)\big)^4\\
		&\leq \frac{1}{1000}D_{vp}(t)^2
		+C\nu^2 t^3\big(1+E(t)\big)^6
		+Ce^{-\frac{C_1}{\nu t}}D(t)^2.
 \end{align*}

 \textbf{Estimate of $I_7$.}
 Using the fact $\left\|\mathcal V^{(1-\chi_{vp})\mathcal W_a}\right\|_{L^\infty}\leq C\|(1-\chi_{vp})\mathcal W_a\|_{L^{4/3}\cap L^4}\leq Ce^{-\frac{C_1}{\nu t}}$ and Proposition \ref{prop: app estimates}, we have
 \begin{align*}
 	|F_5|\leq& Ce^{-\frac{C_1}{\nu t}}\Big(|\nabla_\eta\mathcal W_R(\eta,t)|+|\mathcal V^{\mathcal W_R}(\eta,t)|\mathbb I_{\operatorname{supp}(1-\chi_{vp})}
 	+\big|\widetilde{\mathcal V^{\mathcal W_R}}(\eta+\widetilde\eta,t)\big|\mathbb I_{\operatorname{supp}(1-\chi_{vp})}\\
 	&+|U_R(X,t)|\mathbb I_{\operatorname{supp}\nabla\chi_{vp}}\Big)
 	+C\nu t^2\big(|U_R\omega_R|
 	+|\omega_R|+|\nabla\omega_R|\big)(X,t)\mathbb I_{\operatorname{supp}\nabla\chi_{vp}},
 \end{align*}
 thus, by the definition of $p$, we have
 \begin{align*}
 	&|I_7|\leq \int_{\mathbb R^2}p|\mathcal W_R||F_5|d\eta\\
 	&\leq Ce^{-\frac{C_1}{\nu t}}
 	\left\|p^{1/2}\mathcal W_R\right\|_{L^2}
 	\Big(\left\|p^{1/2}\nabla_\eta\mathcal W_R\right\|_{L^2}
 	+(\nu t)^{-\f14}e^{\f12 a_0^2(\nu t)^{-1/2}}\|\mathcal V^{\mathcal W_R}\|_{L^4}
 	+(\nu t)^{-\f12}\|U_R\|_{L^4(\nabla\chi_{vp})} \Big)\\
 	&\qquad+C(\nu t)^{-\f12}e^{\f12 a_0^2(\nu t)^{-1/2}}
 	\left\|p^{1/2}\mathcal W_R\right\|_{L^2}
 	\Big(\|U_R\|_{L^4(\nabla\chi_{vp})}\|\omega_R\|_{L^4(\nabla\chi_{vp})}\\
 	&\qquad\qquad+\|\omega_R\|_{L^2(\nabla\chi_{vp})}
 	+\|\nabla\omega_R\|_{L^2(\nabla\chi_{vp})}\Big)\\
 	&\leq Ce^{-\frac{C_1}{2\nu t}}
 	\Big(\left\|p^{1/2}\mathcal W_R\right\|_{L^2}^2
 	+
 	\left\|p^{1/2}\nabla_\eta\mathcal W_R\right\|_{L^2}^2
 	+\left\|p^{1/2}\mathcal W_R\right\|_{L^2}
 	\big(\|e^\Psi\psi\chi_m\omega_R\|_{L^2} \\
 	&\qquad
 	+\nu E(t)\|e^\Psi\psi\chi_m\omega_R\|_{L^2}^{1/2}\|e^\Psi\psi\nabla(\chi_m\omega_R)\|_{L^2}^{1/2}
 	+\|e^\Psi\psi\nabla(\chi_m\omega_R)\|_{L^2}\big)\Big)\\
 	&\leq Ce^{-\frac{C_1}{2\nu t}}
 	\Big(\left\|p^{1/2}\mathcal W_R\right\|_{L^2}^2
 	+
 	\left\|p^{1/2}\nabla_\eta\mathcal W_R\right\|_{L^2}^2
 	+\big(1+E(t)\big)^4
 	+D_m(t)^2 \Big).
 \end{align*}

 By now, we finish the proof of Lemma \ref{lem: estimates for I1,I8}.
 \end{proof}

\section{Energy estimate in the middle region}\label{sec: est of Em}

This section is devoted to the proof of Proposition \ref{prop: est of Em(t)}.

\subsection{Weighted $L^2$ estimate}
This subsection is devoted to proving the first inequality in Proposition \ref{prop: est of Em(t)}. Recalling the cutoff function $\chi_m$ defined in \eqref{def: chi m} and multiplying both sides of \eqref{eq: omega R, U R} by $\chi_m$, we obtain
\begin{align}\label{eq: omega R chi m}
	\pa_t(\chi_m\omega_R)-\nu\Delta(\chi_m\omega_R)
	=F_1+F_2+F_3+F_4,
\end{align}
where
\begin{align*}
	&F_1=-U_a\cdot\nabla(\chi_m\omega_R),
	\qquad
	F_2=-\chi_mU_R\cdot\nabla\omega_a
	-\chi_m(R_b+R_{vp}), \\
	&F_3=-U_R\cdot\nabla(\chi_m\omega_R),
	\qquad F_4=(U_a+U_R)\cdot\nabla\chi_m\omega_R
	-\nu\Delta\chi_m\omega_R
	-2\nu\nabla\chi_m\cdot\nabla\omega_R,
\end{align*}

Taking $L^2$ inner product with $e^{2\Psi_{t_0}}\psi^2\chi_m\omega_R$ on both sides of \eqref{eq: omega R chi m}, we obtain
\begin{align*}
	&\frac{1}{2}\frac{d}{dt}\|e^{\Psi_{t_0} }\psi\chi_m\omega_R\|^2_{L^2}
			+\frac{20\eps_0\gamma}{\nu(t+t_0)}\int_{\Gamma(t)}e^{2\Psi_{t_0}}|\chi_m\psi\omega_R|^2\big(1-\gamma t-\theta(x,y)\big)_+ dxdy\\
			&\quad+\frac{20\eps_0}{\nu(t+t_0)^2}\|e^{\Psi_{t_0}} \psi\chi_m\omega_R\big(1-\gamma t-\theta(x,y)\big)_+\|_{L^2}^2
			-\nu\langle\Delta(\chi_m\omega_R),e^{2\Psi_{t_0}}\chi_m\psi^2\omega_R\rangle\\
			&=\sum_{i=1}^4\langle F_i,e^{2\Psi_{t_0}}\psi^2\chi_m\omega_R\rangle
			=\sum_{i=1}^4 I_i,
\end{align*}
where we utilize the fact
		\begin{align*}
			-\pa_t\Psi_{t_0}
			=\frac{40\eps_0\gamma}{\nu(t+t_0)}\big(1-\gamma t-\theta(x,y)\big)_+ \chi_{\Gamma(t)}
			+\frac{20\eps_0}{\nu(t+t_0)^2}\big(1-\gamma t-\theta(x,y)\big)_+^2.
		\end{align*}
		
\underline{Dissipative term.}
A direct computation gives
	\begin{align}\label{nable e^Psi}
			|\nabla(e^{2\Psi_{t_0}}\psi^2)|
			\leq \frac{C_0\eps_0}{\nu(t+t_0)}\big(1-\gamma t-\theta(x,y)\big)_+\cdot e^{2\Psi_{t_0}}\psi^2 \chi_{\Gamma(t)}
			+C_0 e^{2\Psi_{t_0}}\psi^2,
	\end{align}
	which implies
	\begin{align*}
			&-\nu\langle\Delta(\chi_m\omega_R),e^{2\Psi_{t_0}}\chi_m\psi^2\omega_R\rangle
			=\nu\|e^{\Psi_{t_0}}\psi\nabla(\chi_m\omega_R)\|^2_{L^2}
			+\nu\langle \nabla(\chi_m\omega_R), \nabla(e^{2\Psi_{t_0}}\psi^2)\chi_m\omega_R\rangle\\
			&\geq \nu\|e^{\Psi_{t_0}}\psi\nabla(\chi_m\omega_R)\|^2_{L^2}
			-C_0\nu\int_{\mathbb R^2_+}e^{2\Psi_{t_0}}\psi^2|\chi_m\omega_R||\nabla(\chi_m\omega_R)|dxdy \\
			&\quad-\frac{C_0\eps_0}{t+t_0}\int_{\Gamma(t)}e^{2\Psi_{t_0}}\psi^2\big(1-\gamma t-\theta(x,y)\big)_+|\chi_m\omega_R||\nabla(\chi_m\omega_R)|dxdy\\
			&\geq \frac{4\nu}{5}D_m(t)^2
			-C_0\nu\|e^{\Psi_{t_0} }\psi\chi_m\omega\|^2_{L^2}
			-\frac{C_0\eps_0^2}{\nu(t+t_0)^2}\|e^{\Psi_{t_0}} \psi\chi_m\omega_R\big(1-\gamma t-\theta(x,y)\big)_+\|_{L^2}^2.
		\end{align*}
		
\underline{Force terms.}
For $I_1$, we take integration by parts and use Proposition \ref{prop: app estimates}, \eqref{nable e^Psi} to obtain
\begin{align*}
	|I_1|&\leq 
	\int_{\mathbb R^2_+}|U_a||\nabla(e^{2\Psi_{t_0}}\psi^2)|(\chi_m\omega_R)^2 dxdy\\
	&\leq C\|e^{\Psi_{t_0} }\psi\chi_m\omega_R\|^2_{L^2}
	+\frac{C\eps_0}{\nu(t+t_0)}
	\int_{\Gamma(t)}e^{2\Psi_{t_0}}|\chi_m\psi\omega_R|^2\big(1-\gamma t-\theta(x,y)\big)_+ dxdy.
\end{align*}

For $I_2$, by Lemma \ref{lem: velocity estimates 2}, Proposition \ref{prop: app estimates} and Proposition \ref{prop: est of remainder}, we take $\eps_0$ small enough to deduce for some $C_1>0$,
\begin{align*}
	|I_2|&\leq \|e^{\Psi_{t_0} }\psi\chi_m\omega_R\|_{L^2}
	\big(\|U_R\|_{L^\infty(\chi_m)}\|e^{\Psi_{t_0}}\psi\chi_m\nabla\omega_a\|_{L^2}
	+\|e^{\Psi_{t_0} }\psi\chi_m(R_b+R_{vp})\|_{L^2} \big)\\
	&\leq Ce^{-\frac{C_1}{\nu t}}E(t)
	\big(\nu E(t)
		+e^{-\frac{C_1}{\nu t}}E(t)^{3/4}D(t)^{1/4}
		+1 \big)\\
		&\leq Ce^{-\frac{C_1}{\nu t}}\big(1+E(t)\big)^2
		+Ce^{-\frac{C_1}{\nu t}}D(t)^2.
\end{align*}

For $I_3$, by Lemma \ref{lem: velocity estimates 2}, we derive for some $C_1>0$,
\begin{align*}
	|I_3|&\leq \|U_R\|_{L^\infty(\chi_m)}
	\|e^{\Psi_{t_0} }\psi\nabla(\chi_m\omega_R)\|_{L^2}
	\|e^{\Psi_{t_0} }\psi\chi_m\omega_R\|_{L^2}\\
	&\leq C\big(\nu E(t)
		+e^{-\frac{C_1}{\nu t}}E(t)^{3/4}
		D(t)^{1/4} \big)E(t)D_m(t)\\
		&\leq \frac{\nu}{5}D_m(t)^2
		+C\big(1+E(t)\big)^6
		+Ce^{-\frac{C_1}{\nu t}}D(t)^2.
\end{align*}

For $I_4$, since $\Psi_{t_0}=0$ on $\operatorname{supp}\nabla\chi_m$, we apply Lemma \ref{lem: velocity estimates 2} and Proposition \ref{prop: app estimates} to obtain
\begin{align*}
	|I_4|&\leq 
	\big(\|U_a+U_R\|_{L^\infty(\chi_m)}\|\psi\omega_R\|_{L^2(\nabla\chi_m)}
	+\nu\|\psi(\omega_R,\nabla\omega_R)\|_{L^2(\nabla\chi_m)}\big)
	\|e^{\Psi_{t_0} }\psi\chi_m\omega_R\|_{L^2}\\
	&\leq Ce^{-\frac{C_1}{\nu t}}
	\Big\{ \big(\nu E(t)+E(t)^{3/4}D(t)^{1/4}+1\big)\big(E_b(t)+E_{vp}(t)\big)
	+D_{vp}(t) \Big\}E_m(t)\\
	&\leq Ce^{-\frac{C_1}{\nu t}}\big(1+E(t)\big)^3
	+Ce^{-\frac{C_1}{\nu t}}D(t)^2.
\end{align*}

Summing up these estimates, choosing $\gamma$ large enough and $t$ small, integrating from $0$ to $t$ and letting $t_0\rightarrow0$, we obtain the desired result.

\subsection{Higher-order estimates}\label{sec: High order estimates}

This subsection is devoted to proving the second inequality in Proposition \ref{prop: est of Em(t)} and to giving the high-order estimates of $\om_R$ in the interaction region. 

\begin{lemma}\label{lem: sobolev 1}
	There exists $T_0$ such that for $0\leq t\leq T_0,j=1,2,3$, 
	\begin{align*}
		&\sup_{[0,t]}\|\nabla^j\omega_R\|^2_{L^2(\frac{1}{2}+\frac{j}{8}\leq y\leq 5-\frac{j}{3})}
		+\nu\int_0^t\|\nabla^{1+j}\omega_R\|^2_{L^2(\frac{1}{2}+\frac{j}{8}\leq y\leq 5-\frac{j}{3})}ds
		\leq e^{-\frac{(40-5j)\eps_0}{\nu t}}\big(1+E(t)\big)^{4(1+j)}.
	\end{align*}
\end{lemma}

\begin{proof}
	We only prove the case $j=1$, since other cases are similar. We choose a smooth function $\eta_1(y)$ satisfying
	\begin{align*}
		\eta_1(y)=
		\left\{
		\begin{aligned}
			&1,\qquad \frac{1}{2}+\frac{1}{8}\leq y\leq 5-\frac{1}{3},\\
			&0,\qquad y\leq\frac{1}{2} \quad \text{or}\quad y\geq5.
		\end{aligned}
		\right.
	\end{align*}
	
	We apply $\pa_x$ on both sides of \eqref{eq: omega R, U R} and take $L^2$ inner product with $\eta_1^2\pa_x\omega$ and integrate over $0\leq s\leq t$ to have
	\begin{align*}
		&\frac{1}{2}\left\|\eta_1\pa_x\omega_R(t)\right\|_{L^2}^2
		+\nu\int_0^t \left\|\eta_1\pa_x\nabla\omega_R(t)\right\|_{L^2}^2ds
		=-2\nu\int_0^t\int_{\mathbb R^2_+}\eta_1'\eta_1\pa_x\omega_R\cdot\pa_x\pa_y\omega_R dxdy\\
		&\qquad\qquad+\int_0^t\int_{\mathbb R^2_+}\eta_1^2\big(U_a\cdot\nabla\omega_R+U_R\cdot\nabla\omega_a+R_b+R_{vp}\big)\cdot\pa_x^2\omega_R dxdy\\
		&\qquad\qquad-\int_0^t\int_{\mathbb R^2_+}\eta_1^2(U_R\cdot\nabla\omega_R)\cdot\pa_x^2\omega_R dxdy:=\sum_{1\leq i\leq3}I_i.
	\end{align*}
	
	For $I_1$, it holds that
	\begin{align*}
		|I_1|&\leq \frac{\nu}{2}\int_0^t \left\|\eta_1\pa_x\pa_y\omega_R(t)\right\|_{L^2}^2ds
		+Ce^{-\frac{40\eps_0}{\nu t}}\int_0^t\|e^\Psi\psi\nabla(\chi_m\omega_R)\|_{L^2}^2 ds\\
		&\leq \frac{\nu}{2}\int_0^t \left\|\eta_1\pa_x\pa_y\omega_R(t)\right\|_{L^2}^2ds
		+Ce^{-\frac{40\eps_0}{\nu t}}\int_0^t D(s)^2ds\\
		&\leq \frac{\nu}{2}\int_0^t \left\|\eta_1\pa_x\pa_y\omega_R(t)\right\|_{L^2}^2ds
		+Ce^{-\frac{35\eps_0}{\nu t}} E(t)^2.
	\end{align*}
	
	For $I_2$, we use Proposition \ref{prop: app estimates} and Lemma \ref{lem: velocity estimates 2} to obtain
	\begin{align*}
		|I_2|
		&\leq \int_0^t \big(\|U_a\|_{L^\infty(\frac{1}{2}\leq y\leq5)}\|\eta_1\nabla\omega_R\|_{L^2}
		+\|U_R\|_{L^\infty(\frac{1}{2}\leq y\leq5)}
		\|\eta_1\nabla\omega_a\|_{L^2}\\
		&\qquad\qquad+\|\eta_1(R_b+R_{vp})\|_{L^2} \big)
		\|\eta_1\pa_x^2\omega_R\|_{L^2}ds\\
		&\leq \frac{\nu}{5}\int_0^t
		\|\eta_1\pa_x^2\omega_R\|_{L^2}^2ds
		+Ce^{-\frac{40\eps_0}{\nu t}}
		\int_0^t \big(\left\|e^\Psi\psi\nabla(\chi_m\omega_R)\right\|_{L^2}^2
		+\|U_R\|^2_{L^\infty(\frac{1}{2}\leq y\leq5)}+1\big) ds\\
		&\leq \frac{\nu}{5}\int_0^t
		\|\eta_1\pa_x^2\omega_R\|_{L^2}^2ds
		+Ce^{-\frac{35\eps_0}{\nu t}}\big(1+E(t)\big)^2.
	\end{align*}
	
	For $I_3$, Lemma \ref{lem: velocity estimates 2} implies
	\begin{align*}
		|I_3|
		&\leq \int_0^t \|U_R\|_{L^\infty(\f12\leq y\leq5)}\|\eta_1\nabla\omega_R\|_{L^2}\left\|\eta_1\pa_x^2\omega_R\right\|_{L^2}ds\\
		&\leq \frac{\nu}{5}\int_0^t \|\eta_1\pa_x^2\omega_R\|^2_{L^2}ds
		+\frac{5}{\nu}\int_0^t\|U_R\|_{L^\infty(\f12\leq y\leq5)}^2\|\eta_1\nabla\omega_R\|^2_{L^2}ds\\
		&\leq \frac{\nu}{5}\int_0^t \|\eta_1\pa_x^2\omega_R\|^2_{L^2}ds
		+e^{-\frac{40\eps_0}{\nu t}}\sup_{[0,t]}\|\eta_1\nabla\omega_R\|_{L^2}
		\int_0^t \big(E(s)^2+E(s)^{3/2}D(s)^{1/2}\big)\left\|e^\Psi\psi\nabla(\chi_m\omega_R)\right\|_{L^2}ds\\
		&\leq \frac{\nu}{5}\int_0^t \|\eta_1\pa_x^2\omega_R\|^2_{L^2}ds
		+e^{-\frac{10\eps_0}{\nu t}}\sup_{[0,t]}\|\eta_1\nabla\omega_R\|_{L^2}^2
		+e^{-\frac{35\eps_0}{\nu t}}\big(1+E(t)\big)^8.
	\end{align*}
	Since $\eta_1\pa_y\omega_R$ can be estimated in a same way, we
	collect these estimates together to obtain the desired result.
	\end{proof}


\section{Energy  estimate near the boundary}\label{sec: estimate of E_b(t)}

This section is devoted to the proof of  Proposition \ref{prop: est of Eb(t)}.

\subsection{Solution formula near the boundary}

In this subsection, we derive the remainder system near the boundary.
Multiplying $\chi_b$ on both sides of \eqref{eq: omega R, U R} yields
\begin{align}\label{eq: chib omegaR}
	\left\{
	\begin{aligned}
		&\pa_t(\chi_b\omega_R)-\nu\Delta(\chi_b\omega_R)=N_R,\\
		&\chi_b\omega_R|_{t=0}=0,
		\qquad
		\nu(\pa_y+|D_x|)(\chi_b\omega_R)|_{y=0}=B_R,
	\end{aligned}
	\right.
\end{align}
where $B_R$ is defined in \eqref{def: BR} and
\begin{align}\label{def: NR}
	N_R=-\chi_b\big( U_a\cdot\nabla\omega_R
	+ U_R\cdot\nabla\omega_a
	+ U_R\cdot\nabla\omega_R\big)
	-\nu\chi_b''\omega_R
	-2\nu\chi_b'\pa_y\omega_R
	-\chi_b R_b.
\end{align}

We also need to derive corresponding estimates for  $x\omega_R$. Multiplying $x$ on both sides of \eqref{eq: chib omegaR} yields
\begin{align}\label{eq: chib x omegaR}
	\left\{
	\begin{aligned}
		&\pa_t(\chi_b x\omega_R)-\nu\Delta(\chi_b x\omega_R)=\widetilde{N_R} ,\\
		&\chi_b x\omega_R|_{t=0}=0,
		\qquad
		\nu(\pa_y+|D_x|)(\chi_b x\omega_R)|_{y=0}=\widetilde{B_R},
	\end{aligned}
	\right.
\end{align}
where 
\begin{align}\label{def: widetilde NR}
	\widetilde{N_R}
	=&-\chi_b\big( U_a\cdot\nabla(x\omega_R)
	+U_R\cdot\nabla(x\omega_a)
	+U_R\cdot\nabla(x\omega_R)\big)
	-\chi_b xR_b\\
	\nonumber
	&+\chi_b\big(u_a\omega_R+u_R\omega_a+u_R\omega_R\big)
	-\nu\chi_b'' x\omega_R
	-2\nu\chi_b'\pa_y(x\omega_R)
	-2\nu\chi_b\pa_x\omega_R.
\end{align}
And a direct computation implies
\begin{align}\label{def: widetilde BR}
	(\widetilde{B_R})_\xi
	&=\nu(\pa_y+|\xi|)(\chi_b x\omega_R)_\xi|_{y=0}
	=i\nu(\pa_y+|\xi|)\pa_\xi(\chi_b\omega_R)_\xi|_{y=0}\\
	\nonumber
	&=i\nu\pa_\xi\Big((\pa_y+|\xi|)(\chi_b\omega_R)_\xi\Big)|_{y=0}
	-i\nu sgn\xi(\omega_R)_\xi|_{y=0}\\
	\nonumber
	&=i\pa_\xi\big(B_R)_\xi
	-i\nu sgn\xi(\omega_R)_\xi|_{y=0},
\end{align}
where we denote by $f_\xi$ the Fourier transform of $f$ with respect to the horizontal variable $x$. 

By utilizing the solution formula in \cite{Maekawa}, we obtain the integral equations of $\omega_R$ and $x\omega_R$ as
\begin{align}\label{eq: integral eq of omega R}
	(\chi_b\omega_R)_\xi(t,y)
	=&\int_0^t\int_0^{+\infty}\big( H_\xi(t-s,y,z)
 	+R_\xi(t-s,y,z)\big) (N_R)_\xi(s,z)dzds\\
 	\nonumber
 	&-\int_0^t\big(H_\xi(t-s,y,0)+ R_\xi(t-s,y,0)\big) (B_R)_\xi(s)ds,
\end{align}
\begin{align}\label{eq: integral eq of widetilde omega R}
	(\chi_b x\omega_R)_\xi(t,y)
	=&\int_0^t\int_0^{+\infty}\big( H_\xi(t-s,y,z)
 	+R_\xi(t-s,y,z)\big) (\widetilde{N_R})_\xi(s,z)dzds\\
 	\nonumber
 	&-\int_0^t\big(H_\xi(t-s,y,0)+ R_\xi(t-s,y,0)\big) (\widetilde{B_R})_\xi(s)ds,
\end{align}
where
 \begin{align}
 	&H_\xi(t,y,z)=e^{-\nu\xi^2t}\big(g(\nu t,y-z)+g(\nu t,y+z)\big), \label{def of H xi} \\
	 &R_\xi(t,y,z)=\big( \Gamma (\nu t,x,y+z)-\Gamma(0,x,y+z)\big)_\xi, \label{def of R xi}
 \end{align}
 with 
\begin{align*}
	g(t,x)=\frac{1}{(4\pi t)^{1/2}}e^{-\frac{x^2}{4t}},\quad  \Gamma(t,x,y)=\big(\Xi E\ast G(t)\big)(x,y),
\end{align*}
 Here, 
 \begin{align*}
 	\Xi=2(\pa_x^2+|D_x|\pa_y\big)
   ,\
   E(x,y)=-\frac{1}{2\pi}\operatorname{log}\sqrt{x^2+y^2} ,\quad G(t,x,y)=g(t,x)g(t,y).
 \end{align*}
 
 In \cite{Kukavica}, \cite{TT Nguyen} and \cite{HWYZ},
 $R_\xi$ enjoys the following properties.
 \begin{lemma}\label{lem: property of Rxi}
 	(1) $\pa_y R_\xi(t,y,z)=\pa_z R_\xi(t,y,z).$
 	
 	(2) It holds that
 	\begin{align*}
 	&R_\xi(t,y,z)=-2\nu\int_0^t(-\xi^2+|\xi|\pa_y)\Big(e^{-\nu s\xi^2}g(\nu s,y+z)\Big)ds,\\
 		&|\pa_z^k R_\xi(t,y,z)|\leq C a^{k+1}e^{-\theta_0 a(y+z)}
 		+\frac{C}{(\nu t)^{(k+1)/2}}e^{-\theta_0\frac{(y+z)^2}{\nu t}}e^{-\frac{\nu\xi^2t}{8}},\quad k\geq0\quad a=|\xi|+\frac{1}{\sqrt{\nu}},\\
 		&|(y\pa_y)^kR_\xi(t,y,z)|\leq Cae^{-\frac{\theta_0}{2}a(y+z)}
 		+\frac{C}{\sqrt{\nu t}}e^{-\frac{\theta_0}{2}\frac{(y+z)^2}{\nu t}}e^{-\frac{\nu\xi^2t}{8}},\quad k=0,1,2,
 	\end{align*}
 	where $\theta_0$ is a universal constant and $C$ depends only on $\theta_0$.
 	
 \end{lemma}

\subsection{Semigroup estimates}

Our aim is to establish estimates of $\|(1,x)\omega_R\|_{Y_1(t)\cap Y_2(t)}$ by controlling the right hand side of \eqref{eq: integral eq of omega R}, \eqref{eq: integral eq of widetilde omega R} which is achieved via Lemma \ref{lem: est of HRN} and \ref{lem: est of HRB} below.

For later use, we denote
\begin{align}
	\mu_1:=\mu+\f12(\mu_0-\mu-\gamma s),
\end{align}
and observe that $0<\mu<\mu_1<\mu_0-\gamma s$. For convenience, we introduce the following norm.
\begin{align}\label{def: new norm of right}
	\|N(s)\|_{W_{\mu,s}}
	:=\sum_{i+j\leq1}\|\pa_x^i(y\pa_y)^jN(s)\|_{Y^1_{\mu,s}\cap Y^2_{\mu,s}}
	+e^{\frac{2\eps_0}{\nu s}}
	\sum_{i+j\leq2}\left\|\left\|\pa_x^i\pa_y^jN(s)\right\|_{L^2_x}\right\|_{L^1_y(y\geq1+\mu)}.
\end{align}

We have the following semigroup estimates.
\begin{lemma}\label{lem: est of HRN}
	For $\mu<\mu_0-\gamma t$, we have
	\begin{align*}
		&\sum_{i+j\leq1}\left\|\pa_x^i(y\pa_y)^j\int_0^t\int_0^{+\infty}\big( H(t-s,y,z)+R(t-s,y,z)\big)N_R(s,z)dzds\right\|_{Y^1_{\mu,t}\cap Y^2_{\mu,t}}\\
		&\leq C\int_0^t \|N_R(s)\|_{W_{\mu,s}} ds,
	\end{align*}
	and
	\begin{align*}
		&\sum_{i+j\leq2}\left\|\pa_x^i(y\pa_y)^j\int_0^t\int_0^{+\infty}\big( H(t-s,y,z)+R(t-s,y,z)\big)N_R(s,z)dzds\right\|_{Y^1_{\mu,t}\cap Y^2_{\mu,t}}\\
		&\leq C\int_0^t\big((\mu_0-\mu-\gamma s)^{-1}+(\mu_0-\mu-\gamma s)^{-\frac{1}{2}}(t-s)^{-\frac{1}{2}}\big) \|N_R(s)\|_{W_{\mu_1,s}} ds.
	\end{align*}
\end{lemma}

\begin{lemma}\label{lem: est of HRB}
	For $\mu<\mu_0-\gamma t$, we have 
	\begin{align*}
		&\sum_{i+j\leq1}\left\|\pa_x^i(y\pa_y)^j\int_0^t \big( H(t-s,y,0)+R(t-s,y,0)\big)B_R(s)ds\right\|_{Y^1_{\mu,t}\cap Y^2_{\mu,t}}\\
		&\leq C\int_0^t\sum_{i\leq1}\left\| e^{\eps_0(1+\mu)|\xi|}\xi^i (B_R)_\xi(s)\right\|_{L^1_\xi\cap L^2_\xi}ds,
	\end{align*}
	and
	\begin{align*}
		&\quad\sum_{i+j\leq2}\left\|\pa_x^i(y\pa_y)^j\int_0^t \big( H(t-s,y,0)+R(t-s,y,0)\big)B_R(s)ds\right\|_{Y^1_{\mu,t}\cap Y^2_{\mu,t}}\\
		&\leq C\int_0^t(\mu_0-\mu-\gamma s)^{-1}
		\sum_{i\leq1}\left\| e^{\eps_0(1+\mu_1)|\xi|}\xi^i (B_R)_\xi(s)\right\|_{L^1_\xi\cap L^2_\xi}ds.
	\end{align*}
\end{lemma}
We emphasize that Lemma \ref{lem: est of HRN} and \ref{lem: est of HRB} remain valid if $N_R, B_R$ are replaced by $\widetilde{N_R}, \widetilde{B_R}$, respectively. As these lemmas can be established via arguments analogous to those in \cite{WYZ}, we omit their proofs here.

\subsection{Estimates for $N_R, \widetilde{N_R}$}

We provide estimates for the force terms $N_R, \widetilde{N_R}$ in this subsection. The following proposition is the main result.
\begin{proposition}\label{prop: est of N in new norm}
	For $0<\mu<\mu_0-\gamma s$, it holds that
	\begin{align*}
		\left\|\big( N_R(s),\widetilde{ N_R}(s) \big)\right\|_{W_{\mu,s}}
		\leq C\nu(\mu_0-\mu-\gamma s)^{-\beta}\big(1+E(s)\big)^2.
	\end{align*}
\end{proposition}
The norm $W_{\mu,s}$ defined in \eqref{def: new norm of right} consists of two parts: near the boundary and away from the boundary. The estimate for the boundary part is provided in Lemma \ref{lem: est of N near the boundary}, and for the interior part in Lemma \ref{lem: est of N away from the boundary}. Proposition \ref{prop: est of N in new norm} is then proved by combining these two lemmas.

\begin{lemma}\label{lem: est of N near the boundary}
	For $0<\mu<\mu_0-\gamma s$, it holds that
	\begin{align*}
		\sum_{i+j\leq1}\|\pa_x^i(y\pa_y)^j\big( N_R(s),\widetilde{ N_R}(s) \big)\|_{Y^1_{\mu,s}\cap Y^2_{\mu,s}}
		\leq C\nu(\mu_0-\mu-\gamma s)^{-\beta}\big(1+E(s)\big)^2.
	\end{align*}
\end{lemma}

\begin{proof}
	Due to the definition of $\chi_b$, we have for $0<y<1+\mu$,
	\begin{align*}
		N_R&=U_a\cdot\nabla\omega_R
	+ U_R\cdot\nabla\omega_a
	+ U_R\cdot\nabla\omega_R-R_b,\\
	\widetilde{N_R}&=
	U_a\cdot\nabla(x\omega_R)
	+U_R\cdot\nabla(x\omega_a)
	+U_R\cdot\nabla(x\omega_R)-xR_b\\
	&\qquad+u_a\omega_R+u_R\omega_a+u_R\omega_R
	-2\nu\pa_x\omega_R.
	\end{align*}
	
	\underline{Case 1: $i=j=0$.} Lemma \ref{lem: product estimate}, Lemma \ref{lem: velocity estimates 1} and Proposition \ref{prop: app estimates} give rise to
	\begin{align}\label{est: NR without derivative}
		\|N_R(s)\|_{Y^1_{\mu,s}\cap Y^2_{\mu,s}}
		&\leq \left\|\sup_{0<y<1+\mu}e^{\eps_0(1+\mu-y)_+|\xi|}|(u_a)_\xi(s,y)|\right\|_{L^1_\xi}\|\pa_x\omega_R(s)\|_{Y^1_{\mu,s}\cap Y^2_{\mu,s}}\\
		\nonumber
		&\quad+\left\|\sup_{0<y<1+\mu}e^{\eps_0(1+\mu-y)_+|\xi|}|(\frac{v_a}{y})_\xi(s,y)|\right\|_{L^1_\xi}\|y\pa_y\omega_R(s)\|_{Y^1_{\mu,s}\cap Y^2_{\mu,s}}\\
		\nonumber
		&\quad+\left\|\sup_{0<y<1+\mu}e^{\eps_0(1+\mu-y)_+|\xi|}|(u_R)_\xi(s,y)|\right\|_{L^1_\xi}
		\|\pa_x(\omega_a,\omega_R)(s)\|_{Y^1_{\mu,s}\cap Y^2_{\mu,s}}\\
		\nonumber
		&\quad+\left\|\sup_{0<y<1+\mu}e^{\eps_0(1+\mu-y)_+|\xi|}|(\frac{v_R}{y})_\xi(s,y)|\right\|_{L^1_\xi}\|y\pa_y(\omega_a,\omega_R)(s)\|_{Y^1_{\mu,s}\cap Y^2_{\mu,s}}\\
		\nonumber
		&\quad+\|R_b(s)\|_{Y^1_{\mu,s}\cap Y^2_{\mu,s}} \\
		\nonumber
		&\leq CE_b(s)+C\nu E(s)\big(1+E_b(s)\big)
		+C\nu \\
		\nonumber
		&\leq C\nu\big(1+E(s)\big)^2.
	\end{align}
	
	\underline{Case 2: $i+j=1$.} Similarly, we use Lemma \ref{lem: product estimate}, Lemma \ref{lem: velocity estimates 1} and Proposition \ref{prop: app estimates} to have
	\begin{align*}
		&\|\pa_x N_R(s)\|_{Y^1_{\mu,s}\cap Y^2_{\mu,s}}
		+\|y\pa_y N_R(s)\|_{Y^1_{\mu,s}\cap Y^2_{\mu,s}}\\
		&\leq C\nu(\mu_0-\mu-\gamma s)^{-\beta}\big(1+E(s)\big)^2
		+C\nu^{-1}(\mu_0-\mu-\gamma s)^{-\beta}\|\omega_R(s)\|_{H^3(1\leq y\leq2)}^2\\
		&\leq C\nu(\mu_0-\mu-\gamma s)^{-\beta}\big(1+E(s)\big)^2.
	\end{align*}
	
	Estimates of $\widetilde{N_R}$ can be estimated similarly.
\end{proof}
 
Next lemma involves the interior part in $W_{\mu,s}$ defined in \eqref{def: new norm of right}.
 
\begin{lemma}\label{lem: est of N away from the boundary}
	For $0<\mu<\mu_0-\gamma s$, it holds that
	\begin{align*}
		e^{\frac{2\eps_0}{\nu s}}
	\sum_{i+j\leq2}\left\|\left\|\pa_x^i\pa_y^j\big( N_R(s),\widetilde{ N_R}(s) \big)\right\|_{L^2_x}\right\|_{L^1_y(y\geq1)}
	\leq C\nu \big(1+E(s)\big)^2.
	\end{align*}
\end{lemma}

\begin{proof}
	A direct computation gives 
	\begin{align*}
		&\sum_{i+j\leq2}\left\|\left\|\pa_x^i\pa_y^j\big( N_R(s),\widetilde{ N_R}(s) \big)\right\|_{L^2_x}\right\|_{L^1_y(y\geq1)}\\
		&\leq C_0\sum_{k=0}^2 \sum_{i+j\leq k}\|\pa_x^i\pa_y^j U_a(s)\|_{L^\infty(1\leq y\leq3)}
		\sum_{i+j\leq3-k}\left\|\left\|\pa_x^i\pa_y^j(1,x)\omega_R(s)\right\|_{L^2_x}\right\|_{L^1_y(1\leq y\leq3)}\\
		&\quad
		+C_0\sum_{k=0}^2 \sum_{i+j\leq k}\|\pa_x^i\pa_y^j U_R(s)\|_{L^\infty(1\leq y\leq3)}
		\sum_{i+j\leq3-k}\left\|\left\|\pa_x^i\pa_y^j(1,x)\omega_a(s)\right\|_{L^2_x}\right\|_{L^1_y(1\leq y\leq3)}\\
		&\quad
		+C_0\sum_{k=0}^2 \sum_{i+j\leq k}\|\pa_x^i\pa_y^j U_R(s)\|_{L^\infty(1\leq y\leq3)}
		\sum_{i+j\leq3-k}\left\|\left\|\pa_x^i\pa_y^j(1,x)\omega_R(s)\right\|_{L^2_x}\right\|_{L^1_y(1\leq y\leq3)}\\
		&\quad
		+\nu\sum_{i+j\leq2}\left\|\left\|\pa_x^i\pa_y^j(1,x)\omega_R(s)\right\|_{L^2_x}\right\|_{L^1_y(1\leq y\leq3)}
		+\sum_{i+j\leq3}\left\|\left\|\pa_x^i\pa_y^j(1,x)R_b(s)\right\|_{L^2_x}\right\|_{L^1_y(1\leq y\leq3)}.
	\end{align*}
	Using Lemma \ref{lem: velocity estimates 2}, we obtain for some $C_1>0$,
	\begin{align*}
		&\sum_{i+j\leq2}\left\|\left\|\pa_x^i\pa_y^j\big( N_R(s),\widetilde{ N_R}(s) \big)\right\|_{L^2_x}\right\|_{L^1_y(y\geq1)}
		\leq C\|(1,x)\omega_R(s)\|_{H^3(1\leq y\leq3)}
		+Ce^{-\frac{C_1}{\nu s}}\\
		&\qquad\qquad
		+C\big(e^{-\frac{C_1}{\nu s}}
		+\|(1,x)\omega_R(s)\|_{H^3(1\leq y\leq3)}\big)
		\big(\nu E(s)+\|(1,x)\omega_R(s)\|_{H^3(\frac{7}{8}\leq y\leq4)}\big)\\
		&\leq Ce^{-\frac{C_1}{\nu s}}\big(1+E(s)\big)
		+C\|(1,x)\omega_R(s)\|_{H^3(\frac{7}{8}\leq y\leq4)}\big(1+E(s)\big)
		+C\|(1,x)\omega_R(s)\|_{H^3(\frac{7}{8}\leq y\leq4)}^2,
	\end{align*}
	which implies for $\eps_0$ small enough,
	\begin{align*}
		&e^{\frac{2\eps_0}{\nu s}}
	\sum_{i+j\leq2}\left\|\left\|\pa_x^i\pa_y^j\big( N_R(s),\widetilde{N_R}(s) \big)\right\|_{L^2_x}\right\|_{L^1_y(y\geq1)}
	\leq Ce^{-\frac{C_1}{2\nu s}}\big(1+E(s)\big)\\
	&\qquad\qquad
	+Ce^{\frac{2\eps_0}{\nu s}}
	\|(1,x)\omega_R(s)\|_{H^3(\frac{7}{8}\leq y\leq4)}\big(1+E(s)\big)
	+Ce^{\frac{2\eps_0}{\nu s}}
	\|(1,x)\omega_R(s)\|_{H^3(\frac{7}{8}\leq y\leq4)}^2\\
	&\leq C\nu \big(1+E(s)\big)^2.
	\end{align*}
\end{proof}

\subsection{Estimates of $B_R, \widetilde{B_R}$}
We provide estimates of the boundary term below.
\begin{proposition}\label{prop: est of BR}
	For $0<\mu<\mu_0-\gamma s$, it holds that
	\begin{align*}
		&\sum_{i\leq1}\left\|e^{\eps_0(1+\mu)|\xi|}\xi^i
		\Big((B_R)_\xi(s),(\widetilde{B_R})_\xi(s) \Big)\right\|_{L^1_\xi \cap L^2_\xi}
		\leq C\nu(\mu_0-\mu-\gamma s)^{-\beta} \big(E(s)+1\big)^4\\
		&\qquad\qquad\qquad\qquad+C(\nu s)^{1/2}\big(1+E(s)\big)^{3/2}D_{vp}(s)^{1/2} +Ce^{-\frac{C_1}{\nu s}}E(s)^{3/2}D(s)^{1/2} .
	\end{align*}
\end{proposition}

The estimate for $B_R$ is provided in Lemma \ref{lem: est of BR} and for $\widetilde{B_R}$ is provided in Lemma \ref{lem: est of widetilde BR}.

Recalling \eqref{def: BR}, we decompose $B_R$ into three parts as follows to describe vorticity behavior in different regions.
\begin{align}\label{decomposition of BR}
	B_R:=\sum_{1\leq i\leq3} B_{R,i},
\end{align}
where
\begin{align*}
	&B_{R,1}=\pa_y\Delta_D^{-1}\big\{ 
	\chi_b\big(U_a\cdot\nabla\omega_R
	+U_R\cdot\nabla\omega_a
	+U_R\cdot\nabla\omega_R\big)\big\}|_{y=0}-\nu|D_x|\omega_b^{(1)} ,\\
	&B_{R,2}=
	\pa_y\Delta_D^{-1}\big\{(1-\chi_b-\chi_{vp})\big(U_a\cdot\nabla\omega_R
	+U_R\cdot\nabla\omega_a
	+U_R\cdot\nabla\omega_R\big)\big\}|_{y=0},\\
	&B_{R,3}=
	\pa_y\Delta_D^{-1}\big\{\chi_{vp}\big(U_a\cdot\nabla\omega_R
	+U_R\cdot\nabla\omega_a
	+U_R\cdot\nabla\omega_R\big)\big\}|_{y=0}.
\end{align*}

Next lemma provides estimate for $B_R$.
\begin{lemma}\label{lem: est of BR}
	For $0<\mu<\mu_0-\gamma s$, it holds that
	\begin{align*}
		&\sum_{i\leq1}\left\|e^{\eps_0(1+\mu)|\xi|}\xi^i
		(B_R)_\xi(s)\right\|_{L^1_\xi \cap L^2_\xi}
		\leq C\nu(\mu_0-\mu-\gamma s)^{-\beta}
		\big(E(s)+1\big)^4\\
		&\quad\qquad\qquad\qquad\qquad
		+C(\nu s)^{1/2}\big(1+E(s)\big)^{3/2}D_{vp}(s)^{1/2}
		+Ce^{-\frac{C_1}{\nu s}}E(s)^{3/2}D(s)^{1/2}.
	\end{align*}
\end{lemma}

\begin{proof}
We handle $B_{R,1}\sim B_{R,3}$ respectively.

\underline{Estimate of $B_{R,1}$.}
	Due to  Lemma \ref{lem: velocity formula}, we have
	\begin{align*}
		&|(B_{R,1})_\xi(s)|
		\leq \nu|\xi|\left|(\omega_b^{(1)})_\xi(s,0)\right|\\
		&\qquad+\big(\int_0^{1+\mu}+\int_{1+\mu}^{+\infty}\big) e^{-|\xi|y}\chi_b(y)
		\left|\big(U_a\cdot\nabla\omega_R
	+U_R\cdot\nabla\omega_a
	+U_R\cdot\nabla\omega_R\big)_\xi(s,y)\right|dy .
	\end{align*}
	Using the following fact
	\begin{align}\label{weight transform 3}
		e^{\eps_0(1+\mu)|\xi|}
		e^{-|\xi|y}
		\leq 
		\left\{
		\begin{aligned}
			&e^{\eps_0(1+\mu-y)_+|\xi|},\qquad for \quad 0<y<1+\mu,\\
			&e^{-|\xi|/2},\qquad for \quad y\geq1+\mu,
		\end{aligned}
		\right.
	\end{align}
	we have
	\begin{align*}
		&\sum_{i\leq1}\left\|e^{\eps_0(1+\mu)|\xi|}\xi^i
		(B_{R,1})_\xi(s)\right\|_{L^1_\xi \cap L^2_\xi}
		\leq I_1+I_2+I_3 \\
		&:= \sum_{i\leq1}\left\|\pa_x^i\big(U_a\cdot\nabla\omega_R
	+U_R\cdot\nabla\omega_a
	+U_R\cdot\nabla\omega_R\big)\right\|_{Y^1_{\mu,s}\cap Y^2_{\mu,s}}\\
	&\quad+\int_1^3 \left\|\langle\xi\rangle e^{-|\xi|/2}\right\|_{L^2_\xi \cap L^\infty_\xi}
	\left\|\big(U_a\cdot\nabla\omega_R
	+U_R\cdot\nabla\omega_a
	+U_R\cdot\nabla\omega_R\big)_\xi(s,y)\right\|_{L^2_\xi}dy\\
	&\quad+\nu\left\|e^{\eps_0(1+\mu)|\xi|}\langle\xi\rangle^2(\omega_b^{(1)})_\xi(s,0)\right\|_{L^1_\xi \cap L^2_\xi} .
	\end{align*}
	
	For $I_1$, using Lemma \ref{lem: product estimate} and Proposition \ref{prop: app estimates}, we have
	\begin{align*}
		I_1 &\leq 
		\sum_{i_1+i_2\leq1}\left\|\sup_{0<y<1+\mu}e^{\eps_0(1+\mu-y)_+|\xi|}\xi^{i_1}\left|(u_a,\frac{v_a}{y})_\xi(s,y)\right|\right\|_{L^1_\xi}
		\|\pa_x^{i_2}(\pa_x,y\pa_y)\omega_R\|_{Y^1_{\mu,s}\cap Y^2_{\mu,s}}\\
		&
		+\sum_{i_1+i_2\leq1}\left\|\sup_{0<y<1+\mu}e^{\eps_0(1+\mu-y)_+|\xi|}\xi^{i_1}\left|(u_R,\frac{v_R}{y})_\xi(s,y)\right|\right\|_{L^1_\xi}
		\|\pa_x^{i_2}(\pa_x,y\pa_y)(\omega_a,\omega_R)\|_{Y^1_{\mu,s}\cap Y^2_{\mu,s}}\\
		&\leq C(\mu_0-\mu-\gamma s)^{-\beta} E_b(s)
		+C\nu(\mu_0-\mu-\gamma s)^{-\beta} E(s)\big(E_b(s)+1\big)\\
		&\quad+C\nu E(s)\big((\mu_0-\mu-\gamma s)^{-\beta}E_b(s)+1\big)\\
		&\leq C\nu(\mu_0-\mu-\gamma s)^{-\beta} \big(E(s)+1\big)^2.
	\end{align*}

	For $I_2$, due to Proposition \ref{prop: app estimates} and Lemma \ref{lem: velocity estimates 2}, we have
	\begin{align*}
		I_2&\leq \left\|U_a\cdot\nabla\omega_R
	+U_R\cdot\nabla\omega_a
	+U_R\cdot\nabla\omega_R\right\|_{L^2(1\leq y\leq3)}\\
	&\leq \|U_a\|_{L^\infty(1\leq y\leq3)}\|\omega_R\|_{H^3(1\leq y\leq3)}
	+\|U_R\|_{L^\infty(1\leq y\leq3)}\|(\omega_a,\omega_R)\|_{H^3(1\leq y\leq3)}\\
	&\leq C\|\omega_R\|_{H^3(\frac{7}{8}\leq y\leq4)}
	+C\big(\nu E(s)+\|\omega_R\|_{H^3(\frac{7}{8}\leq y\leq4)}\big)
	\big(e^{-\frac{C_1}{\nu s}}
	+\|\omega_R\|_{H^3(\frac{7}{8}\leq y\leq4)}\big)\\
	&\leq C\nu\big(E(s)+1\big)^2.
	\end{align*}
	
	For $I_3$, Proposition \ref{prop: WP of boundary layer} and the remark under which yield $I_3\leq C\nu$.
	
	Collecting these estimates gives
	\begin{align*}
		\sum_{i\leq1}\left\|e^{\eps_0(1+\mu)|\xi|}\xi^i
		(B_{R,1})_\xi(s)\right\|_{L^1_\xi \cap L^2_\xi}
		\leq C\nu(\mu_0-\mu-\gamma s)^{-\beta}\big(E(s)+1\big)^2.
	\end{align*}
	
\underline{Estimate of $B_{R,2}$.}
Due to \eqref{BS law formulation 2} and integration by parts, we have
\begin{align*}
	(B_{R,2})_\xi(s)&=\f12\int_0^{+\infty}e^{-|\xi|y}\Big((1-\chi_b-\chi_{vp})\dv(U_a\omega_R+U_R\omega_a+U_R\omega_R ) \Big)_\xi(s,y) dy\\
	&=\f12\int_0^{+\infty}e^{-|\xi|y}
	\Big(\nabla(\chi_b+\chi_{vp})\cdot(U_a\omega_R+U_R\omega_a+U_R\omega_R ) \Big)_\xi(s,y)dy\\
	&\quad+\f12\int_0^{+\infty}(i\xi,|\xi|)e^{-|\xi|y}\Big((1-\chi_b-\chi_{vp})(U_a\omega_R+U_R\omega_a+U_R\omega_R ) \Big)_\xi(s,y) dy,
\end{align*}
which implies
\begin{align*}
	&\sum_{i\leq1}\left\|e^{\eps_0(1+\mu)|\xi|}\xi^i
		(B_{R,2})_\xi(s)\right\|_{L^1_\xi \cap L^2_\xi}
		\leq \int_2^{+\infty} \left\|\nabla(\chi_b+\chi_{vp})\cdot(U_a\omega_R+U_R\omega_a+U_R\omega_R )\right\|_{L^2_x} dy\\
		&\qquad+\int_2^{+\infty} \left\|(1-\chi_b-\chi_{vp})(U_a\omega_R+U_R\omega_a+U_R\omega_R )\right\|_{L^2_x} dy \\
		&\leq \|y(U_a\omega_R+U_R\omega_a+U_R\omega_R )\|_{L^2(1-\chi_b-\chi_{vp})}
		+\|U_a\omega_R+U_R\omega_a+U_R\omega_R \|_{L^2(\nabla(\chi_b+\chi_{vp}))}\\
		&\leq Ce^{-\frac{C_1}{\nu s}}
		\Big(\|U_a\|_{L^\infty(\chi_m)}\|e^\Psi\psi\chi_m\omega_R\|_{L^2}
		+\|U_R\|_{L^\infty(\chi_m)}\|e^\Psi\psi\chi_m(\omega_a,\omega_R)\|_{L^2} \Big)\\
		&\leq Ce^{-\frac{C_1}{\nu s}}
		\Big\{ E(s)+\big(1+E(s)\big)
		\big(\nu E(s)+E(s)^{3/4}D(s)^{1/4} \big) \Big\}\\
		&\leq Ce^{-\frac{C_1}{\nu s}}\big(1+E(s)\big)^2
		+Ce^{-\frac{C_1}{\nu s}}\big(1+E(s)\big)^{7/4}D(s)^{1/4}.
\end{align*}

\underline{Estimate of $B_{R,3}$.} 
We firstly rewrite $B_{R,3}$ as follows
\begin{align*}
	B_{R,3}=\pa_y\Delta_D^{-1}\big\{\chi_{vp}\curl\dv\big(U_a\otimes U_R+U_R\otimes U_a+U_R\otimes U_R\big)\big\}|_{y=0}.
\end{align*}

Thus, due to \eqref{BS law formulation 2} and integration by parts, it holds that
\begin{align*}
	|(B_{R,3})_\xi(s)|
	&\leq
	\left|\int_{\mathbb R^2_+}e^{2\pi ix'\xi-2\pi y'|\xi|}\chi_{vp}\curl\dv\big(U_a\otimes U_R+U_R\otimes U_a+U_R\otimes U_R\big)(x',y')dx'dy'\right|\\
	&\leq \int_{\chi_{vp}} \left|\nabla^2_{x',y'}(e^{2\pi ix'\xi-2\pi y'|\xi|}\chi_{vp})\right|
	\big(2|U_a||U_R|+|U_R|^2 \big)(x',y')dx'dy'\\
	&\leq e^{-10|\xi|}\int_{\chi_{vp}} \big(2|U_a||U_R|+|U_R|^2 \big)(x',y')dx'dy'.
\end{align*}

Armed with \eqref{est: app3} in Proposition \ref{prop: app estimates} and Corollary \ref{cor: est of UR}, we have
\begin{align*}
	&|(B_{R,3})_\xi(s)|
	\leq Ce^{-10|\xi|}\int_{\chi_{vp}} 
	\big(\frac{1}{|X-X(s)|}+e^{-\frac{C_1}{\nu s}}\big)
	\Big\{\nu\big(1+E(s)\big)^2
		+e^{-\frac{C_1}{\nu s}}E(s)^{3/4}D_m(s)^{1/4}\\
		&\quad +(\nu s)^{1/2}E(s)(1+|\eta|)^{-2}+(\nu s)^{1/4}E(s)^{3/4}D_{vp}(s)^{1/4}(1+|\eta|)^{-2} \Big\}
		+\Big\{\nu\big(1+E(s)\big)^2\\
		&\quad +e^{-\frac{C_1}{\nu s}}E(s)^{3/4}D_m(s)^{1/4}
		+(\nu s)^{1/2}E(s)(1+|\eta|)^{-2}+(\nu s)^{1/4}E(s)^{3/4}D_{vp}(s)^{1/4}(1+|\eta|)^{-2} \Big\}^2 dX\\
		&\leq Ce^{-10|\xi|} \big\{ \nu\big(1+E(s)\big)^4
		+(\nu s)^{3/4}E(s)^{3/4}D_{vp}(s)^{1/4}\\
		&\quad+(\nu s)^{3/2}E(s)^{3/2}D_{vp}(s)^{1/2}+e^{-\frac{C_1}{\nu s}}E(s)^{3/2}D_m(s)^{1/2} \big\},
\end{align*}
where we utilize the following fact
\begin{align*}
	\int_{\chi_{vp}}\frac{1}{|X-X(s)|}\cdot\frac{1}{(1+\frac{|X-X(s)|}{\sqrt{\nu s}})^2}dX \leq C(\nu s)^{1/2}.
\end{align*}
Thus, we have
\begin{align*}
	&\sum_{i\leq1}\left\|e^{\eps_0(1+\mu)|\xi|}\xi^i
		(B_{R,3})_\xi(s)\right\|_{L^1_\xi \cap L^2_\xi}\\
		&\leq \nu\big(1+E(s)\big)^4
		+(\nu s)^{1/2}E(s)^{3/2}D_{vp}(s)^{1/2}+e^{-\frac{C_1}{\nu s}}E(s)^{3/2}D_m(s)^{1/2}.
\end{align*}

Collecting these estimates together, we finish the proof.
\end{proof}

\begin{remark}\label{rem: est for BR i=0}
	In fact, we have proved 
	\begin{align}\label{est: BR i=0}
		&\left\|e^{\eps_0(1+\mu)|\xi|}
		(B_R)_\xi(s)\right\|_{L^1_\xi \cap L^2_\xi}
		\\
		\nonumber
		&\leq C\nu\big(E(s)+1\big)^4
		+C(\nu s)^{1/2}\big(1+E(s)\big)^{3/2}D_{vp}(s)^{1/2}
		+e^{-\frac{C_1}{\nu s}}E(s)^{3/2}D(s)^{1/2}.
	\end{align}
\end{remark}

Next lemma  provides estimate for $\widetilde{B_R}$.
\begin{lemma}\label{lem: est of widetilde BR}
	For $0<\mu<\mu_0-\gamma s$, it holds that
	\begin{align*}
		&\sum_{i\leq1}\left\|e^{\eps_0(1+\mu)|\xi|}\xi^i
		(\widetilde{B_R})_\xi(s)\right\|_{L^1_\xi \cap L^2_\xi}
		\leq C\nu(\mu_0-\mu-\gamma s)^{-\beta} \big(E(s)+1\big)^4\\
		&\quad+C(\nu s)^{1/2}\big(1+E(s)\big)^{3/2}D_{vp}(s)^{1/2} +Ce^{-\frac{C_1}{\nu s}}E(s)^{3/2}D(s)^{1/2} .
	\end{align*}
\end{lemma}

\begin{proof}
	Recalling the relation \eqref{def: widetilde BR}: $(\widetilde{B_R})_\xi=i\pa_\xi(B_R)_\xi-i\nu sgn\xi(\omega_R)_\xi|_{y=0}$, we have
	\begin{align*}
		&\sum_{i\leq1}\left\|e^{\eps_0(1+\mu)|\xi|}\xi^i
		(\widetilde{B_R})_\xi(s)\right\|_{L^1_\xi \cap L^2_\xi}\\
		&\leq \sum_{i\leq1}\left\|e^{\eps_0(1+\mu)|\xi|}\xi^i
		\pa_\xi({B_R})_\xi(s)\right\|_{L^1_\xi \cap L^2_\xi}
		+\nu\sum_{i\leq1}\left\|e^{\eps_0(1+\mu)|\xi|}\xi^i(\omega_R)_\xi|_{y=0}\right\|_{L^1_\xi \cap L^2_\xi}.
	\end{align*}
	The first term can be estimated by taking $\pa_\xi$ on $B_{R,1}\sim B_{R,3}$ in Lemma \ref{lem: est of BR} and by the relation $i\pa_\xi f_\xi=(xf)_\xi$. And we have
	\begin{align*}
		&\sum_{i\leq1}\left\|e^{\eps_0(1+\mu)|\xi|}\xi^i
		\pa_\xi({B_R})_\xi(s)\right\|_{L^1_\xi \cap L^2_\xi}
		\leq C\nu(\mu_0-\mu-\gamma s)^{-\beta}\big(E(s)+1\big)^4\\
		&\quad\qquad\qquad\qquad\qquad
		+C(\nu s)^{1/2}\big(1+E(s)\big)^{3/2}D_{vp}(s)^{1/2}
		+Ce^{-\frac{C_1}{\nu s}}E(s)^{3/2}D(s)^{1/2}.
	\end{align*}
	
	Taking $y=0$ in \eqref{eq: integral eq of omega R} gives
	\begin{align*}
		(\omega_R)_\xi|_{y=0}(s)
		=&\int_0^s\int_0^{+\infty}\big(H_\xi(s-\tau,0,z)
		+R_\xi(s-\tau,0,z)\big)(N_R)_\xi(\tau,z)dzd\tau\\
		&-\int_0^s \big(H_\xi(s-\tau,0,0)
		+R_\xi(s-\tau,0,0)\big)(B_R)_\xi(\tau)d\tau:=J_1+J_2.
	\end{align*}
	
	For $J_1$, by \eqref{def of H xi}, \eqref{def of R xi}, Lemma \ref{lem: property of Rxi} and a direct computation, we have
	\begin{align*}
		e^{\eps_0(1+\mu)|\xi|}|\xi|^i
		|H_\xi(s-\tau,0,z)|
		\leq 
		\left\{
		\begin{aligned}
			&\frac{C|\xi|^i}{\nu^{1/2}(s-\tau)^{1/2}}e^{\eps_0(1+\mu-z)_+|\xi|},\qquad z<1+\mu,\\
			&C,\qquad z\geq1+\mu,
		\end{aligned}
		\right.
	\end{align*}
	and
	\begin{align*}
		e^{\eps_0(1+\mu)|\xi|}|\xi|^i
		|R_\xi(s-\tau,0,z)|
		\leq 
		\left\{
		\begin{aligned}
			&C\big(\frac{|\xi|^i}{\nu^{1/2}(s-\tau)^{1/2}}+|\xi|^{1+i}\big)e^{\eps_0(1+\mu-z)_+|\xi|},\qquad z<1+\mu,\\
			&C,\qquad z\geq1+\mu.
		\end{aligned}
		\right.
	\end{align*}
	Thus, it holds that
	\begin{align*}
		&\nu\sum_{i\leq1}\left\|e^{\eps_0(1+\mu)|\xi|}\xi^i
		J_1\right\|_{L^1_\xi \cap L^2_\xi}
		\leq C\nu^{1/2}\int_0^s(s-\tau)^{-1/2}\sum_{i\leq1}\|\pa_x^i N_R\|_{Y_{\mu,\tau}^1\cap Y_{\mu,\tau}^2}d\tau\\
		&\qquad+C\nu\int_0^s \sum_{i\leq2}\|\pa_x^i N_R\|_{Y_{\mu,\tau}^1\cap Y_{\mu,\tau}^2}d\tau
		+C\nu \int_0^s \sum_{i\leq1}\left\|\left\|(\pa_x^i N_R)_\xi(\tau,z)\right\|_{L^1_z(z\geq1+\mu)}\right\|_{L^1_\xi\cap L^2_\xi}d\tau\\
		&\leq C\nu^{1/2}\int_0^s\big((s-\tau)^{-1/2}+(\mu_0-\mu-\gamma\tau)^{-1}\big)\sum_{i\leq1}\|\pa_x^i N_R\|_{Y_{\mu_2,\tau}^1\cap Y_{\mu_2,\tau}^2}d\tau\\
		&\qquad+C\nu \int_0^s \sum_{i\leq2}\left\|\left\|\pa_x^i N_R(\tau,z)\right\|_{L^2_x}\right\|_{L^1(z\geq1+\mu)}d\tau\\
		&\leq C\nu^{1/2}\int_0^s\big((s-\tau)^{-1/2}+(\mu_0-\mu-\gamma\tau)^{-1}\big) \|N_R(\tau)\|_{W_{\mu_2,\tau}}d\tau ,
	\end{align*}
	where we take $\mu_2=\mu+\f12(\mu_0-\mu-\gamma\tau)$ and use Lemma \ref{lem: analytic recovery} in the last but one step. By Proposition \ref{prop: est of N in new norm} and Lemma \ref{lem: integral computation}, we have
	\begin{align*}
		&\nu\sum_{i\leq1}\left\|e^{\eps_0(1+\mu)|\xi|}\xi^i
		J_1\right\|_{L^1_\xi \cap L^2_\xi}
		\\
		& \leq C\nu^{1/2}\int_0^s \big((s-\tau)^{-1/2}+(\mu_0-\mu-\gamma\tau)^{-1}\big)(\mu_0-\mu-\gamma\tau)^{-\beta}\cdot\nu\big(1+E(\tau)\big)^2 d\tau\\
		&\leq \frac{C\nu^{3/2}}{\gamma^{1/2}}(\mu_0-\mu-\gamma s)^{-\beta}
		\big(1+E(s)\big)^2.
	\end{align*}
	
	For $J_2$, as in $J_1$, we use Lemma \ref{lem: est of BR} to have
	\begin{align*}
		&\nu\sum_{i\leq1}\left\|e^{\eps_0(1+\mu)|\xi|}\xi^i
		J_2\right\|_{L^1_\xi \cap L^2_\xi}
		\leq C\nu^{1/2}\int_0^s(s-\tau)^{-1/2}\sum_{i\leq1}\left\|e^{\eps_0(1+\mu)|\xi|}\xi^i(B_R)_\xi(\tau)\right\|_{L^1_\xi \cap L^2_\xi}d\tau\\
		&\qquad+C\nu\int_0^s \sum_{i\leq2}\left\|e^{\eps_0(1+\mu)|\xi|}\xi^i(B_R)_\xi(\tau)\right\|_{L^1_\xi \cap L^2_\xi}d\tau\\
		&\leq C\nu^{1/2}\int_0^s \big((s-\tau)^{-1/2}+(\mu_0-\mu-\gamma\tau)^{-1}\big)
		\sum_{i\leq1}\left\|e^{\eps_0(1+\mu_2)|\xi|}\xi^i(B_R)_\xi(\tau)\right\|_{L^1_\xi \cap L^2_\xi}d\tau\\
		&\leq C\nu^{1/2}\int_0^s \big((s-\tau)^{-1/2}+(\mu_0-\mu-\gamma\tau)^{-1}\big)
		\big\{e^{-\frac{C_1}{\nu \tau}}E(\tau)^{3/2}D_m(\tau)^{1/2} \\
		&\qquad+\nu(\mu_0-\mu-\gamma\tau)^{-\beta}\big(1+E(\tau)\big)^4
		+(\nu \tau)^{1/2}\big(1+E(\tau)\big)^{3/2}D_{vp}(\tau)^{1/2} \big\}d\tau\\
		&\leq \frac{C\nu^{3/2}}{\gamma^{1/2}}(\mu_0-\mu-\gamma s)^{-\beta} \big(E(s)+1\big)^4
		+Ce^{-\frac{C_1}{\nu s}}
		\big(\int_0^s D_m(\tau)^2 d\tau\big)^{1/4}\\
		&\qquad+C\nu\big(E(s)+1\big)^{3/2}\big(s^{3/4}+\gamma^{-3/4}s^{1/2}(\mu_0-\mu-\gamma s)^{-1/4}\big)
		\big(\int_0^s D_{vp}(\tau)^2 d\tau\big)^{1/4}\\
		&\leq C\nu^{3/2}\big(\gamma^{-1/2}(\mu_0-\mu-\gamma s)^{-\beta}+s^{5/4}\big)\big(E(s)+1\big)^4\\
		&\leq C\nu(\mu_0-\mu-\gamma s)^{-\beta} \big(E(s)+1\big)^4.
	\end{align*}

    Collecting these estimates together, we finish the proof.
\end{proof}

\subsection{Proof of Proposition \ref{prop: est of Eb(t)}}

\underline{The first inequality in Proposition \ref{prop: est of Eb(t)}.} \
Using Lemma \ref{lem: est of HRN}, Lemma \ref{lem: est of HRB}, Proposition \ref{prop: est of N in new norm} and Proposition \ref{prop: est of BR}, we have for $\mu<\mu_0-\gamma t$,
\begin{align*}
	&\sum_{i+j\leq1}\|\pa_x^i(y\pa_y)^j\big((1,x)\omega_R\big)(t)\|_{Y^1_{\mu,t}\cap Y^2_{\mu,t}}
	\\
	&\leq C\int_0^t\left\|\big(N_R(s),\widetilde{N_R}(s)\right\|_{W_{\mu,s}}ds+C\int_0^t\sum_{i\leq1}\left\|e^{\eps_0(1+\mu)|\xi|}\xi^i\big((B_R)_\xi(s),(\widetilde{B_R})_\xi(s)\big)\right\|_{L^1_\xi\cap L^2_\xi}ds\\
	&\leq C\int_0^t \Big\{\nu(\mu_0-\mu-\gamma s)^{-\beta}\big(1+E(s)\big)^4
	+\nu^{1/2}\big(1+E(s)\big)^{3/2}D_{vp}(s)^{1/2}
	+e^{-\frac{C_1}{\nu s}}E(s)^{3/2}D(s)^{1/2}  \Big\}ds\\
	&\leq C(\frac{\nu}{\gamma}+\nu t)\big(1+E(t)\big)^4.
\end{align*}
Similarly, we have for $\mu_1=\mu+\f12(\mu_0-\mu-\gamma s)$,
\begin{align*}
	&\sum_{i+j=2}\|\pa_x^i(y\pa_y)^j\big((1,x)\omega_R\big)(t)\|_{Y^1_{\mu,t}\cap Y^2_{\mu,t}}\\
	&\leq C\int_0^t \big((\mu_0-\mu-\gamma s)^{-1}+(\mu_0-\mu-\gamma s)^{-1/2}(t-s)^{-1/2}\big)\left\|\big(N_R(s),\widetilde{N_R}(s)\right\|_{W_{\mu_1,s}}ds \\
	&\qquad+C\int_0^t (\mu_0-\mu-\gamma s)^{-1}\left\|e^{\eps_0(1+\mu_1)|\xi|}\xi^i\big((B_R)_\xi(s),(\widetilde{B_R})_\xi(s)\big)\right\|_{L^1_\xi\cap L^2_\xi}ds\\
	&\leq C\int_0^t \big((\mu_0-\mu-\gamma s)^{-1}+(\mu_0-\mu-\gamma s)^{-1/2}(t-s)^{-1/2}\big)(\mu_0-\mu-\gamma s)^{-\alpha}\cdot \nu\big(1+E(s)\big)^4 ds\\
	&\qquad+C\int_0^t (\mu_0-\mu-\gamma s)^{-1}
	\Big\{\nu(\mu_0-\mu-\gamma s)^{-\beta}\big(1+E(s)\big)^4 \\
	&\qquad+\nu^{1/2}\big(1+E(s)\big)^{3/2}D_{vp}(s)^{1/2}
	+e^{-\frac{C_1}{\nu s}}E(s)^{3/2}D(s)^{1/2}  \Big\}ds\\
	&\leq C(\mu_0-\mu-\gamma s)^{-\beta}(\frac{\nu}{\gamma^{1/2}}+\nu t)\big(1+E(t)\big)^4.
\end{align*}
Thus,
\begin{align*}
	\|(1,x)\omega_R(t)\|_{Y_1(t)\cap Y_2(t)}
	&\leq C(\frac{\nu}{\gamma^{1/2}}+\nu t)\big(1+E(t)\big)^4.
\end{align*}

\medskip

\underline{The second inequality in Proposition \ref{prop: est of Eb(t)}.}
Using \eqref{eq: integral eq of omega R}, for $\mu<\mu_0-\gamma t$, we have
	\begin{align*}
		&e^{\frac{\eps_0 y^2}{\nu t}}|(\chi_b\omega_R)_\xi(t,y)|
		\leq \int_0^t\int_0^{+\infty}e^{\frac{\eps_0y^2}{\nu t}}H_\xi(t-s,y,z)|(N_R)_\xi(s,z)|dzds\\
		&\quad+\int_0^t\int_0^{+\infty}e^{\frac{\eps_0y^2}{\nu t}}|R_\xi(t-s,y,z)||(N_R)_\xi(s,z)|dzds\\
		&\quad+\int_0^te^{\frac{\eps_0y^2}{\nu t}}H_\xi(t-s,y,0)|(B_R)_\xi(s)|ds
		+\int_0^te^{\frac{\eps_0y^2}{\nu t}}|R_\xi(t-s,y,0)||(B_R)_\xi(s)|ds:=\sum_{i=1}^4 T_i.
	\end{align*}
	For $T_1$, a direct computation gives
	\begin{align*}
		e^{\frac{\eps_0 y^2}{\nu t}}H_\xi(t-s,y,z)
		\leq 
		\left\{
		\begin{aligned}
			&C\nu^{-1/2}(t-s)^{-1/2}e^{\frac{\eps_0 z^2}{\nu s}},\qquad z<1+\mu,\\
			&C\nu^{-1/2}(t-s)^{-1/2},\qquad z\geq1+\mu,
		\end{aligned}
		\right.
	\end{align*}
	which along with \eqref{est: NR without derivative} and Lemma \ref{lem: est of N away from the boundary} implies
	\begin{align*}
		&\|T_1\|_{L^1_\xi\cap L^2_\xi}
		\leq C\int_0^t \nu^{-1/2}(t-s)^{-1/2}\big(\|N_R(s)\|_{Y^1_{\mu,s}\cap Y^2_{\mu,s}}
		+\left\|\left\|(N_R)_\xi(s,z)\right\|_{L^1_z}\right\|_{L^1_\xi\cap L^2_\xi} \big)ds\\
		&\leq C\int_0^t \nu^{-1/2}(t-s)^{-1/2}\big(\|N_R(s)\|_{Y^1_{\mu,s}\cap Y^2_{\mu,s}}
		+\sum_{i\leq1}\|\pa_x^iN_R(s)\|_{L^2(1\leq z\leq3)} \big)ds\\
		&\leq C\int_0^t \nu^{-1/2}(t-s)^{-1/2}
		\Big(\nu\big(1+E(s)\big)^2
		+e^{\frac{5\eps_0}{\nu s}}\|\omega_R(s)\|^2_{H^3(\f78\leq y\leq4)} \Big)ds\\
		&\leq C\nu^{1/2}t^{1/2}\big(1+E(t)\big)^2.
	\end{align*}
	
For $T_2$, if $y\geq(\nu t)^{1/2}$, we utilize Lemma \ref{lem: property of Rxi} to obtain
	\begin{align*}
		&|T_2|
		\leq C\nu\int_0^t\int_0^{+\infty}\int_0^{t-s} e^{\frac{\eps_0y^2}{\nu t}}
		\Big|(-\xi^2+\xi\pa_y)\big(e^{-\nu\tau\xi^2}g(\nu\tau,y+z) \big) \Big||(N_R)_\xi(s,z)|dzd\tau ds.
	\end{align*}
	A direct computation gives
	\begin{align*}
		e^{\frac{\eps_0y^2}{\nu t}}\big|\nu(-\xi^2+\xi\pa_y)\big(e^{-\nu\tau\xi^2}g(\nu\tau,y+z) \big)\big|
		\leq \frac{C}{t}\cdot\frac{1}{(\nu\tau)^{1/2}}e^{-\frac{(y+z)^2}{5\nu\tau}},
	\end{align*}
	which implies
	\begin{align*}
		|T_2|
		&\leq \frac{C}{t}\int_0^t \int_0^{+\infty} \int_0^{t-s}\frac{1}{(\nu\tau)^{1/2}}e^{-\frac{(y+z)^2}{5\nu\tau}}|(N_R)_\xi(s,z)|d\tau dz ds\\
		&\leq \int_0^t\int_0^{+\infty}\frac{C}{\nu^{1/2}(t-s)^{1/2}}|(N_R)_\xi(s,z)|dzds
	\end{align*}
	Thus, as in $T_1$, we have
	\begin{align*}
		\|T_2\|_{L^1_\xi\cap L^2_\xi}
		&\leq C\nu^{-1/2}\int_0^t \big(\|N_R(s)\|_{Y^1_{\mu,s}\cap Y^2_{\mu,s}}
		+\sum_{i\leq1}\|\pa_x^iN_R(s)\|_{L^2(1\leq z\leq3)} \big)ds\\
		&\leq C\nu^{1/2}t^{1/2}\big(1+E(t)\big)^2.
	\end{align*}
	If $y\leq(\nu t)^{1/2}$, we still use Lemma \ref{lem: property of Rxi} to obtain
	\begin{align*}
		&e^{-|\xi|}|T_2|
		\leq C e^{-|\xi|}\int_0^t\int_0^{+\infty}\big( (|\xi|+\frac{1}{\sqrt{\nu}})e^{-\theta_0(|\xi|+\frac{1}{\sqrt{\nu}})z}
		+\frac{1}{\sqrt{\nu(t-s)}}e^{-\theta_0\frac{z^2}{\nu(t-s)}}\big)|(N_R)_\xi(s,z)|dzds\\
		&\leq \int_0^t\int_0^{+\infty}\frac{C}{\nu^{1/2}(t-s)^{1/2}}|(N_R)_\xi(s,z)|dzds,
	\end{align*}
	which implies
	\begin{align*}
		\left\|e^{-|\xi|}|T_2|\right\|_{L^1_\xi\cap L^2_\xi}
		&\leq  \int_0^t \frac{C}{\sqrt{\nu(t-s)}}\big(\|N_R(s)\|_{Y^1_{\mu,s}\cap Y^2_{\mu,s}}
		+\left\|\left\|(N_R)_\xi(s,z)\right\|_{L^1_z}\right\|_{L^1_\xi\cap L^2_\xi} \big)ds\\
		&\leq C\nu^{1/2}t^{1/2}\big(1+E(t)\big)^2.
	\end{align*}
	
	For $T_3, T_4$, we similarly have
	\begin{align*}
		&\left\|e^{-|\xi|}(T_3+T_4)\right\|_{L^1_\xi\cap L^2_\xi}
		\leq \int_0^t \frac{C}{\nu^{1/2}(t-s)^{1/2}}\left\|(B_R)_\xi(s)\right\|_{L^1_\xi\cap L^2_\xi} ds\\
		&\leq \int_0^t \frac{C}{\nu^{1/2}(t-s)^{1/2}}
		\Big(\nu\big(E(s)+1\big)^4
		+(\nu s)^{1/2}\big(E(s)+1\big)^{3/2}
		D_{vp}(s)^{1/2}
		+e^{-\frac{C_1}{\nu s}}E(s)^{3/2}D(s)^{1/2}\Big) ds\\
		&\leq C(\nu t)^{1/2}\big(E(t)+1\big)^4,
	\end{align*}
	where we used \eqref{est: BR i=0}.
	
	Summing these estimates, we obtain the desired result.

\appendix 

\section{Some technical lemmas}


The following lemmas are frequently used in constructing the approximate solutions.
First, we present the velocity formula via the Biot-Savart law, whose proof is provided in \cite{HWYZ}.

\begin{lemma}\label{lem: derivation of velocity formula}
	It holds that
	\begin{align}
  	BS_{\mathbb R^2_+}[\chi_{vp}\omega](t,x,y)
  	=\frac{\alpha}{(\nu t)^{1/2}}\{\mathcal V^{\mathcal W}\big(\eta,t)-
  	\widetilde{\mathcal V^{\mathcal W}}(\eta+\widetilde\eta,t)\}.
  \end{align}
\end{lemma}

The next lemma is employed to estimate the positional variation of the point vortex.
\begin{lemma}\label{lem: ODE est}
	For $f(t)\in C\big( [0,T] ; \mathbb R^2_+ \big)$, there exist $T>0$ and a unique solution $X(t)\in C^1\big( [0,T] ; \mathbb R^2_+ \big)$ satisfying 
	\begin{align*}
		X'(t)
		=-\frac{\alpha}{(\nu t)^{1/2}}\mathcal V^G(\widetilde\eta)
		+f(t),
		\quad X(0)=X_0,
	\end{align*}
	and there exists $C$ independent of $\nu$ such that 
	 \begin{align*}
	 	|X'(t)|\leq C,\qquad |X(t)-X_0|\leq Ct.
	 \end{align*}
\end{lemma}

\medskip

Next lemma which has been proved in \cite{Gallay 5} reveals the decay rate of the velocity at infinity.

\begin{lemma}\label{lem: decay rate of velocity}
	If one of the following holds
	
	(1) $m=1$,
	
	(2) $m=2$, and $\int_{\mathbb R^2}w(\eta)d\eta=0$,
	
	(3) $m=3$, and $\int_{\mathbb R^2}w(\eta)d\eta=\int_{\mathbb R^2}\eta_1 w(\eta)d\eta=\int_{\mathbb R^2}\eta_2 w(\eta)d\eta=0$,
	
	then we have that for $1<p_1<2<p_2<+\infty$,
	\begin{align*}
		\left\|\langle\eta\rangle^m \mathcal V^w(\eta)\right\|_{L^\infty}
		\leq C_{p_1,p_2}\left\|\langle\eta\rangle^m w(\eta)\right\|_{L^{p_1}\cap L^{p_2}}.
	\end{align*}
\end{lemma}

\medskip

Next lemma provides velocity estimates and is established by the Biot-Savart law \eqref{BS law formulation 2}.
\begin{lemma}\label{lem: BS est appendix}
	It holds that
	
	(1) $\|BS_{\mathbb R^2_+}[f]\|_{L^p}\leq C\|f\|_{L^{\frac{2p}{p+2}}}$ \quad for \quad $2<p<+\infty$.
	
	(2) $\|BS_{\mathbb R^2_+}[f]\|_{L^\infty(A)}\leq C\|f\|_{L^1}$ \quad if \quad $dist(A,\operatorname{supp}f)>0$.
\end{lemma}

Direct computation yields the following lemma, which is employed to address the loss of derivatives.

\begin{lemma}\label{lem: analytic recovery}
	For $\widetilde \mu>\mu\geq0$, we have
\begin{align*}
		e^{\eps_0(1+\mu-y)_+|\xi|}|(\pa_x f)_\xi(y)|
		&\leq \frac{C}{\widetilde\mu-\mu}
		e^{\eps_0(1+\widetilde\mu-y)_+|\xi|}|f_\xi(y)|, 
\end{align*}
and
\begin{align*}	
		 e^{\eps_0(1+\mu)\frac{y^2}{\nu t}}
		\left|y\pa_y\Big( e^{-\frac{(y-z)^2}{4\nu(t-s)}} \Big)\right|
		&\leq\frac{C}{\sqrt{(\widetilde\mu-\mu)(t-s)}}
		e^{\eps_0(1+\widetilde\mu)\frac{y^2}{\nu t}}
		e^{-\frac{(y-z)^2}{5\nu(t-s)}}.
	\end{align*}
\end{lemma}

The following lemma is used to close the uniform boundedness and is proved in \cite{HWYZ}.
\begin{lemma}\label{lem: integral computation}
	For $\frac{1}{2}<\beta<1$, $0<\zeta<1$, $\gamma>0$ and $\mu<\mu_0-\gamma t$, it holds that
	\begin{align*}
		(\mu_0-\mu-\gamma t)^\beta\int_0^t(\mu_0-\mu-\gamma s)^{-1-\beta}ds&
		\leq \frac{C}{\gamma},\\
		(\mu_0-\mu-\gamma t)^\beta\int_0^t(\mu_0-\mu-\gamma s)^{-\frac{1}{2}-\beta}(t-s)^{-\f12}ds&\leq \frac{C}{\gamma^{\f12}},\\
		\sup_{\mu<\mu_0-\gamma t}(\mu_0-\mu-\gamma t)^\zeta \ln\frac{\mu_0-\mu}{\mu_0-\mu-\gamma t}&\leq C(\gamma t)^\zeta,
	\end{align*}
	here $C$ is a constant depending on $\mu_0$, $\beta$ and $\zeta$.
\end{lemma}

\medskip

To deal with the nonlinear terms, we require the following product estimates which can be found in \cite{HWYZ}.
\begin{lemma}\label{lem: product estimate}
	For $0<\mu<\mu_0-\gamma s$, we have for $k=1,2$,
	\begin{align*}
		\|fg\|_{Y^k_{\mu,s}}
		\leq \left\|\sup_{0<y<1+\mu}e^{\eps_0(1+\mu-y)_+|\xi|}|f_\xi(s,y)|\right\|_{L^1_\xi}\|g(s)\|_{Y^k_{\mu,s}}.
	\end{align*}
\end{lemma}

\section{Estimates for boundary layer part}

This section is devoted to estimating $\omega_b^{(0)}, \omega_b^{(1)}$ which are defined by \eqref{eq: omega p0} and \eqref{eq: omega p1}. 
First, we provide a solution formula for the heat equation subject to the Neumann boundary condition. As it follows from a direct computation, we omit the proof.

\begin{lemma}\label{lem: solution formula for basic equation0}
Let $w(t,z)$ satisfy  
	\begin{align*}
		\left\{
		\begin{aligned}
			& \pa_t w-\pa_z^2 w=f,\qquad (t,z)\in \mathbb R_+\times\mathbb R_+,\\
			& w|_{t=0}=b,\qquad z\in \mathbb R_+,\\
			&\pa_z w|_{z=0}=h,\qquad t\in \mathbb R_+,
		\end{aligned}
		\right.
	\end{align*}
Then, for $t>0$, $w$ can be expressed as
	\begin{align*}
		w(t,z)=e^{t\Delta_N}b(z)
		+\int_0^t e^{(t-s)\Delta_N}f(s,z)ds
		-\int_0^t\int_z^{+\infty}\frac{\tilde z}{(4\pi)^{1/2}(t-s)^{3/2}}e^{-\frac{\tilde z^2}{4(t-s)}}h(s)d\tilde zds,
	\end{align*}
	where
	\begin{align*}
		e^{t\Delta_N}b(z)
		:=\int_0^{+\infty}\frac{1}{(4\pi t)^{1/2}}\big(e^{-\frac{(z-\tilde z)^2}{4t}}
		+e^{-\frac{(z+\tilde z)^2}{4t}} \big)b(\tilde z)d\tilde z.
	\end{align*}
\end{lemma}

\begin{remark}
	During the above lemma, when the compatibility condition $\pa_z b|_{z=0}=h|_{t=0}$ does not hold,  $w|_{t=0}$ is defined as $\lim_{t\rightarrow0^+}w$. 
\end{remark}

Now, we are in a position to state the main result of this section.

\begin{proposition}\label{prop: WP of boundary layer}
	For $T$ small, there exists $C,C'>0$ such that for $0\leq t\leq T, k=0,1$, 
	\begin{align}\label{est: omega pk U pk}
		&\sum_{i+j\leq8}\left\|e^{C'|\xi|}\left\|e^{\frac{C'z^2}{t}}\big( (1,x)\pa_x^i(z\pa_z)^j\omega_b^{(k)} \big)_\xi(t,z)\right\|_{L^1_z}\right\|_{L^1_\xi \cap L^2_\xi}
		+t^{1/2}\left\|e^{C'|\xi|}\big((1,x)\omega_b^{(k)}\big)_\xi|_{z=0}\right\|_{L^1_\xi \cap L^2_\xi} \\
		\nonumber
		&+\sum_{i+j\leq8}\left\|e^{C'|\xi|}\left\|\big(\pa_x^i(z\pa_z)^j U_b^{(k)} \big)_\xi(t,y)\right\|_{L^\infty_y}\right\|_{L^1_\xi}
		\leq C,
	\end{align} 
	where $U_b^{(k)}$ is defined by \eqref{eq: def of up(k)}.
\end{proposition}

\begin{remark}
	We stress here the factor $t^{1/2}$ before the second term in \eqref{est: omega pk U pk} can be removed in the case $k=1$, since the initial data and boundary condition of $\omega_b^{(1)}$ are compatible.
\end{remark}

\begin{proof}
	We only focus on $\omega_b^{(0)}$, since $\omega_b^{(1)}$ can be proved in the same manner. Recall $\omega_b^{(0)}$ satisfies
	\begin{align*}
		\left\{
		\begin{aligned}
			&\pa_t \omega_b^{(0)}
		-\pa_z^2 \omega_b^{(0)}
		=-\big(u_b^{(0)}+u_{vp}^{(0)}(t,x,0) \big)\pa_x\omega_b^{(0)}
		-\big(\frac{v_b^{(0)}}{y} +\pa_y v_{vp}^{(0)}(t,x,0) \big) z\pa_z\omega_b^{(0)}:=F,\\
		&\lim_{t\rightarrow0^+}\omega_b^{(0)}=-u_0\delta_{\partial \mathbb R^2_+},
		\qquad\qquad\quad
		\pa_z\omega_b^{(0)}|_{z=0}
		=\big(\pa_y\Delta_D^{-1}BC_{b,0} \big)|_{y=0}
	+\big(\pa_y\Delta_D^{-1}BC_{vp,0} \big)|_{y=0},
		\end{aligned}
		\right.
	\end{align*}
	with $BC_{b,0}, BC_{vp,0}$ defined by
	\begin{align*}
		BC_{b,0}
	=&\nu^{-1/2}\Big\{ (u_{vp}^{(0)}+u_b^{(0)})(t,x,y)(\chi_b(y)\pa_x\omega_b^{(0)})(t,x,\frac{y}{\nu^{1/2}})\\
	\nonumber
	&+\frac{v_{vp}^{(0)}+v_b^{(0)}}{y}(t,x,y)y\pa_y\big(\chi_b(y)\omega_b^{(0)}(t,x,\frac{y}{\nu^{1/2}})\big) \Big\},\\
	BC_{vp,0}
	=&-\frac{\alpha}{\sqrt{\nu t}}\mathcal V^G(\frac{X-X(t)^\ast}{\sqrt{\nu t}})\cdot\nabla \Big\{\frac{\alpha}{\nu t}G(\frac{X-X(t)}{\sqrt{\nu t}})\chi_{vp} \Big\}.
	\end{align*}
	
	We note $BC_{b,0}|_{t=0}=0$, $\lim_{t\rightarrow0}BC_{vp,0}\neq 0$, thus the initial data and boundary condition of $\omega_b^{(0)}$ are incompatible. Therefore, we make the decomposition $\omega_b^{(0)}=\omega_c+\omega_{re}$, where $\omega_c$ is the corrector and $\omega_{re}$ is the remainder defined by
	\begin{align}\label{eq: omega c}
		\left\{
		\begin{aligned}
			&\pa_t\omega_c-\pa_z^2\omega_c=0,\quad \lim_{t\rightarrow0^+}\omega_c=-u_0\delta_{\partial \mathbb R^2_+},\\
			&\pa_z\omega_c|_{z=0}=\big(\pa_y\Delta_D^{-1}BC_{vp,0} \big)|_{y=0},
		\end{aligned}
		\right.
		\qquad
		\left\{
		\begin{aligned}
			&\pa_t\omega_{re}-\pa_z^2\omega_{re}=F,\quad \omega_{re}|_{t=0}=0, \\
			&\pa_z\omega_{re}|_{z=0}=\big(\pa_y\Delta_D^{-1}BC_{b,0} \big)|_{y=0}.
		\end{aligned}
		\right.
	\end{align}
	We denote 
	\begin{align*}
		U_c(t,x,y)=BS_{\mathbb R^2_+}[\chi_b(y)\omega_c(t,x,\frac{y}{\nu^{1/2}})],\qquad U_{re}(t,x,y)=BS_{\mathbb R^2_+}[\chi_b(y)\omega_{re}(t,x,\frac{y}{\nu^{1/2}})].
	\end{align*}
	 The proof of this proposition is derived by Lemma \ref{lem: est of omega c U c} and Lemma \ref{lem: est of omega re U re}.
	\end{proof}

\medskip

Next lemma provides estimates for $\omega_c$.
\begin{lemma}\label{lem: est of omega c U c}
	For $T$ small, there exists $C,C'>0$ such that for $0\leq t\leq T$, 
	\begin{align}\label{est: omega c U c}
		&\sum_{i+j\leq8}\left\|e^{C'|\xi|}\left\|e^{\frac{C'z^2}{t}}\big( (1,x)\pa_x^i(z\pa_z)^j\omega_c \big)_\xi(t,z)\right\|_{L^1_z}\right\|_{L^1_\xi \cap L^2_\xi}
		+t^{1/2}\left\|e^{C'|\xi|}\big((1,x)\omega_c \big)_\xi(t,0)\right\|_{L^1_\xi \cap L^2_\xi}\\
		\nonumber
		&\qquad+\sum_{i+j\leq8}\left\|e^{C'|\xi|}\left\|\pa_x^i(z\pa_z)^j\big(\pa_x^i(z\pa_z)^j U_c \big)_\xi(t,y)\right\|_{L^{\infty}_y}\right\|_{L^1_\xi}
		\leq C.
	\end{align}
\end{lemma}

\begin{proof}
	By \eqref{BS law formulation 2} and integration by parts, we obtain
	\begin{align*}
		&\pa_y\Delta_D^{-1}BC_{vp,0}|_{y=0}(t,x)\\
		&=-\frac{1}{\pi}\frac{\alpha^2}{(\nu t)^{3/2}}\int_{\mathbb R^2_+} \nabla_{\tilde x,\tilde y}\big(\frac{\tilde y \chi_{vp}(\tilde x,\tilde y)}{(x-\tilde x)^2+\tilde y^2}\big)
		\cdot \mathcal V^G(\frac{\widetilde X-X(t)^\ast}{\sqrt{\nu t}}) G(\frac{\widetilde X-X(t)}{\sqrt{\nu t}})d\widetilde X.
	\end{align*}
	Thus, we have
	\begin{align*}
		&\left|\big(\pa_x^i\pa_y\Delta_D^{-1}BC_{vp,0}\big)_\xi|_{y=0}(t)\right|\\
		&=\frac{\alpha^2|\xi|^i}{(\nu t)^{3/2}}
		\left|\int_{\chi_{vp}}\nabla_{\tilde x,\tilde y}
		\big(e^{-2\pi i\tilde x\xi}e^{-2\pi y|\xi|}\chi_{vp}(\tilde x,\tilde y)\big)
		\cdot \mathcal V^G(\frac{\widetilde X-X(t)^\ast}{\sqrt{\nu t}}) G(\frac{\widetilde X-X(t)}{\sqrt{\nu t}})d\widetilde X\right|\\
		&\leq \frac{C}{(\nu t)^{1/2}}e^{-10|\xi|}
		\big(\int_{|\eta|\leq\frac{8}{\sqrt{\nu t}}}+\int_{|\eta|>\frac{8}{\sqrt{\nu t}}}\big)
		\left|\mathcal V^G(\eta+\frac{X(t)-X(t)^\ast}{\sqrt{\nu t}})\right|G(\eta)d\eta
		\leq Ce^{-10|\xi|}.
	\end{align*}

	By Lemma \ref{lem: solution formula for basic equation0} and taking Fourier transform w.r.t $x$, we have
	\begin{align*}
		(\pa_x^i(z\pa_z)^j\omega_c)_\xi&(t,z)
		=-\frac{1}{(\pi t)^{1/2}}(z\pa_z)^j e^{-\frac{z^2}{4t}}(\pa_x^i u_0)_\xi|_{y=0} \\
		&-\int_0^t
		(z\pa_z)^j\int_z^{+\infty}
		\frac{\tilde z}{(4\pi)^{1/2}(t-s)^{3/2}}e^{-\frac{\tilde z^2}{4(t-s)}}
		\big(\pa_x^i\pa_y\Delta_D^{-1}BC_{vp,0}\big)_\xi|_{y=0}(s)d\tilde zds,
	\end{align*}
	which along with Lemma \ref{lem: est of Ue(k)} and the following estimates which hold for $C', t$ small
	\begin{align}\label{est of integral of gauss kernel}
		(z\pa_z)^j\int_z^{+\infty}
		\frac{\tilde z}{(t-s)^{3/2}}e^{-\frac{\tilde z^2}{4(t-s)}}d\tilde z
		\leq \frac{C}{(t-s)^{1/2}}e^{-\frac{z^2}{5(t-s)}},
		\qquad
		e^{\frac{C'z^2}{t}}e^{-\frac{z^2}{5(t-s)}}\leq e^{-\frac{z^2}{6(t-s)}},
	\end{align}
	gives
	\begin{align}\label{est medium omega c}
		e^{C'|\xi|}e^{\frac{C'z^2}{t}}|(\omega_c)_\xi|
		\leq \int_0^t \frac{C}{(t-s)^{1/2}}e^{-\frac{z^2}{6(t-s)}}ds\cdot e^{-9|\xi|}
		+\frac{C}{t^{1/2}}e^{-\frac{z^2}{5t}}e^{C'|\xi|}\left|(\pa_x^i u_0)_\xi|_{y=0}\right| ,
	\end{align}
	which implies
	\begin{align*}
		\left\|e^{C'|\xi|}\left\|e^{\frac{C'z^2}{t}}(\pa_x^i(z\pa_z)^j\omega_c)_\xi(t,z)\right\|_{L^1_z}\right\|_{L^1_\xi \cap L^2_\xi}\leq C.
	\end{align*}
	Similarly,
	\begin{align*}
		\left\|e^{C'|\xi|}(\omega_c )_\xi(t,0)\right\|_{L^1_\xi \cap L^2_\xi}
		&\leq \int_0^t(t-s)^{-1/2}\left\|e^{C'|\xi|} \big(\pa_y\Delta_D^{-1}BC_{vp,0}\big)_\xi|_{y=0}(t)\right\|_{L^1_\xi \cap L^2_\xi}\\
		&\quad+\frac{C}{t^{1/2}}\left\| e^{C'|\xi|}\left|( u_0)_\xi|_{y=0}\right|\right\|_{L^1_\xi \cap L^2_\xi} 
		\leq Ct^{-1/2}.
	\end{align*}
	
	The term $(x\omega_c)_\xi$ can be estimated similarly by the fact $(x\omega_c)_\xi=\frac{i}{2\pi}\pa_\xi(\omega_c)_\xi$. 
	
	For the third term in \eqref{est: omega c U c}, we utilize Lemma \ref{lem: velocity formula} and \eqref{est medium omega c} to deduce that
	\begin{align*}
		\|(\pa_x^i(z\pa_z)^j U_c)_\xi(t,y)\|_{L^\infty_y}
		\leq \|(\pa_x^i(z\pa_z)^j\omega_c)_\xi(t,z)\|_{L^1_z},
	\end{align*}
	which implies the desired result.
	\end{proof}

Next lemma provides estimates for $\omega_{re}$.

\begin{lemma}\label{lem: est of omega re U re}
	For $T$ small, there exists $C,C'>0$ such that for $0\leq t\leq T$, 
	\begin{align}\label{est: omega re U re}
		&\sum_{i+j\leq8}\left\|e^{C'|\xi|}\left\|e^{\frac{C'z^2}{t}}\big( (1,x)\pa_x^i(z\pa_z)^j\omega_{re} \big)_\xi(t,z)\right\|_{L^1_z}\right\|_{L^1_\xi \cap L^2_\xi}
		+\left\|e^{C'|\xi|}\big((1,x)\omega_{re} \big)_\xi(t,0)\right\|_{L^1_\xi \cap L^2_\xi}\\
		\nonumber
		&\qquad+\sum_{i+j\leq8}\left\|e^{C'|\xi|}\left\|\pa_x^i(z\pa_z)^j\big(\pa_x^i(z\pa_z)^j U_{re} \big)_\xi(t,y)\right\|_{L^{\infty}_y}\right\|_{L^1_\xi}
		\leq C.
	\end{align}
\end{lemma}

In order to prove Lemma \ref{lem: est of omega re U re}, we introduce the following norms for $\mu<\mu_0-\gamma t$, $\beta\in(\f12,1)$ and sufficiently small $\eps_0$ to be determined later.
\begin{align*}
	\|f\|_{\widetilde {X}_{\mu,t}}
	=\left\|\left\|e^{\eps_0(1+\mu)|\xi|}f_\xi(t,z)\right\|_{L^\infty_z}\right\|_{L^1_\xi},\qquad
	\|f\|_{\widetilde {Y}_{\mu,t}}
	=\left\|\left\|e^{\eps_0(1+\mu)\frac{z^2}{t}}e^{\eps_0(1+\mu)|\xi|}f_\xi(t,z)\right\|_{L^1_z}\right\|_{L^1_\xi\cap L^2_\xi},\\
	\|\omega_{re}\|_{Z(t)}
	=\sup_{\mu<\mu_0-\gamma t}
	\Big(\sum_{i+j\leq8}\left\|\pa_x^i(z\pa_z)^j\omega_{re}\right\|_{\widetilde{Y}_{\mu,t}}
	+(\mu_0-\mu-\gamma t)^\beta \sum_{i+j=9}\left\|\pa_x^i(z\pa_z)^j\omega_{re}\right\|_{\widetilde{Y}_{\mu,t}} \Big).
\end{align*}

\noindent\textbf{Proof of Lemma \ref{lem: est of omega re U re}.} 
By Lemma \ref{lem: solution formula for basic equation0} and taking Fourier transform w.r.t $x$, we have
\begin{align}\label{eq: solution formula for omega re}
	&(\omega_{re})_\xi(t,z)
		=\int_0^t e^{(t-s)\Delta_N}F_\xi(s,z)ds\\
		\nonumber
		&\qquad-\int_0^t\int_z^{+\infty}
		\frac{\tilde z}{(4\pi)^{1/2}(t-s)^{3/2}}e^{-\frac{\tilde z^2}{4(t-s)}}
		\big(\pa_y\Delta_D^{-1}BC_{b,0}\big)_\xi|_{y=0}(s)d\tilde zds.
\end{align}
Thus, it holds that
\begin{align*}
	\big(\pa_x^i(z\pa_z)^j&\omega_{re}\big)_\xi(t,z)
	=\int_0^t (z\pa_z)^j\int_0^{+\infty}
	\frac{1}{(4\pi (t-s))^{1/2}}\big(e^{-\frac{(z-\tilde z)^2}{4(t-s)}}
		+e^{-\frac{(z+\tilde z)^2}{4(t-s)}} \big)(\pa_x^i F)_\xi(s,\tilde z)d\tilde zds\\
		&-\int_0^t (z\pa_z)^j \big(\int_z^{+\infty}
		\frac{\tilde z}{(4\pi)^{1/2}(t-s)^{3/2}}e^{-\frac{\tilde z^2}{4(t-s)}} d\tilde z\big)(\pa_x^i\pa_y\Delta_D^{-1}BC_{b,0})_\xi|_{y=0}(s)ds.
\end{align*}
Integration by parts yields
\begin{align*}
	&\left| (z\pa_z)\int_0^{+\infty}
	\frac{1}{(4\pi (t-s))^{1/2}}\big(e^{-\frac{(z-\tilde z)^2}{4(t-s)}}
		+e^{-\frac{(z+\tilde z)^2}{4(t-s)}} \big)(\pa_x^i F)_\xi(s,\tilde z)d\tilde zds \right|\\
		&=\left| \int_0^{+\infty}\frac{z}{(4\pi (t-s))^{1/2}}
		\pa_{\tilde z}\big(e^{-\frac{(z-\tilde z)^2}{4(t-s)}}
		-e^{-\frac{(z+\tilde z)^2}{4(t-s)}} \big)(\pa_x^i F)_\xi(s,\tilde z)d\tilde zds \right|\\
		&\leq \int_0^{+\infty}\frac{|z-\tilde z|}{(4\pi (t-s))^{1/2}}
		\left|\pa_{\tilde z}\big(e^{-\frac{(z-\tilde z)^2}{4(t-s)}}
		-e^{-\frac{(z+\tilde z)^2}{4(t-s)}} \big)\right| \left|(\pa_x^i F)_\xi(s,\tilde z)\right| d\tilde zds\\ 
		&\quad+\int_0^{+\infty}\frac{1}{(4\pi (t-s))^{1/2}}
		\big(e^{-\frac{(z-\tilde z)^2}{4(t-s)}}
		+e^{-\frac{(z+\tilde z)^2}{4(t-s)}} \big)
		\left|\big(\pa_x^i(\tilde z\pa_{\tilde z})F\big)_\xi(s,\tilde z)\right|d\tilde zds\\
		&\leq\int_0^{+\infty}\frac{C}{(t-s)^{1/2}}e^{-\frac{(z-\tilde z)^2}{5(t-s)}}\big( \left|(\pa_x^i F)_\xi(s,\tilde z)\right|
		+\left|\big(\pa_x^i(\tilde z\pa_{\tilde z})F\big)_\xi(s,\tilde z)\right|\big)d\tilde zds.
\end{align*}
Similarly, by integration by parts several times, we obtain
\begin{align}\label{est: conormal transform  back}
	&\left| (z\pa_z)^j\int_0^{+\infty}
	\frac{1}{(4\pi (t-s))^{1/2}}\big(e^{-\frac{(z-\tilde z)^2}{4(t-s)}}
		+e^{-\frac{(z+\tilde z)^2}{4(t-s)}} \big)(\pa_x^i F)_\xi(s,\tilde z)d\tilde zds \right|\\
		\nonumber
		&\leq \int_0^{+\infty}\frac{C}{(t-s)^{1/2}}e^{-\frac{(z-\tilde z)^2}{5(t-s)}}\sum_{j_0\leq j}\left|\big(\pa_x^i(\tilde z\pa_{\tilde z})^{j_0}F\big)_\xi(s,\tilde z)\right| d\tilde zds.
\end{align}
A direct computation gives
\begin{align}\label{est: conormal on integral}
	(z\pa_z)^j \big(\int_z^{+\infty}
		\frac{\tilde z}{(t-s)^{3/2}}e^{-\frac{\tilde z^2}{4(t-s)}} d\tilde z\big)
		\leq \frac{C}{(t-s)^{1/2}}e^{-\frac{z^2}{5(t-s)}}.
\end{align}
Using \eqref{est: conormal transform  back}, \eqref{est: conormal on integral} and the following fact
\begin{align}\label{est: weight transform 0}
	e^{\frac{\eps_0(1+\mu)z^2}{t}}e^{-\frac{(z-\tilde z)^2}{10(t-s)}}
	\leq e^{\frac{\eps_0(1+\mu)\tilde z^2}{s}},
\end{align}
we have
\begin{align*}
	&\sum_{i+j\leq 8}e^{\frac{\eps_0(1+\mu)z^2}{t}}
	\left|\big(\pa_x^i(z\pa_z)^j\omega_{re}\big)_\xi(t,z)\right|\\
	&\leq \int_0^t\int_0^{+\infty}
	\frac{C}{(t-s)^{1/2}}e^{-\frac{(z-\tilde z)^2}{10(t-s)}}e^{\frac{\eps_0(1+\mu)\tilde z^2}{s}}
	\sum_{i+j\leq8}\left|\big(\pa_x^i(\tilde z\pa_{\tilde z})^jF\big)_\xi(s,\tilde z)\right| d\tilde zds\\
	&\qquad+\int_0^t
	\frac{C}{(t-s)^{1/2}}e^{-\frac{z^2}{10(t-s)}}
	\sum_{i\leq8}\left|(\pa_x^i\pa_y\Delta_D^{-1}BC_{b,0})_\xi|_{y=0}(s) \right|ds,
\end{align*}
which implies
\begin{align*}
	\sum_{i+j\leq 8}
	\left\|\pa_x^i(z\pa_z)^j\omega_{re}\right\|_{\widetilde{Y}_{\mu,t}}
	\leq& C\int_0^t \sum_{i+j\leq 8}
	\left\|\pa_x^i(z\pa_z)^jF(s)\right\|_{\widetilde{Y}_{\mu,s}}ds\\
	&+C\int_0^t \sum_{i\leq8}\left\|e^{\eps_0(1+\mu)|\xi|}(\pa_x^i\pa_y\Delta_D^{-1}BC_{b,0})_\xi|_{y=0}(s)\right\|_{L^1_\xi\cap L^2_\xi} ds.
\end{align*}
Denoting $\mu_1=\mu+\f12(\mu_0-\mu-\gamma s)$ and using Lemma \ref{lem: analytic recovery}, we proceed as above to derive that
\begin{align*}
	&\sum_{i+j=9}
	\left\|\pa_x^i(z\pa_z)^j\omega_{re}\right\|_{\widetilde{Y}_{\mu,t}}
	\leq C\int_0^t (\mu_0-\mu-\gamma s)^{-1} \sum_{i\leq8}\left\|e^{\eps_0(1+\mu_1)|\xi|}(\pa_x^i\pa_y\Delta_D^{-1}BC_{b,0})_\xi|_{y=0}(s)\right\|_{L^1_\xi\cap L^2_\xi} ds \\
	&\qquad+C\int_0^t \big((\mu_0-\mu-\gamma s)^{-1}+(\mu_0-\mu-\gamma s)^{-1/2}(t-s)^{-1/2}\big) \sum_{i+j\leq 8}
	\left\|\pa_x^i(z\pa_z)^jF(s)\right\|_{\widetilde{Y}_{\mu_1,s}}ds.
\end{align*}

Armed with Lemma \ref{lem: est of F BC in omega re} below, we arrive at
\begin{align*}
	\|\omega_{re}\|_{Z(t)}
	\leq& C\int_0^t 1+(\mu_0-\mu-\gamma s)^{-\beta}\big(\|\omega_{re}\|_{Z(s)}+\|\omega_{re}\|_{Z(s)}^2\big)ds\\
	&+C(\mu_0-\mu-\gamma t)^{\beta} \int_0^t 
	\big((\mu_0-\mu-\gamma s)^{-1}+(\mu_0-\mu-\gamma s)^{-1/2}(t-s)^{-1/2}\big)\\
	&\qquad\cdot\Big(1+(\mu_0-\mu-\gamma s)^{-\beta}\big(\|\omega_{re}\|_{Z(s)}+\|\omega_{re}\|_{Z(s)}^2\big) \Big)ds,
\end{align*}
which along with Lemma \ref{lem: integral computation} implies
\begin{align*}
	\sup_{0<s<t}\|\omega_{re}\|_{Z(s)}
	\leq C(t+\gamma^{-1/2}) \sup_{0<s<t}\big(\|\omega_{re}\|_{Z(s)}+\|\omega_{re}\|_{Z(s)}^2\big)
	+C\gamma^{-1/2}.
\end{align*}
Therefore, we obtain a uniform estimate for $\omega_{re}$ via a continuous argument. And $x\omega_{re}$ can be estimated in a similar manner. The third term in \eqref{est: omega re U re} is bounded by Lemma \ref{lem: velocity formula}. For the second term, we utilize \eqref{eq: solution formula for omega re}  to obtain
\begin{align*}
	\left|e^{C'|\xi|}(\omega_{re})_\xi|_{z=0}\right|
	&\leq \int_0^t\int_0^{+\infty}\frac{1}{(t-s)^{1/2}}e^{-\frac{\widetilde z^2}{4(t-s)}}e^{C'|\xi|}|F_\xi(s,\widetilde z)|d\widetilde zds\\
	&\quad+\int_0^t \frac{1}{(t-s)^{1/2}}e^{C'|\xi|} \left|(\pa_y\Delta_D^{-1}BC_{b,0})_\xi|_{y=0}(s)\right|ds,
\end{align*}
which along with Lemma \ref{lem: est of F BC in omega re} below and the uniform boundedness for $\omega_{re}$ implies
\begin{align*}
	\left\|e^{C'|\xi|}(\omega_{re})_\xi|_{z=0}\right\|_{L^1_\xi\cap L^2_\xi}\leq C.
\end{align*}

\ef

To complete the proof, it remains only to establish the following lemma, which deals with $F$ and the boundary terms.
\begin{lemma}\label{lem: est of F BC in omega re}
	For $\mu<\mu_0-\gamma t$, it holds that
	\begin{align*}
		&\sum_{i+j\leq 8}
	\left\|\pa_x^i(z\pa_z)^jF(t)\right\|_{\widetilde{Y}_{\mu,t}}
	+\sum_{i\leq8}\left\|e^{\eps_0(1+\mu)|\xi|}(\pa_x^i\pa_y\Delta_D^{-1}BC_{b,0})_\xi|_{y=0}(t)\right\|_{L^1_\xi\cap L^2_\xi}\\
		&\quad\leq C+(\mu_0-\mu-\gamma t)^{-\beta}\big(\|\omega_{re}\|_{Z(t)}+\|\omega_{re}\|_{Z(t)}^2\big).
	\end{align*}
\end{lemma}

\begin{proof}
	Using the definition of $F$ and Lemma \ref{lem: product estimate}, we have
	\begin{align*}
		\sum_{i+j\leq 8}
	\left\|\pa_x^i(z\pa_z)^jF(t)\right\|_{\widetilde{Y}_{\mu,t}}
	&\leq \sum_{i\leq 8}\left\|e^{\eps_0(1+\mu)|\xi|} \big(\pa_x^i\big(u_b^{(0)}+u_{vp}^{(0)}(t,0) \big)\big)_\xi\right\|_{L^1_\xi}
	\sum_{i+j\leq9}\left\|\pa_x^i(z\pa_z)^j\omega_b^{(0)}\right\|_{\widetilde{Y}_{\mu,t}}\\
	&+\sum_{i\leq 9}\left\|e^{\eps_0(1+\mu)|\xi|} \big(\pa_x^i\big(u_b^{(0)}+u_{vp}^{(0)}(t,0) \big)\big)_\xi\right\|_{L^1_\xi}
	\sum_{i+j\leq8}\left\|\pa_x^i(z\pa_z)^j\omega_b^{(0)}\right\|_{\widetilde{Y}_{\mu,t}} ,
	\end{align*}
	which along with Lemma \ref{lem: est of Ue(k)} and Lemma \ref{lem: est of omega c U c} implies for $\eps_0<C'/2$,
	\begin{align*}
		&\sum_{i\leq 9}\left\|e^{\eps_0(1+\mu)|\xi|} \big(\pa_x^i\big(u_b^{(0)}+u_{vp}^{(0)}(t,0) \big)\big)_\xi\right\|_{L^1_\xi}
		\leq C+\sum_{i\leq 9}\left\|e^{\eps_0(1+\mu)|\xi|} \big(\pa_x^iu_{re}\big)_\xi(t,0)\right\|_{L^1_\xi}\\
		&\leq C+\sum_{i\leq 9}\left\|\pa_x^i\omega_{re}\right\|_{\widetilde{Y}_{\mu,t}}
		\leq C+C(\mu_0-\mu-\gamma t)^{-\beta}\|\omega_{re}\|_{Z(t)},
	\end{align*}
	and
	\begin{align*}
		\sum_{i\leq 8}\left\|e^{\eps_0(1+\mu)|\xi|} \big(\pa_x^i\big(u_b^{(0)}+u_{vp}^{(0)}(t,0) \big)\big)_\xi\right\|_{L^1_\xi}
		\leq C+\sum_{i\leq 8}\left\|e^{\eps_0(1+\mu)|\xi|} \big(\pa_x^iu_{re}\big)_\xi(t,0)\right\|_{L^1_\xi}
		\leq C+C\|\omega_{re}\|_{Z(t)}.
	\end{align*}
	Therefore, we obtain
	\begin{align*}
		\sum_{i+j\leq 8}
	\left\|\pa_x^i(z\pa_z)^jF(t)\right\|_{\widetilde{Y}_{\mu,t}}
	\leq C\big(1+(\mu_0-\mu-\gamma t)^{-\beta}\|\omega_{re}\|_{Z(t)}\big)\big(1+\|\omega_{re}\|_{Z(t)}\big).
	\end{align*}
	
	We now focus on the second inequality. Applying Lemma \ref{lem: velocity formula}, we obtain
	\begin{align*}
		(\pa_y\Delta_D^{-1}BC_{b,0})_\xi|_{y=0}(t)
		=-\int_0^{+\infty}e^{-|\xi|y}(BC_{b,0})_\xi(t,y)dy.
	\end{align*}
	Armed with the definition of $BC_{b,0}$ and the transformation $z=\frac{y}{\nu^{1/2}}$, we have
	\begin{align*}
		\left|(\pa_x^i\pa_y\Delta_D^{-1}BC_{b,0})_\xi|_{y=0}(t)\right|
		\leq& \int_0^{+\infty}\left|\Big(\pa_x^i\big((u_{vp}^{(0)}+u_b^{(0)})(t,x,\nu^{1/2}z)(\pa_x\omega_b^{(0)})(t,x,z)\big)\Big)_\xi\right|dz\\
		&+\int_0^{+\infty}\left|\Big(\pa_x^i\big(\frac{v_{vp}^{(0)}+v_b^{(0)}}{y}(t,x,\nu^{1/2}z)(z\pa_z\omega_b^{(0)})(t,x,z)\big)\Big)_\xi\right|dz.
	\end{align*}
	Due to Lemma \ref{lem: product estimate}, we deduce that
	\begin{align*}
		&\sum_{i\leq8}\left\|e^{\eps_0(1+\mu)|\xi|}(\pa_x^i\pa_y\Delta_D^{-1}BC_{b,0})_\xi|_{y=0}(t)\right\|_{L^1_\xi\cap L^2_\xi}\\
		&\leq \sum_{i\leq 8}\left\|e^{\eps_0(1+\mu)|\xi|}\sup_{z>0}\left| \big(\pa_x^i(u_b^{(0)}+u_{vp}^{(0)})\big)_\xi(t,z)\right|\right\|_{L^1_\xi}
	\sum_{i+j\leq9}\left\|\pa_x^i(z\pa_z)^j\omega_b^{(0)}\right\|_{\widetilde{Y}_{\mu,t}}\\
	&\quad+\sum_{i\leq 9}\left\|e^{\eps_0(1+\mu)|\xi|}\sup_{z>0}\left| 
	 \big(\pa_x^i(u_b^{(0)}+u_{vp}^{(0)})\big)_\xi(t,z)\right|\right\|_{L^1_\xi}
	\sum_{i+j\leq8}\left\|\pa_x^i(z\pa_z)^j\omega_b^{(0)}\right\|_{\widetilde{Y}_{\mu,t}}.
	\end{align*}
	The remaining steps are identical to those for the first inequality in this lemma.
\end{proof}

\medskip

\section{Estimates for approximate solutions}\label{sec: Estimates for approximate solutions}
This section is devoted to proving Proposition \ref{prop: app estimates}. We start by presenting the following lemma, which provides estimates for $U_{vp}^{(k)}$ near the boundary.
\begin{lemma}\label{lem: est of Ue(k)}
	For $T$ small and $k=0,1$, there exist $C,C'>0$ such that
	\begin{align*}
		\sum_{\substack{i\leq 10 \\ l\leq 1}}\left\|\sup_{0\leq y\leq 5}e^{C'|\xi|}\left|\big((1,x)\pa_t^l (u_{vp}^{(k)},\frac{v_{vp}^{(k)}}{y})\big)_\xi(t,y)\right|\right\|_{L^1_\xi\cap L^2_\xi}
		+\sum_{i\leq 10}\left\|e^{C'|\xi|}\big((1,x)\pa_x^i u_0\big)_\xi|_{y=0}\right\|_{L^1_\xi\cap L^2_\xi} \leq C.
	\end{align*}
\end{lemma}

\begin{proof}
	Since $U_{a,vp}=BS_{\mathbb R^2_+}[\omega_{a,vp}]=U_{vp}^{(0)}+\nu^{1/2} U_{vp}^{(1)}+O(\nu)$, we utilize \eqref{BS law formulation 2} and take Fourier transformation to obtain
	\begin{align*}
		\big((1,x) u_{a,vp}\big)_\xi(t,y)
		=\frac{1}{2\pi}\int_{\chi_{vp}}(1,isgn\xi)e^{-2\pi ix'\xi}\big(e^{-2\pi|y-y'||\xi|}-e^{-2\pi|y+y'||\xi|} \big)\omega_{a,vp}(t,x',y')dx'dy',
	\end{align*}
	and
	\begin{align*}
		\big((1,x)u_0\big)_\xi(y)
		=\frac{(1,isgn\xi)}{\pi}e^{-2\pi(y+y_0)|\xi|}.
	\end{align*}
	The identities above obviously yield the desired result.
\end{proof}

\textbf{Proof of Proposition \ref{prop: app estimates}.}  The first estimate \eqref{est: app1} is derived from the definition of $\omega_{a,b}$ \eqref{def: app solution ultimate 2} and Proposition \ref{prop: WP of boundary layer}.
 
  The first inequality in \eqref{est: app2} follows from Proposition \ref{prop: properties of Lambda} and equations of $\Omega_2,\Omega_3$ given in \eqref{eq: eq of Omega2}, \eqref{eq: eq of Omega3}. For the second identity in \eqref{est: app2}, we use the fact $\int_{\mathbb R^2}\Omega_0 d\eta=1$ together with $ \Omega_2\in\mathcal Y_2, \Omega_3\in \mathcal Y_2\oplus\mathcal Y_3$ which imply $\int_{\mathbb R^2}\Omega_2 d\eta=\int_{\mathbb R^2}\Omega_3 d\eta=0$.
  
  To establish the first inequality in \eqref{est: app3}, we recall $U_{a,b}=(u_{a,b},v_{a,b})=BS_{\mathbb R^2_+}[\omega_{a,b}]$ and only prove the case $j=1$ for $u_{a,b}$. The remaining cases can be treated analogously. By Lemma \ref{lem: velocity formula}, we have
  \begin{align*}
  	(\pa_x^i\pa_y u_{a,b})_\xi(y)=-\xi^i(\omega_{a,b})_\xi(y)
  	&+\frac{|\xi|\xi^i}{2}\Big(\int_0^y e^{-|\xi|(y-z)}(1-e^{-2|\xi|z})(\omega_{a,b})_\xi(z)dz\\
  	&-\int_y^{+\infty}e^{-|\xi|(z-y)}(1+e^{-2|\xi|y})(\omega_{a,b})_\xi(z)dz\Big),
  \end{align*}
  which implies
  \begin{align*}
  	&\left\|\pa_x^i\pa_y u_{a,b}\right\|_{L^\infty(y\geq1/4)}
  	\leq \left\|\|(\pa_x^i\pa_y u_{a,b})_\xi(y)\|_{L^\infty_y(y\geq1/4)}\right\|_{L^1_\xi}\\
  	&\leq \left\|\|\xi^i(\omega_{a,b})_\xi(y)\|_{L^\infty_y(y\geq\f14)}\right\|_{L^1_\xi}
  	+\left\|\left\|\xi^{i+2}\frac{1-e^{-2|\xi|z}}{2|\xi|z}z(\omega_{a,b})_\xi(z)\right\|_{L^1_z}\right\|_{L^1_\xi}
  	+\left\|\left\|\xi^{i+1}(\omega_{a,b})_\xi(z)\right\|_{L^1_z(z\geq\f14)}\right\|_{L^1_\xi}\\
  	&\leq \left\|\|\xi^i(\omega_{a,b})_\xi(y)\|_{L^\infty_y(y\geq\f14)}\right\|_{L^1_\xi}
  	+5\left\|\left\|\langle\xi\rangle^{i+2}z(\omega_{a,b})_\xi(z)\right\|_{L^1_z}\right\|_{L^1_\xi}
  \end{align*}
  By Sobolev embedding and Proposition \ref{prop: WP of boundary layer}, we have
  \begin{align*}
  	\left\|\left\|\xi^i(\omega_{a,b})_\xi(y)\right\|_{L_y^\infty(y\geq\f14)}\right\|_{L^1_\xi}
  	&\leq C\sum_{k\leq6}\left\|\left\|\xi^i\pa_y^k(\omega_{a,b})_\xi(y)\right\|_{L_y^1(y\geq\f14)}\right\|_{L^1_\xi}\\
  	&\leq Ce^{-\frac{C_1}{\nu t}}\sum_{k\leq6}\left\|\left\|e^{\frac{\eps_0y^2}{\nu t}}\xi^i\pa_y^k(\omega_{a,b})_\xi(y)\right\|_{L_y^1(y\geq\f14)}\right\|_{L^1_\xi}
  	\leq Ce^{-\frac{C_1}{\nu t}}.
  \end{align*}
 By Proposition \ref{prop: WP of boundary layer}, we have
 \begin{align*}
 	&\left\|\left\|\langle\xi\rangle^{i+2}z(\omega_{a,b})_\xi(z)\right\|_{L^1_z}\right\|_{L^1_\xi}
 	\leq\left\|\left\|\langle\xi\rangle^{i+2}z e^{\frac{\eps_0z^2}{\nu t}}\chi_b(\nu^{-1/2}\omega_b^{(0)}+\omega_b^{(1)}+\nu^{1/2}\omega_b^{(2)})_\xi(t,\frac{z}{\nu^{1/2}})\right\|_{L^1_z}\right\|_{L^1_\xi}\\
 	&\leq \nu^{1/2}\left\|\left\|\langle\xi\rangle^{i+2}z' e^{\eps_0 z'^2}(\omega_b^{(0)}+\nu^{1/2}\omega_b^{(1)}+\nu\omega_b^{(2)})_\xi(t,z')\right\|_{L^1_{z'}}\right\|_{L^1_\xi}
 	\leq C\nu^{1/2}.
 \end{align*}
 Collecting these estimates together, we obtain
 \begin{align*}
 	\left\|\pa_x^i\pa_y u_{a,b}\right\|_{L^\infty(y\geq1/4)}
  	\leq C\nu^{1/2}.
 \end{align*}
 
 For the second inequality in \eqref{est: app3}, we use \eqref{BS law formulation 2} and the support property of $\chi_{vp}$ to obtain
\begin{align*}
	&U_{a,vp}(X,t)=\frac{1}{2\pi}\int_{\mathbb R^2}\big(\frac{(X-Y)^\perp}{|X-Y|^2}-\frac{(X-Y^\ast)^\perp}{|X-Y^\ast|^2}\big)\cdot\frac{\alpha}{\nu t}\chi_{vp}(Y)\mathcal W_a(\frac{Y-X(t)}{(\nu t)^{1/2}},t)dY\\
	&=\frac{1}{2\pi}\int_{\mathbb R^2}\frac{(X-Y)^\perp}{|X-Y|^2}\cdot\frac{\alpha}{\nu t}\mathcal W_a(\frac{Y-X(t)}{(\nu t)^{1/2}},t)dY\\
	&\qquad+\frac{1}{2\pi}\int_{\mathbb R^2}\frac{(X-Y)^\perp}{|X-Y|^2}\cdot\frac{\alpha}{\nu t}(\chi_{vp}(Y)-1)\mathcal W_a(\frac{Y-X(t)}{(\nu t)^{1/2}},t)dY\\
	&\qquad-\frac{1}{2\pi}\int_{\mathbb R^2}\frac{(X-Y^\ast)^\perp}{|X-Y^\ast|^2}\cdot\frac{\alpha}{\nu t}\chi_{vp}(Y)\mathcal W_a(\frac{Y-X(t)}{(\nu t)^{1/2}},t)dY
	:=I_1+I_2+I_3.
\end{align*}  
For $I_1$, Lemma \ref{lem: decay rate of velocity} implies
\begin{align*}
	\left\|(X-X(t))I_1\right\|_{L^\infty}
	\leq C\|\mathcal W_a\|_{L^{4/3}\cap L^4}\leq C,
\end{align*}
thus we obtain
\begin{align*}
	|I_1|\leq \frac{C}{|X-X(t)|}.
\end{align*}
For $I_2$, if $2<p<+\infty$, we use Hardy-Littlewood-Sobolev inequality to obtain
\begin{align*}
	\|I_2\|_{L^p}
	\leq \frac{C}{\nu t}\|(1-\chi_{vp})\mathcal W_a\|_{L^{\frac{2p}{p-2}}}
	\leq Ce^{-\frac{C_1}{\nu t}},
\end{align*}
and if $p=+\infty$, Sobolev embedding yields
\begin{align*}
	\|I_2\|_{L^\infty}
	\leq \frac{C}{\nu t}\|(1-\chi_{vp})\mathcal W_a\|_{L^{4/3}\cap L^4}
	\leq Ce^{-\frac{C_1}{\nu t}}.
\end{align*}
Therefore, we have $I_2=O_{L^p}(e^{-\frac{C_1}{\nu t}})$ for $2<p\leq+\infty$.

For $I_3$, since $Y\in\operatorname{supp}\chi_{vp}$, it holds that $|X-Y^\ast|\geq C|X-X(t)|$ for $X\in\mathbb R^2_+$. Therefore, we have 
\begin{align*}
	|I_3|\leq \frac{C}{|X-X(t)|} \int_{\mathbb R^2}\frac{\alpha}{\nu t}\mathcal W_a(\frac{Y-X(t)}{(\nu t)^{1/2}},t)dY
	\leq \frac{C}{|X-X(t)|}.
\end{align*}
Collecting these estimates together, we obtain the second inequality in \eqref{est: app3}.

For \eqref{est: app4}, we only prove the estimate for $u_a$, as the bound for $\frac{v_a}{y}$ follows by an identical argument. Recall that $u_a=u_{a,b}+u_{a,vp}$. For $u_{a,b}$, by Lemma \ref{lem: velocity formula} and Proposition \ref{prop: WP of boundary layer}, we obtain
\begin{align*}
	\sum_{i+j\leq8}\left\| e^{C'|\xi|}\sup_{0<y<5}\left|\big(\pa_x^i(y\pa_y)^j u_{a,b}\big)_\xi(t,y)\right|\right\|_{L^1_\xi}
	\leq \sum_{i+j\leq8}\left\| e^{C'|\xi|}\left\|\big(\pa_x^i(y\pa_y)^j \omega_{a,b}\big)_\xi(t,y)\right\|_{L^1_y}\right\|_{L^1_\xi}
	\leq C.
\end{align*}
For $u_{a,vp}$, by \eqref{BS law formulation 2}, it holds that
\begin{align*}
	u_{a,vp}(t,x,y)
	=\frac{1}{2\pi}\int_{\mathbb R^2_+}
	\big(\frac{y+y'}{(x-x')^2+(y+y')^2}
	-\frac{y-y'}{(x-x')^2+(y-y')^2}\big)
	\omega_{a,vp}(t,x',y')dx'dy'.
\end{align*}
Taking Fourier transformation w.r.t. $x$ variable, we have
\begin{align*}
	(u_{a,vp})_\xi(t,y)
	=\int_{\mathbb R^2_+}
	e^{2\pi ix'\xi}\big(e^{-2\pi(y+y')|\xi|}-e^{-2\pi(y-y')|\xi|}\big)
	\omega_{a,vp}(t,x',y')dx'dy'.
\end{align*}
Due to $\operatorname{supp}\omega_{a,vp}\subseteq\{y'\geq10\}$, it holds that
\begin{align*}
	&\sum_{i+j\leq8}\left\| e^{C'|\xi|}\sup_{0<y<5}\left|\big(\pa_x^i(y\pa_y)^j u_{a,vp}\big)_\xi(t,y)\right|\right\|_{L^1_\xi}\\
	&\leq C\sum_{i+j\leq8}\left\| e^{C'|\xi|}e^{-10\pi|\xi|}\xi^{i+j}\right\|_{L^1_\xi}
	\cdot\int_{\mathbb R^2_+}\left|\omega_{a,vp}(t,x',y')\right|dx'dy'
	\leq C.
\end{align*}

\ef

\end{document}